\documentclass[10pt,a4paper,openany,oneside]{article}
\usepackage{times}
\usepackage{authblk}
\usepackage[latin1]{inputenc}
\usepackage[english]{babel}
\usepackage{mathrsfs,bbm}
\usepackage{stmaryrd}
\usepackage{amssymb,amsmath,amsthm, mathtools}
\usepackage{ulem}
\usepackage{hyperref}
\usepackage[dvipsnames]{xcolor}
\usepackage{graphicx}
\usepackage[font=normalsize]{subcaption}
\usepackage{eurosym}
\usepackage{color}
\usepackage{enumerate}
\usepackage[shortlabels]{enumitem}
\usepackage{multirow}
\usepackage{stmaryrd}
\usepackage{booktabs}
\usepackage{float}
\usepackage{tikz}
\theoremstyle{plain}

\newtheorem{maintheorem}{Theorem}

\newtheorem{thm}{Theorem}[section]
\newtheorem{theorem}{Theorem}
\newtheorem{cor}[thm]{Corollary}
\newtheorem{lemma}[thm]{Lemma}
\newtheorem{prop}[thm]{Proposition}

\newtheorem{example}{Example}
\newtheorem{assumption}{Assumption}
\theoremstyle{definition}
\newtheorem{defn}{Definition}[section]
\theoremstyle{remark}
\newtheorem{rem}{Remark}
\newcommand{\C}{\mathbb C}
\newcommand{\R}{\mathbb R}
\newcommand{\N}{\mathbb N}

\def\eps{\varepsilon}

\def\ri{{\rm i}}
\def\bi{\begin{itemize}}
	\def\ei{\end{itemize}}

\newcommand{\Z}{\mathbb{Z}}

\newcommand{\cD}{{\mathcal D}}
\newcommand{\cE}{{\mathcal E}}

\newcommand{\cH}{{\mathcal H}}

\newcommand{\cL}{{\mathcal L}}

\newcommand{\cP}{{\mathcal P}}

\newcommand{\pa}{\partial}

\newcommand{\ga}{\gamma}

\newcommand{\bU}{{\mathbf U}}
\newcommand{\bV}{{\mathbf V}}

 \def\dd{\, {\rm d}}

\DeclareMathOperator{\Real}{Re}
\DeclareMathOperator{\Imag}{Im}

\newcommand{\dist}{\text{\rm dist}}
\newcommand{\supp}{\text{\rm supp}}

\DeclareMathOperator\perm{\boldsymbol{\epsilon}}

\newcommand{\bspm}{\left(\begin{smallmatrix}}\newcommand{\espm}{\end{smallmatrix}\right)}
\newcommand{\bpm}{\begin{pmatrix}}\newcommand{\epm}{\end{pmatrix}}

\def\blem{\begin{lemma}}\def\elem{\end{lemma}}
\def\bthm{\begin{theorem}}\def\ethm{\end{theorem}}
\def\bcor{\begin{corollary}}\def\ecor{\end{corollary}}
\def\beq{\begin{equation}}\def\eeq{\end{equation}}
\def\beqq{\begin{equation*}}\def\eeqq{\end{equation*}}
\def\bal{\begin{align}}\def\eal{\end{align}}
\def\bpf{\begin{proof}}\def\epf{\end{proof}}
\def\bex{\begin{example}}\def\eex{\end{example}}
\def\brem{\begin{rem}}\def\erem{\end{rem}}
\def\bass{\begin{assumption}}\def\eass{\end{assumption}}
\def\bprop{\begin{prop}}\def\eprop{\end{prop}}
\def\bdefi{\begin{definition}}\def\edefi{\end{definition}}

\makeatletter\newcommand\customitem[1]{\item[#1{\MakeLinkTarget[item]{}}]\def\@currentlabel{#1}}\makeatother
\begin{document}
	\date{\today}
	\title{Polychromatic Localized Waves with Complex Frequencies in Nonlinear Maxwell Equations with Material Dispersion}
	\author{ Tom\'a\v{s} Dohnal$^1$, Maximilian Hanisch$^1$, and Runan He$^2$ }
	\affil{$^1$ Institut f\"{u}r Mathematik,  Martin-Luther-Universit\"{a}t Halle-Wittenberg,\\ 06120 Halle (Saale), Germany\\
		$^2$ Instituto de Ciencias Matem\'aticas (ICMAT), Calle Nicolas Cabrera 13, 28049 Madrid, Spain\\ 
		\small{tomas.dohnal@mathematik.uni-halle.de, maximilian.hanisch@mathematik.uni-halle.de, runan.he@icmat.es}}
	\maketitle

	\abstract{We study the existence of polychromatic solutions of cubically nonlinear Maxwell equations in the whole space and with dispersive media, i.e., with a time delayed polarization. Due to the complex nature of the dielectric function, the frequencies are complex, resulting in a decay in time.  The geometry is that of a waveguide in $x$ with the propagation direction being $y$ and the solutions are localized in $x$ and have a transverse magnetic polarization. These are often referred to as breathers. They are given as a Fourier series in $y$ and $t$ with the leading frequency $\omega$ being an eigenvalue of a corresponding operator pencil on $\R$ (in the $x$ variable). Each term in the series corresponds to a different temporal decay rate or a different frequency. The series is constructed iteratively via a sequence of linear ordinary differential equations. Our general result provides the existence under some assumptions  on the spectrum and on estimates of the resolvent of the corresponding linear operator. We also produce an example of a waveguide given by the interface of two spatially homogeneous physically relevant media for which these assumptions are satisfied. For such an interface setting the constructed solutions correspond to nonlinear polychromatic surface plasmons.}
	
	\medskip
	\noindent
	\textbf{Mathematics Subject Classification:} 35Q61, 35L72, 78A60
	
	\noindent
	\textbf{Keywords:} nonlinear Maxwell equations, polychromatic solution, breather, dispersive material, retarded material law, surface plasmon, material interface
	\section{Introduction}
	
	Dispersive nonlinear Maxwell equations, in which the polarization involves nonlinear terms and temporal convolutions, are a standard model for the description of electromagnetic waves in media, see e.g. \cite{newell2003,agrawal2007}. In the recent years there have been  several results on the existence of monochromatic  \cite{DR2021,DR22-corr,DH2024} and polychromatic  \cite{OR2024,Ohrem_Reichel_2025,Ohrem2025} breather solutions  of this problem. Generic nonlinearities  acting on real fields generate higher harmonics and monochromatic  solutions are not possible unless these harmonics are neglected.  Only nonlinearities  which are averaged in time do not produce higher harmonics.  It is the aim of this paper to  construct  polychromatic solutions for polynomial nonlinearities in  a one dimensional geometry in the presence of losses. The resulting solutions have complex frequencies and decay exponentially in time.

	In detail, we study the existence of polychromatic solutions of Maxwell equations in the absence of free charges
	\begin{equation}\label{eqn:MaxwellEq}
		\partial_t\mathcal{D}=\nabla\times\mathcal{H},\quad	\boldsymbol{\mu}_0\partial_t\mathcal{H}=-\nabla\times\mathcal{E},\quad \nabla\cdot\mathcal{D}=\nabla\cdot\mathcal{H}=0,
	\end{equation}
	in a waveguide geometry. The dependence of the electric displacement field $\cD$ on the electric field $\cE$ is cubically nonlinear and generally nonlocal (in time), i.e., 
	\begin{equation}\label{eqn:Displacement}
		\mathcal{D} = \perm_0 \cE + \cP, \quad \cP=\cP^{(1)} +\cP^{(2)} +\cP^{(3)},
	\end{equation}
	where $\cE=\cE(x,y,t)$ and, for each $j=1,2,3$,
	\begin{equation}\label{E:P}
		\begin{aligned}
			\cP^{(1)}_j&=\cP^{(1)}_j(x,y,t)=\perm_0 \sum_{p=1}^3 \left(\chi^{(1)}_{j,p}(x)*_\R\mathcal{E}_p(x,y,\cdot)\right)(t), \\ 
			\cP^{(2)}_j&=\cP^{(2)}_j(x,y,t)=\perm_0 \sum_{p,q=1}^3
			\left(\chi^{(2)}_{j,p,q}(x)*_{\R^2}
			(\cE_p(x,y,\cdot)\cE_q(x,y,\cdot))\right)(t,t),\\
			\cP^{(3)}_j&=\cP^{(3)}_j(x,y,t)=\perm_0 \sum_{p,q,r=1}^3
			\left(\chi^{(3)}_{j,p,q,r}(x)*_{\R^3}
			(\cE_p(x,y,\cdot)\cE_q(x,y,\cdot)\cE_r(x,y,\cdot))\right)(t,t,t)
		\end{aligned}
	\end{equation}
	and where $\chi^{(1)}, \chi^{(2)}$, and $\chi^{(3)}$ are the linear, quadratic, and cubic electric susceptibility tensors, respectively, whose components are generally distribution-valued. The symbol $*_{\R^n}$ denotes the $n$-fold convolution (in time) over $\R$.
	As indicated, we allow the material (described by the susceptibilities) to depend on the spatial variable $x$.  Natural function spaces for solutions of \eqref{eqn:MaxwellEq} are based on $H(\mathrm{curl})$ and $H(\mathrm{div})$ spaces with respect to the spatial variables. We construct solutions which are smooth in $t$ and satisfy
	\beq\label{E:fn-sp-Maxw}
	\cE \in C^\infty((0,\infty), H_\mathrm{loc}(\mathrm{curl},\R^3)), \ \cH \in  C^\infty((0,\infty), H^1_\mathrm{loc}(\R^3)), \ \text{and} \ \nabla \cdot \cD(\cdot,t)=\nabla \cdot \cH(\cdot,t)=0 \ \forall t>0.
	\eeq
	In addition, our solutions are decaying in $t$, periodic in $y$, independent of $z$, and localized in $x$, i.e., they satisfy $L^2$-integrability in $x$ over the whole $\R$, see Theorem \ref{T:main}. In particular, we have
	$$\cE=\cE(x,y,t), \ \cH=\cH(x,y,t), \ \cD=\cD(x,y,t), \ \cP = \cP(x,y,t).$$
	Moreover, below we assume a plane wave type dependence on $y$ and $t$.
	
	Next, let us clarify the meaning of the convolution terms in the polarization $\cP$. If $\chi^{(1)}:\R^2\to\R^{3\times 3}$, $\chi^{(2)}:\R^3\to\R^{3\times 3\times 3}$, and $\chi^{(3)}:\R^4\to\R^{3\times 3\times 3\times 3}$  are functions, then we have
	\begin{equation}\label{eqn:D-for-fns}
		\begin{aligned}
			\cP^{(1)}_j(x,y,t)&= \sum_{p=1}^3\int_{\mathbb{R}}\chi^{(1)}_{j,p}(x, t-s)\mathcal{E}_p(x,y,s)\mathrm{d}s,\\
			\cP^{(2)}_j(x,y,t)&=\sum_{p,q=1}^3\int_{\mathbb{R}^2}\chi^{(2)}_{j,p,q}(x, t-s_1,t-s_2)\mathcal{E}_p(x,y,s_1)\mathcal{E}_q(x,y,s_2)\mathrm{d}(s_1,s_2),\\
			\cP^{(3)}_j(x,y,t)&=\sum_{p,q,r=1}^3\int_{\mathbb{R}^3}\chi^{(3)}_{j,p,q,r}\left(x, t-s_1, t-s_2,t-s_3\right) \mathcal{E}_p(x,y,s_1)\mathcal{E}_q(x,y,s_2)\mathcal{E}_r(x,y,s_3)\mathrm{d}(s_1,s_2,s_3).
		\end{aligned}
	\end{equation}
	Another typical example is the case of the instantaneous response when $\chi^{(n)}(x)=a^{(n)}(x)\delta$, where $a^{(1)}:\R\to \R^{3\times 3}$, $a^{(2)}:\R\to \R^{3\times 3\times 3}$, $a^{(3)}:\R\to \R^{3\times 3\times 3\times 3}$ are some functions and $\delta$ is the Dirac delta in $\R^n$. Then \eqref{E:P} becomes
	$$			\begin{aligned}
		\cP^{(1)}_j(x,y,t)&= \sum_{p=1}^3a^{(1)}_{j,p}(x)\mathcal{E}_p(x,y,t),\\
		\cP^{(2)}_j(x,y,t)&=\sum_{p,q=1}^3a^{(2)}_{j,p,q}(x)  \mathcal{E}_p(x,y,t)\mathcal{E}_q(x,y,t),\\
		\cP^{(3)}_j(x,y,t)&=\sum_{p,q,r=1}^3a^{(3)}_{j,p,q,r}(x)  \mathcal{E}_p(x,y,t)\mathcal{E}_q(x,y,t)\mathcal{E}_r(x,y,t),
	\end{aligned}
	$$
	see Appendix \ref{S:distrib} for more explanations on convolutions of distributions. In order to cover both of the above cases (and more), we allow $\chi^{(n)}(x), n=1,2,3,$ to be (tensor valued) distributions in time for each $x\in \R$.
	We choose $\chi^{(n)}(x)$ such that the system is causal and possesses a finite memory, i.e., there are $T^{(1)},T^{(2)},T^{(3)}>0$ such that $\supp (\chi^{(n)}(x))\subset [0,T^{(n)}]^n$ for all $n\in\{1,2,3\}, x\in \R$. The finiteness of the memory is a technical assumption needed for 
	the convergence of the integral in a Fourier-Laplace transform and the assumption $\supp (\chi^{(n)}(x))\subset [0,\infty)$ is for causality. 
	Moreover, we allow for the material to feature an interface of two (generally non-homogeneous) media. Altogether, this means that the susceptibility has the form  
	\begin{equation}
		\label{eqn:suscept}
		\chi^{(n)}(x)=\begin{cases}
			\chi^{(n)}_+(x),~&x>0,\\
			\chi^{(n)}_-(x),~&x<0,
		\end{cases}
	\end{equation}
	for each $n=1,2,3$, where $\chi^{(n)}_\pm(x)$ are tensor valued distributions of compact support for each $x\in\R_\pm$, i.e.,
	$$\chi^{(1)}_{\pm_{j,p}}(x) \in \cD_c'(\R),  \chi^{(2)}_{\pm_{j,p,q}}(x) \in \cD_c'(\R^2), \ \text{and} \ \chi^{(3)}_{\pm_{j,p,q,r}}(x) \in \cD_c'(\R^3) \ \text{for all} \ j,p,q,r \in \{1,2,3\}, x \in \R_\pm,$$
	where $\R_+:=[0,\infty)$ and $\R_-:=(-\infty,0]$ are the positive and negative half-lines and $\cD'_c(\R^n)$ denotes the space of distributions over $\R^n$ with compact support, see Appendix \ref{S:distrib} for the definition.

	In the presence of the interface we expect to construct polychromatic solutions of \eqref{eqn:MaxwellEq}, \eqref{eqn:Displacement}, \eqref{E:P}, \eqref{eqn:suscept} localized at $x=0$ and propagating in the $y-$direction in the form of a superposition of plane waves.
	In the linear setting $(\chi^{(2)}=\chi^{(3)}=0)$ monochromatic solutions $\left(\mathcal{E},\ \mathcal{H}\right)=\left(E,\ H\right)(x)e^{\mathrm{i}(ky-\omega t)}+\text{c.c.}$ with $k\in \R, \omega\in \C$, and $(E,H)$ localized at the interface	are known to exist \cite{Pitarke_2007,BDPW25} under a specific condition on $\chi^{(1)}_\pm$ and the frequency $\omega$. They are called surface plasmon polaritons \cite{maier2007,Pitarke_2007}. The abbreviation c.c. denotes the complex conjugate of the preceding terms. The frequency $\omega$ is an eigenvalue of a corresponding time harmonic linear Maxwell operator, more generally of an operator pencil. A linear surface plasmon is the main building block of our polychromatic solutions of the nonlinear problem.
	
	For any $\omega=\omega_R +\ri \omega_I\in \C$, our \textit{polychromatic ansatz} for equations \eqref{eqn:MaxwellEq} is given by
	\begin{equation}
		\label{eqn:polyAnsatz}
		\begin{split}
			(\mathcal{E},\mathcal{H})(x,y,t)=\sum\limits_{n\in \Z} e^{\ri n(ky-\omega_R t)} \sum\limits_{\nu\in \N} e^{\nu\omega_I t} (E^{n,\nu},H^{n,\nu})(x),
		\end{split}
	\end{equation}
	where $k\in \R$ is fixed and the condition $(E^{-n,\nu},H^{-n,\nu})=(\overline{E}^{n,\nu},\overline{H}^{n,\nu})$ is imposed for any $(n,\nu)\in \N_0\times\N$ in order to generate real fields.  We use the notation $\N:=\{1,2,\dots\}$, $\N_0:=\N\cup\{ 0\}$. Physically relevant solutions satisfy $\omega_I\leq 0$ corresponding to lossy or conservative fields.
	Roughly speaking, our main result, see Theorem \ref{T:main}, shows that solutions of the form \eqref{eqn:polyAnsatz} exist in the transverse magnetic setting if $\omega=\omega_R+\ri\omega_I$ is an eigenvalue as described above, the points $n\omega_R+\ri\nu\omega_I$ lie in the resolvent set for all $\nu \in \N\setminus\{1\}$ and all $n\in \Z, |n|\leq \nu$, and if the resolvent operator is bounded in the $H^1$ norm on each side of the interface with an operator norm that decays fast enough in $\nu$. We provide an example of two physically relevant homogeneous media (i.e. $\chi_\pm^{(j)}$ independent of $x,y,z$) for which these assumptions are satisfied.
	
	The motivation for studying polychromatic solutions is that the nonlinear system \eqref{eqn:MaxwellEq}, \eqref{eqn:Displacement}, \eqref{E:P} does not allow monochromatic solutions. Indeed, the monochromatic ansatz
	\beq\label{E:monochrom}
	\left(\mathcal{E},\ \mathcal{H}\right)(x,y,t)=\left(E,\ H\right)(x)e^{\mathrm{i}(ky-\omega t)}+\left(\overline{E},\ \overline{H}\right)(x)e^{-\mathrm{i}(ky-\overline{\omega} t)}
	\eeq
	produces in \eqref{eqn:Displacement}
	\beq\label{E:D-single-harm}
	\begin{aligned}
		\cP^{(1)}_j(x,y,t) = &~\perm_0\sum_{p=1}^3\hat{\chi}^{(1)}_{j,p}(x,\omega)E_p(x)e^{\mathrm{i}(ky-\omega_R t)} e^{\omega_I t}+ \text{c.c.},\\
		\cP^{(2)}_j(x,y,t) = &~ \perm_0\sum_{p,q=1}^3\left(\hat{\chi}^{(2)}_{j,p,q}(x,-\overline{\omega},\omega)\overline{E}_p(x)E_q(x)+\hat{\chi}^{(2)}(x,\omega,-\overline{\omega})E_p(x)\overline{E}_q(x)\right)e^{2\omega_I t} \\
		&+ \perm_0\sum_{p,q=1}^3\hat{\chi}^{(2)}_{j,p,q}(x,\omega,\omega)E_p(x)E_q(x)~e^{2\mathrm{i}(ky-\omega_R t)}e^{2\omega_I t} + \text{c.c.}
	\end{aligned}
	\eeq
	and an analogous expression for $\cP^{(3)}_j(x,y,t)$, where for all  $(\omega_1,\dots,\omega_n)\in \C^n$
	\beq\label{E:FT-L1}
	\hat{f}(\omega_1,\dots,\omega_n):=\int_{\R^n} f(t_1,\dots,t_n)e^{\ri(\omega_1 t_1+\dots+\omega_nt_n)}\dd (t_1,\dots,t_n)
	\eeq
	denotes the \textit{temporal Fourier-Laplace transform} of $f\in L^1(\R^n)$. For the case of $\chi^{(n)}(x)$ being distributions of compact support the form \eqref{E:D-single-harm} is explained in Appendix \ref{S:distrib}. The definition of the Fourier transform for such distributions appears in \eqref{E:FT-distrib-comp}.
	
	Clearly, the nonlinear terms $\cP^{(2)}$ and $\cP^{(3)}$ are proportional to $e^{2\omega_I t}$ and $e^{3\omega_I t}$, respectively, and contain oscillatory functions $e^{\ri n\omega_R t}$ with $n\in\{-2,0,2\}$ and $n\in\{-3,-1,1,3\}$, respectively. These terms are not compensated for by the linear terms in the equation. In end effect, this forces one to include all decay rates $\nu\omega_I$ with $\nu \in \N$ and for each rate $\nu$ all harmonics $n\omega_R$ with $n \in \Z, |n|\leq \nu$ in the solution ansatz. Note that our ansatz \eqref{eqn:polyAnsatz} contains all harmonics $n\in \Z$. In Theorem \ref{T:main} we prove the existence of solutions with $E^{n,\nu}=H^{n,\nu}=0$ if $|n|>\nu$.
	
	The presence of complex frequencies with an arbitrarily large negative imaginary part is the reason why we choose the susceptibility to have compact support (in time). Then the Fourier-Laplace transform at these frequency values is well defined using the Paley-Wiener theorem (see \eqref{E:FT-distrib-comp}).
	
	Clearly, the time delay (memory effect) in the polarization translates to frequency dependence of the $\hat{\chi}^{(j)}$. Such media are called dispersive; or media with a retarded material law.
	
	The case $\omega \in \R$ (i.e. $\omega_I=0$) has been used with nonlinearities defined via time averages over a temporal period  \cite{Stuart_93,sutherland2003handbook}. For this case there is no temporal decay of the solution, no higher harmonics are generated and monochromatic solutions are possible. Existence of such solutions was shown by bifurcation methods in \cite{DH2024} and by variational methods in a series of papers with various settings regarding the geometry, the coefficients, and the nonlinearity, see \cite{Benci2004}, \cite{BARTSCH20174304}, \cite{Hirsch_2019}, \cite{Mandel2022}   as well as the overview article \cite{BM2017}. In \cite{HH2016} the temporal evolution of nonlinear surface plasmons was studied for the pure cubic (Kerr) nonlinearity via asymptotic amplitude equations for the coupling of the right moving and left moving modes, both at the main carrier frequency.	Without the averaging in time, the nonlinearity generates higher harmonics and there are no monochromatic solution of Maxwell equations. In some works, see e.g. \cite{BS2001,Brahms:23,DR2021,DR22-corr,DH2024}, where the time-delayed or instantaneous Kerr nonlinearity (without averaging) was studied, higher harmonics were neglected. Then, families of monochromatic solutions \eqref{E:monochrom}  were constructed. In \cite{DR2021,DR22-corr,DH2024} solutions with both $\omega \in \C$ and $\omega\in \R$ were found in dispersive media. Real frequencies were shown to exist in $PT-$symmetric media, where loss and gain is balanced. In the current paper $PT-$symmetry is not assumed. Also, due to the possible presence of interfaces in the waveguide, the inversion symmetry of the material is generally broken and a Kerr nonlinearity would be unphysical which is the reason for including also quadratic terms. It is our aim to construct solutions of nonlinear Maxwell equations with complex frequencies without neglecting higher harmonics with the polarization $\cP$ as in \eqref{E:P}, i.e., without averaging in time.
	
	Polychromatic spatially localized solutions (also called breathers) in wave propagation have been constructed analytically before, but only in self-adjoint nonlinear problems with frequencies given by multiples of a \textit{real} $\omega\in\R$. For the spatial localization one needs a non-resonance condition, i.e., the multiples of $\omega$ must lie outside the spectrum, which itself is real. In \cite{Blank2011}, \cite{Pelinovsky2012}, and \cite{Hirsch_2019} scalar problems in one dimension were studied and a very specific choice of spatially periodic coefficients was used to
	achieve the highly non-generic situation, where all multiples of $\omega$ lie in spectral gaps. In a recent series of papers  \cite{OR2024,Ohrem_Reichel_2025,Ohrem2025} the authors studied the existence of time periodic traveling breathers of Maxwell equations with time retarded polarizations using a variational mountain pass argument. They considered a transverse electric polarized ansatz leading to a scalar equation. Also in these works the frequencies are real and the non-resonance condition is achieved either by a special choice of periodic coefficients or of the geometry and the propagation speed.

	In our non-self-adjoint setting it is expected that the non-resonance condition is satisfied in most cases. This is because the points $n\omega_R +\nu \ri \omega _I \in \C, \nu \in\{2,3,\dots\}, n\in\Z, |n|\leq \nu$ generically avoid the spectrum which typically occupies a small part of the complex plane. In that sense the non-self-adjoint nature of the problem is beneficial for constructing polychromatic solutions.
	
	We start in Section \ref{Sec_TM} by introducing a transverse magnetic reduction of the nonlinear Maxwell equations. Then in Section  \ref{Sec_lin_prob} we reformulate the system \eqref{eqn:MaxwellEq}, \eqref{eqn:Displacement}, \eqref{E:P} for the ansatz in \eqref{eqn:polyAnsatz} as a series of linear ODEs. In Section \ref{sec.spectrum} we introduce the notion of the spectrum of an operator pencil. Next, in Section \ref{S:main-result} we present our main results on the existence of polychromatic localized breathers (nonlinear surface plasmons). Theorem \ref{T:main} is in the general setting of the susceptibilities being causal compactly supported distributions and Theorem \ref{T:Lor-example} is for a specific example of an interface of two physically relevant homogeneous materials. Theorem \ref{T:main} is then proved in Section \ref{Sec_main_res_proof}.  
	The remaining sections deal with the physical example covered by Theorem \ref{T:Lor-example}. After introducing the linear and nonlinear susceptibility functions and discussing the spectrum of the corresponding operator in Section \ref{sec:spec} we prove in Section \ref{sec:proof-Thm-2} the necessary resolvent estimates, which are assumed in Theorem \ref{T:main}. In Section \ref{S:numerics} we present a numerical example for a specific material interface. We finish with a discussion of our results in Section \ref{Sec_diss}.
	Appendix \ref{S:distrib} reviews some standard facts on the distributions of compact support in relation to the Fourier transform and the convolution with smooth functions. Appendix \ref{app:Drude} concerns the Drude model for a dispersive material and describes the reason why it is not suitable for our finite memory setting. Finally, Appendix \ref{app:num} describes the numerical method used to compute the polychromatic solutions.

	\section{Problem Setting and Main Results}
	
	\subsection{Transverse Magnetic Reduction} \label{Sec_TM}

	Motivated by the fact that in the case of two homogeneous \textit{linear} media ($\chi^{(2,3)}=0$, $\chi^{(1)}_\pm$ independent of $x$) localized monochromatic solutions (surface plasmons)  exist only in the transverse magnetic polarization, see \cite{Pitarke_2007,BDPW25}, we will restrict ourselves to transverse magnetic (TM) fields, i.e.,  
	$$\cE=(\cE_1,\cE_2,0)^\top \ \text{ and } \ \cH=(0,0,\cH_3)^\top \ \text{ with } \ (\cE,\cH)=(\cE,\cH)(x,y,t).$$
	Inserting the TM ansatz into equation \eqref{eqn:MaxwellEq} and balancing terms, we obtain $\pa_t \cP_3=0$. For generic susceptibilities $\chi^{(j)}, j=1,2,3,$ this is possible only for a constant field $\cE$. As we study non-constant fields, we impose the following compatibility condition
	\beq\label{E:P3-zero}
	\cE_3 =0 \quad \Rightarrow \quad (\cP^{(1)}_3=0, \cP^{(2)}_3=0, \cP^{(3)}_3=0)
	\eeq
	on the tensors $\chi^{(j)}$. Clearly, the implication \eqref{E:P3-zero} is equivalent to
	\beq\label{E:chi-ass-TM}
	\chi^{(1)}_{3,1}=\chi^{(1)}_{3,2}=0, \ \chi^{(2)}_{3,p,q}=0  \quad \text{if} \quad  p,q\in \{1,2\}, \quad  \text{and} \   \chi^{(3)}_{3,p,q,r}=0  \quad  \text{if} \quad  p,q,r\in \{1,2\}.
	\eeq
	In Sec. 1.5 of \cite{boyd2008} one can see that a number of crystal types satisfies \eqref{E:chi-ass-TM}. Examples are orthorombic crystals of type $mmm$, tetragonal crystals of types $4/m$ and $4/mmm$ or cubic crystals of types $m3$ and $m3m$.
	
	Under condition \eqref{E:chi-ass-TM} we rewrite Maxwell equations as a system for the three non-zero components of $\cE$ and $\cH$. We define 
	$$\psi(x,y,t):=(\tfrac{1}{\boldsymbol{\mu}_0}\cE_1(x,y,t),\tfrac{1}{\boldsymbol{\mu}_0}\cE_2(x,y,t),\cH_3(x,y,t))^\top.$$
	The Maxwell system \eqref{eqn:MaxwellEq} reduces to 
	\beq\label{E:maxw-TM}
	\begin{pmatrix}
		0\ & 0 \ &-\pa_y \\
		0\ & 0 \ & \partial_{x} \\
		-\pa_y\ & \partial_x \ &0
	\end{pmatrix}\psi+\begin{pmatrix}
		\pa_t\cD_1\\ \pa_t\cD_2\\ \pa_t\psi_3
	\end{pmatrix}=0, \quad \pa_x\cD_1+\pa_y\cD_2=0.
	\eeq
	The divergence condition $\nabla\cdot \cH=0$ is satisfied trivially as $\cH_3$ is independent of $z$.

	\subsection{Reformulation as a Sequence of Linear ODEs} \label{Sec_lin_prob}
	
	Due to the TM reduction we have 
	$E^{n,\nu}_3=0$ and  $H^{n,\nu}_1=H^{n,\nu}_2=0$ in the ansatz \eqref{eqn:polyAnsatz}. We define $u^{n,\nu} = (u^{n,\nu}_1, u^{n,\nu}_2, u^{n,\nu}_3)^\top$, where
	\begin{equation*}
		\label{eqn:TMcondition}
		\begin{split}
			u^{n,\nu}_1(x):=\frac{1}{\boldsymbol{\mu}_0}E_1^{n,\nu}(x),\quad u^{n,\nu}_2(x):=\frac{1}{\boldsymbol{\mu}_0}E_2^{n,\nu}(x),\quad u^{n,\nu}_3(x):=H_3^{n,\nu}(x),
		\end{split}
	\end{equation*}
	as well as the vectors $u^{n,\nu}_E:= (u^{n,\nu}_1,u^{n,\nu}_2, 0)^\top$ and $u^{n,\nu}_H:= (0,0, u^{n,\nu}_3)^\top$. Further, we define for any complex number $\omega\in\C$ and $(n,\nu )\in \Z\times\N$ the mixed multiple 
	$$\omega^{(n,\nu )}:=n\Real(\omega) +\ri\nu\Imag (\omega).$$
	With that, we fix $\omega_0:=\omega_R+\ri\omega_I$ with $\omega_R,\omega_I\in\R$ and get $\omega_0^{(n,\nu)}= n \omega_R + \ri \nu \omega_I$. We write
	$$ \psi (x,y,t)= \sum_{n\in\Z} e^{\ri n(ky-\omega_Rt)} \varphi^n (x,t), \quad \text{where} \quad \varphi^n(x,t):=\sum\limits_{\nu\in\N} e^{\nu\omega_It}u^{n,\nu}(x).$$ Next we want to use \eqref{E:maxw-TM} to deduce equations for the functions $u^{n,\nu}, (n,\nu)\in\Z\times\N$. For that we need 
	to calculate the polarizations $\cP^{(1)}$, $\cP^{(2)}$ and $\cP^{(3)}$ in terms of 
	$$\vec{u}:=(u^{n,\nu})_{(n,\nu)\in\Z\times\N}.$$ 
	To this end let $p,q,r\in \{1,2\}$. Then 
	\begin{align}
		\cE_p(x,y,t_1)\cE_q(x,y,t_2)&=\boldsymbol{\mu}_0^2 \psi_p(x,y,t_1)\psi_q(x,y,t_2) \notag\\
		&=\boldsymbol{\mu}_0^2 \sum_{n\in\Z} e^{\ri nky} \sum_{m\in\Z} e^{-\ri\omega_R(mt_1 + (n-m)t_2)} \varphi_p^{m}(x,t_1)\varphi_q^{n-m}(x,t_2), \label{eqn:convolutionP2}
	\end{align}
	and similarly 
	\beq \label{eqn:convolutionP3}
	\begin{aligned}
		\cE_p(x,y,t_1)&\cE_q(x,y,t_2)\cE_r(x,y,t_3) = \\
		&\boldsymbol{\mu}_0^3 \sum_{n\in\Z} e^{\ri nky} \sum_{m,l\in\Z} e^{-\ri\omega_R(mt_1 +lt_2+ (n-m-l)t_3)} \varphi_p^{m}(x,t_1)\varphi_q^{l}(x,t_2)\varphi_r^{n-m-l}(x,t_3).
	\end{aligned}
	\eeq 
	The product of the $\varphi$'s, each a series itself, can be calculated again as a Cauchy product, namely
	\begin{equation}
		\label{eqn:innerConvolutionP2}
		\varphi_q^{l}(x,t_2)\varphi_r^{n-m-l}(x,t_3)=\sum\limits_{\nu=2}^\infty \sum\limits_{\lambda=1}^{\nu-1} e^{\omega_I (\lambda t_2 + (\nu - \lambda)t_3)}u_q^{l,\lambda}(x) u_r^{n-m-l,\nu-\lambda}(x),
	\end{equation}
	where the outer sum starts at $\nu=2$, because the inner one is empty for $\nu=1$. Similarly, we obtain
	\begin{equation}
		\label{eqn:innerConvolutionP3}
		\varphi_p^{m}(x,t_1)\varphi_q^{l}(x,t_2) \varphi_r^{n-m-l}(x,t_3)= \sum\limits_{\nu=3}^\infty \sum\limits_{\mu=1}^{\nu -2} \sum\limits_{\lambda=1}^{\nu -\mu -1} e^{\omega_I (\mu t_1 + \lambda t_2 + (\nu-\mu-\lambda )t_3)} u_p^{m,\mu}(x)u_q^{l,\lambda}(x)u_r^{n-m-l,\nu-\mu-\lambda}(x).
	\end{equation}
	
	Combining \eqref{eqn:convolutionP2} with \eqref{eqn:innerConvolutionP2} and \eqref{eqn:convolutionP3} with \eqref{eqn:innerConvolutionP3}, respectively, we arrive at the expressions 
	\begin{equation}
		\label{eqn:Esquared}
		\cE_p(x,y,t_1)\cE_q(x,y,t_2)= \boldsymbol{\mu}_0^2 \sum_{n\in\Z}\sum_{\nu=2}^\infty e^{\ri nky} \sum_{m\in\Z}\sum_{\mu =1}^{\nu-1} e^{-\ri  t_1 \omega_0^{(m,\mu)} -\ri t_2 \omega_0^{(n-m,\nu-\mu)}} u_p^{m,\mu}(x)u_q^{n-m,\nu-\mu}(x)
	\end{equation}
	and 
	\begin{equation}
		\label{eqn:Ecubed}
		\begin{aligned}
			&\cE_p(x,y,t_1)\cE_q(x,y,t_2)\cE_r(x,y,t_3)=\\
			&\boldsymbol{\mu}_0^3 \sum\limits_{n\in \Z}\sum\limits_{\nu =3}^\infty e^{\ri nky} \sum\limits_{l,m\in \Z} \sum\limits_{\mu=1}^{\nu -2} \sum\limits_{\lambda=1 }^{\nu -\mu -1} e^{-\ri  t_1\omega_0^{(m,\mu)} - \ri t_2\omega_0^{(l,\lambda)} - \ri t_3\omega_0^{(n-m-l,\nu-\mu-\lambda)}} u_p^{m,\mu}(x)u_q^{l,\lambda}(x)u_r^{n-m-l,\nu-\mu-\lambda}(x),
		\end{aligned}
	\end{equation}
	where the order of summation was changed. This is justified as our construction will generate an absolutely summable sequence $\vec{u}$.
	
	The components of the polarizations $\cP^{(1)}$, $\cP^{(2)}$, and $\cP^{(3)}$ are then constructed as the sum of convolutions of the susceptibilities with the terms given by \eqref{eqn:Esquared} and \eqref{eqn:Ecubed}. Similar to \eqref{E:D-single-harm}, such a convolution takes the form

	\begin{equation*}
		\label{eqn:convolutions}
		\begin{aligned}
			&\left(\chi^{(1)}_{j,p}(x)\ast_\R \cE_p(x,y,\cdot)\right)(t)= \boldsymbol{\mu}_0\sum_{n\in\Z}\sum_{\nu\in\N} e^{\ri nky - \ri t\omega_0^{(n,\nu)}} \hat{\chi}^{(1)}_{j,p}(x,\omega_0^{(n,\nu)})u_p^{n,\nu}(x),\\
			&\left( \chi^{(2)}_{j,p,q}(x)\ast_{\R^2}(\cE_p(x,y,\cdot)\cE_q(x,y,\cdot))\right)(t,t)=\\
			& \qquad\boldsymbol{\mu}_0^2 \sum_{n\in\Z}\sum_{\nu=2}^\infty e^{\ri nky -\ri t \omega_0^{(n,\nu)}}\sum_{m\in\Z}\sum_{\mu =1}^{\nu-1} \hat{\chi}^{(2)}_{j,p,q}(x,\omega_0^{(m,\mu)},\omega_0^{(n-m,\nu-\mu)}) u_p^{m,\mu}(x)u_q^{n-m,\nu-\mu}(x),\\
			&\left( \chi^{(3)}_{j,p,q,r}(x)\ast_{\R^3}(\cE_p(x,y,\cdot)\cE_q(x,y,\cdot)\cE_r(x,y,\cdot))\right)(t,t,t)=\\
			&\qquad \boldsymbol{\mu}_0^3 \sum\limits_{n\in \Z}\sum\limits_{\nu =3}^\infty e^{\ri nky-\ri t\omega_0^{(n,\nu)}} \sum\limits_{l,m\in \Z} \sum\limits_{\mu=1}^{\nu -2} \sum\limits_{\lambda=1 }^{\nu -\mu -1} \hat{\chi}^{(3)}_{j,p,q,r}(x,\omega_0^{(m,\mu)},\omega_0^{(l,\lambda)},\omega_0^{(n-m-l,\nu-\mu-\lambda)}) \\
			&\qquad\qquad\qquad\qquad\qquad \times u_p^{m,\mu}(x)u_q^{l,\lambda}(x)u_r^{n-m-l,\nu-\mu-\lambda}(x).		
		\end{aligned}	
	\end{equation*}
	
	With that, we finally arrive at a formulation of the $\cD$-field as 
	\begin{equation}
		\label{eqn:DField}
		\begin{aligned} 
			\cD_j(x,y,t)&= \sum_{n\in\Z}\sum_{\nu\in\N} e^{\ri nky -\ri t\omega_0^{(n,\nu)}} \perm_0\boldsymbol{\mu}_0\left( u^{n,\nu}_j + \sum_{p=1}^2 \hat{\chi}^{(1)}_{j,p}(x,\omega_0^{(n,\nu)})u^{n,\nu}_p(x)\right) \\
			&+ \sum_{n\in\Z}\sum_{\nu=2}^\infty e^{\ri nky - \ri t\omega_0^{(n,\nu)}} Q^{n,\nu}_j(x,\omega_0,\vec{u}) + \sum_{n\in\Z} \sum_{\nu=3}^\infty e^{\ri nky-\ri t\omega_0^{(n,\nu)}} C^{n,\nu}_j(x,\omega_0,\vec{u}) 	
		\end{aligned} 
	\end{equation}
	for $j=1,2$, where we define the quadratic coefficients for $n\in\Z$, $\nu\in\N\setminus\{1\}$ as 
	$$Q^{n,\nu}_j(x,\omega,\vec{u}):=\perm_0\boldsymbol{\mu}_0^2\sum_{m\in\Z}\sum_{\mu=1}^{\nu-1}\sum_{p,q=1}^2\hat{\chi}^{(2)}_{j,p,q}(x,\omega^{(m,\mu)},\omega^{(n-m,\nu -\mu)})u^{m,\mu}_p(x) u^{n-m,\nu-\mu}_q(x)$$
	and the cubic coefficients for $n\in\Z$ and $\nu\in\N\setminus\{1,2\}$
	$$\begin{aligned} 
		C^{n,\nu}_j(x,\omega,\vec{u}):=\perm_0\boldsymbol{\mu}_0^3\sum_{l,m\in\Z}\sum_{\mu=1}^{\nu-2}\sum_{\lambda=1}^{\nu-\mu-1}\sum_{p,q,r=1}^2 &\hat{\chi}^{(3)}_{j,p,q,r}(x,\omega^{(m,\mu)},\omega^{(l,\lambda)},\omega^{(n-m-l,\nu -\mu-\lambda)})\\
		&\times u^{m,\mu}_p(x) u^{l,\lambda}_q(x)u^{n-m-l,\nu-\mu-\lambda}_r(x).
	\end{aligned}$$
	Note that the sum in $p,q$ and in $p,q,r$ is only over $\{1,2\}$ instead of $\{1,2,3\}$ because $\cE_3 =0$ in the TM polarization.
	
	Substituting expression \eqref{eqn:DField} for the $\cD$-field in \eqref{E:maxw-TM}, we compare coefficients of the same frequency oscillations and temporal decay to arrive for every $(n,\nu)\in\Z\times\N$ at the equation 
	
	\begin{equation}
		\label{eqn:operatorFormulation}
		\cL_{nk}(\omega_0^{(n,\nu)})u^{n,\nu} (x)= h^{n,\nu}(x,\omega_0,\vec{u} ),~~x\in\R,
	\end{equation}
	where for any $\omega \in \C$ and $k \in \R$ the linear operators are given by
	\begin{equation}\label{E:op-def}
		\begin{aligned}
			\cL_{k}(\omega)&:=A_{k}-B(x,\omega),\\
			A_{k}&:=\begin{pmatrix}
				0\ & 0 \ &k \\
				0\ & 0 \ &\ri \partial_{x} \\
				k\ &\ri \partial_x \ &0
			\end{pmatrix}=\bspm -\ri & 0 & 0\\ 0 & -\ri & 0 \\ 0 & 0 & \ri\espm \nabla_{k}\times,\quad \nabla_{k}=\left(\partial_x,\ \mathrm{i}k,\ 0\right)^\top, \qquad \qquad \\
			B(x,\omega)&:=-\omega\begin{pmatrix}
				\boldsymbol{\mu}_0\perm_0(1+\hat{\chi}^{(1)}_{1,1}(x,\omega))\ &\boldsymbol{\mu}_0\perm_0\hat{\chi}^{(1)}_{1,2}(x,\omega) \ &0\\
				\boldsymbol{\mu}_0\perm_0\hat{\chi}^{(1)}_{2,1}(x,\omega)\ & \boldsymbol{\mu}_0 \perm_0(1+\hat{\chi}^{(1)}_{2,2}(x,\omega)) &0\\
				0\ &0\ & 1
			\end{pmatrix},
		\end{aligned}
	\end{equation}
	and the nonlinearity is, for all $n\in\Z$,
	\begin{equation}
		\label{eqn:nonlinearity}
		\begin{aligned}
			h^{n,\nu}_j(x,\omega ,\vec{u}):=& -\omega^{(n,\nu)}\left( Q^{n,\nu}_j(x,\omega,\vec{u}) + C^{n,\nu}_j(x,\omega,\vec{u})\right),\ \forall\nu\geq 3, j=1,2,\\
			h^{n,2}_j(x,\omega,\vec{u}):=&-\omega^{(n,2)} Q_j^{n,2}(x,\omega,\vec{u}),\ j=1,2,\\
			h^{n,1}(x,\omega,\vec{u}):=&~0,\\
			h^{n,\nu}_3(x,\omega,\vec{u}):=&~0,\  \forall \nu\in\N .
		\end{aligned}
	\end{equation}
	For brevity we suppress the $x$-dependence in the notation for $\cL_k$. The vanishing of the nonlinearity at $\nu =1$ reflects that the summation with respect to $\nu$ in the nonlinear terms of \eqref{eqn:DField} starts at $\nu=2$, while $h^{n,\nu}_3=0$ is due to the TM-polarization.
	We finally introduce for $n,m,l\in\Z$, $\nu,\mu,\lambda\in\N$, $j,p,q,r\in\{1,2\}$, $x \in \R$, and $\omega \in \C$ the notation 
	\begin{equation}
		\label{E:betaAndGam}
		\begin{aligned}
			\beta^{n,m,\nu,\mu}_{j,p,q}(x,\omega)&:=-\omega^{(n,\nu)}\perm_0\boldsymbol{\mu}_0^2\hat{\chi}^{(2)}_{j,p,q}(x,\omega^{(m,\mu)},\omega^{(n-m,\nu-\mu)}),\\
			\gamma^{n,m,l,\nu,\mu,\lambda}_{j,p,q,r}(x,\omega)&:=-\omega^{(n,\nu)}\perm_0\boldsymbol{\mu}_0^3\hat{\chi}^{(3)}_{j,p,q,r}(x,\omega^{(m,\mu)},\omega^{(l,\lambda)},\omega^{(n-m-l,\nu-\mu-\lambda)})
		\end{aligned}
	\end{equation}
	for the coefficients of the quadratic and cubic nonlinearity terms, respectively. With that, the nonlinearity for $\nu\geq 2$ presents itself as 
	\begin{equation} 
		\label{eqn:hClosedForm}
		\begin{aligned} 
			h^{n,\nu}_j(x,\omega,\vec{u})=&\sum_{m\in\Z}\sum\limits_{\mu=1}^{\nu-1}\sum_{p,q=1}^2 \beta^{n,m,\nu,\mu}_{j,p,q}(x,\omega)u_p^{m,\mu}(x)u_q^{n-m,\nu-\mu}(x) \\
			&\ + \sum_{l,m\in\Z}\sum_{\mu=1}^{\nu-2}\sum_{\lambda=1}^{\nu-\mu-1}\sum_{p,q,r=1}^2 \gamma^{n,m,l,\nu,\mu,\lambda}_{j,p,q,r}(x,\omega) u^{m,\mu}_p(x) u^{l,\lambda}_q(x)u^{n-m-l,\nu-\mu-\lambda}_r(x),
		\end{aligned}
	\end{equation}
	where for $\nu=2$ only the first sum is non-zero.
	
	If $\hat{\chi}^{(1)}_{j,p}(\cdot ,\omega) \in L^\infty(\R)$ for all $j,p\in\{1,2\}$, then the natural domain of the operator $\cL_{k}(\omega)$ for every $k\in\R$ is
	\beq\label{E:DL}
	D_{\cL}:=\left\{u\in L^2\left(\mathbb{R},\mathbb{C}^3\right):\quad \nabla_k\times\ u\in L^2\left(\mathbb{R}, \mathbb{C}^3\right)\right\} =\left\{u\in L^2\left(\mathbb{R},\mathbb{C}^3\right):\quad u_2,\ u_3\in H^1\left(\mathbb{R}, \mathbb{C}\right)\right\}
	\eeq	
	and the operator maps to $L^2(\R,\C^3)$. Hence, from a spectral analytic point of view we understand $\cL_{k}$ as
	$$\cL_{k}:D_{\cL}\subset L^2(\R,\C^3)\to L^2(\R,\C^3).$$
	Note that if $u^{n,\nu}\in D_\cL$ for each $(n,\nu)\in\Z\times\N$, then we recover $E^{n,\nu}=(u^{n,\nu}_1, u^{n,\nu}_2, 0)^\top \in L^2(\R)\times H^1(\R)\times \{0\}, H^{n,\nu}=(0,0,u^{n,\nu}_3)^\top \in \{0\}\times \{0\}\times H^1(\R)$ and, after checking convergence of the $(n,\nu)$-series, also the conditions on $\cE$ and $\cH$ in \eqref{E:fn-sp-Maxw} are satisfied. We will show in the proof of Theorem \ref{T:main} that, similarly to monochromatic solutions, the divergence condition in  \eqref{E:fn-sp-Maxw} is satisfied automatically by construction also in the polychromatic case. 
	
	For the nonlinear analysis the function space $D_\cL$ is too large as it is not closed under multiplication. Hence, our working space will be the intersection of $D_\cL$ with $\cH^1$, i.e., with the space of functions which are in $H^1$ (with respect to $x$) on $\R_+$ as well as on $\R_-$, see Sec. \ref{S:main-result}.
	
	Clearly, because in \eqref{eqn:hClosedForm} the expression $h^{n,\nu}$ depends only on terms $u^{m,\mu}$ with some $m\in\Z$ and $\mu <\nu$, \eqref{eqn:operatorFormulation} is a linear equation for $u^{n,\nu}$ for each $(n,\nu)\in\Z\times\N$. Hence, the vector $\vec{u}$, and thus also the 
	solution of the nonlinear Maxwell equations, is constructed  iteratively via a sequence of \textit{linear} problems, i.e., by solving first \eqref{eqn:operatorFormulation} for $\nu=1$ and all $n\in\Z$, then using the $u^{n,1}$ constructed this way to solve \eqref{eqn:operatorFormulation}
	for $\nu=2$ and all $n\in\Z$ and so on.
	
	We close this subsection by another crucial investigation, namely that for fixed $\nu\in\N$, equation \eqref{eqn:operatorFormulation} only needs to be solved for finitely many values of $n\in\N$ if we use a frequency $\omega_0\in\C$ that is an eigenvalue of the operator pencil $\cL_k$.
	The spectrum of this operator pencil is discussed in more detail in Section \ref{sec.spectrum} but we call $\omega_0\in\C$ an eigenvalue of $\cL_k$, if the equation $\cL_k(\omega_0)\varphi_0=0$ admits a non-trivial solution $\varphi_0\in D_{\cL}$ which we call an eigenfunction. 
	Using that $h^{n,1}\equiv 0$ by \eqref{eqn:nonlinearity}, for $\nu=1$ equation \eqref{eqn:operatorFormulation} becomes 
	$$\cL_k(\omega_0^{(n,1)})u^{n,1}=0$$
	for $n\in\Z$. Because $\omega_0^{(1,1)}=\omega_0$ we can thus choose $u^{1,1}=\varphi_0$, while for $n=-1$ we use the fact, that $\cL_{-nk}(\omega_0^{(-n,1)})=-\overline{\cL_{nk}(\omega_0^{(n,1)})}$ (see the proof of Theorem \ref{T:main}) such that $u^{-1,1}=\overline{\varphi_0}$ can be chosen there. In the case $|n|\neq 1$ we can simply choose $u^{n,1}=0$. Hence, only one equation needs to be solved for $\nu=1$.
	In the next step, consider $\nu=2$ and the nonlinearity 
	\begin{align*} 
		h^{n,2}_j(x,\omega_0,\vec{u})&= \sum_{m\in\Z}\sum\limits_{\mu=1}^{1}\sum_{p,q=1}^2 \beta^{n,m,2,\mu}_{j,p,q}(x,\omega_0)u_p^{m,\mu}(x)u_q^{n-m,2-\mu}(x)\\
		&=\sum_{m=-1}^1 \sum_{p,q=1}^2 \beta^{n,m,2,1}_{j,p,q}(x,\omega_0)u_p^{m,1}(x)u_q^{n-m,1}(x)
	\end{align*}
	because $u_p^{m,1}=0$ for $|m|>1$. However, since also $u_q^{n-m,1}=0$ for $|n-m|>1$, we arrive at 
	$$h^{n,2}_j(x,\omega_0,\vec{u})=0,~~\text{ for all }|n|>2,$$
	as then all summands in the above sum vanish. Consequently, we can choose $u^{n,2}=0$ as the solution to $\cL_{2k}(\omega_0^{(n,2)})u^{n,2}=0$ for $|n|>2$ and only need to solve 
	the linear ODE problem at the five values $n=-2,-1,0,1,2$ if $\nu=2$.
	
	Inductively one can similarly show (see again the proof of Theorem \ref{T:main}) that $h^{n,\nu}(x,\omega_0,\vec{u})=0$ whenever $|n|>\nu$ and thus $u^{n,\nu}=0$ can be chosen for each such pair $(n,\nu)\in\Z\times\N$.
	This motivates the definition of the sets 
	\begin{equation*}
		\label{eqn:defCone}
		\begin{aligned}
			\mathbf{I}&:=\{ (n,\nu) \in \Z\times \N: |n|\leq \nu\},\\
			\mathbf{S}&:=\{\omega_0^{(n,\nu)}\in\C:\ (n,\nu)\in \mathbf{I}\},
		\end{aligned}
	\end{equation*}
	where $\mathbf{I}$ is the set of all indices, in which the ODE \eqref{eqn:operatorFormulation} actually needs to be solved and $\mathbf{S}$ is the corresponding discrete cone in the complex plane, see Figure \ref{fig:cone}.
	
	\begin{figure}[h]
		\centering
		\begin{tikzpicture}
			\draw[->] (0,0) -- (10,0) node [below] {$\Real (\omega )$};
			\draw[->] (5,-3) -- (5,1) node [right] {$\Imag (\omega )$}; 
			\foreach \y in {1,...,5} {
				\foreach \x in {-\y ,...,\y}{
					\draw[fill=black] (5-0.65*\x,-0.5*\y) circle (0.1cm);
				}
			}
			\draw[fill=red] (5.65,-0.5) circle (0.1cm) node [above right] {$\omega_0$};
		\end{tikzpicture}
		\caption{Schematic of the set $\mathbf{S}$ of mixed integer multiples $\omega_0^{(n,\nu)}$ for $\nu\leq 5$ (with $\omega_R > 0$ and $\omega_I<0$).}
		\label{fig:cone}
	\end{figure}
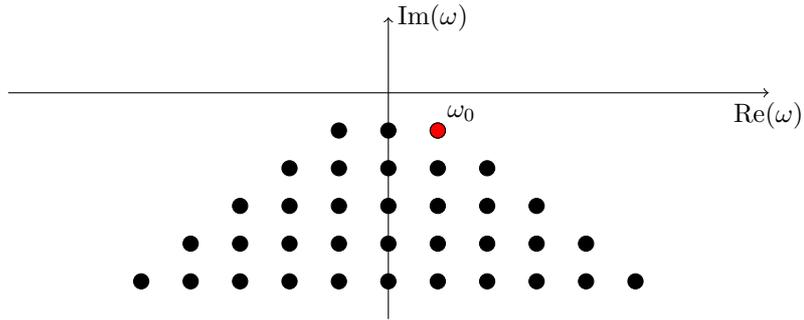
	
	Note, that in the case of $\omega_R=0$, the set $\mathbf{S}$ is located only on the imaginary line. This, however, poses no problem to the analysis because there are no higher harmonics generated by the nonlinearity if only the zeroth harmonic is present.
	The comparison of coefficients that leads to the system of equations \eqref{eqn:operatorFormulation} even gets simpler as one can effectively ignore the index $n\in\Z$ to obtain a system of equations in $\nu\in\N$ only.
	
	On the other hand, the case $\omega_I=0,$ i.e., $\omega_0\in \R$, would immediately make our construction invalid because it heavily relies on the fact that the nonlinear terms $h_{n,\nu}$ depend only $u_{m,\mu}$ with $\mu<\nu$. A real $\omega_0$ would lead to the requirement of all integer multiples of $\omega_0$ lying in spectral gaps, as discussed in the introduction. In 
	section \ref{S:main-result} we will present assumptions, under which the construction of polychromatic solutions is possible. In particular we require $\omega_I<0$.
	
	By a similar argument one can even show that this procedure leads to $h^{n,\nu}(x,\omega_0,\vec{u})=0$ whenever $n+\nu$ is odd, such that one can reduce the cone $\mathbf{S}$ to only those values $\omega^{(n,\nu)}$ where $n$ and $\nu$ have the same parity. While potentially interesting when checking resolvent conditions for specific examples, this fact does not simplify the analysis in the subsequent 
	discussion. For simplicity of the presentation, we will thus consider the full cone $\mathbf{S}$. In the numerics however, we employ this fact by not calculating components $u^{n,\nu}$ with $n+\nu$ odd, but simply setting them to be zero. 
	
	
	\subsection{Spectrum of the Operator Pencil}\label{sec.spectrum}
	Note that $\cL_{nk}=A_{nk}-B$ is an operator pencil parameterized by our spectral parameter $\omega$.
	We define the spectrum of $\cL_{nk}$ in line with the definitions from  \cite{DH2024} by using an additional parameter $\lambda$ and considering for each $\omega \in \C$ fixed the standard eigenvalue problem
	\begin{equation}\label{sec_Lpb}
		\cL_{nk}(\omega)u=\lambda u.
	\end{equation}
	Whether $\omega$ belongs to the (point/discrete/essential) spectrum of $\cL_{nk}$ is defined below by the condition that $\lambda=0$ belongs to the corresponding set for $\cL_{nk}(\omega)$ with $\omega$ fixed.
	
	We first define the \textbf{resolvent set}  of  the pencil $\cL_{nk}$  by
	$$\rho(\cL_{nk}):=\{\omega\in \C: \cL_{nk}(\omega):D_{\cL}\to L^2(\R, \C^3) \text{ is bijective with a bounded inverse}\},$$
	where $D_\cL$ was defined in \eqref{E:DL}. The \textbf{spectrum} of $\cL_{nk}$  is
	\begin{equation*}\label{E:Pspec}
		\sigma(\cL_{nk}):=\C\setminus \rho(\cL_{nk}).
	\end{equation*}
	Note that $\sigma(\cL_{nk})=\{\omega\in \C: 0\in \sigma(\cL_{nk}(\omega))\}$.
	
	The \textbf{point spectrum} is defined by
	$$\sigma_p(\cL_{nk}):=\{\omega\in \C: \exists u \in D_\cL\setminus \{0\}: \cL_{nk}(\omega)u=0 \}.$$
	Elements of $\sigma_p(\cL_{nk})$ are called eigenvalues of $\cL_{nk}$.
	
	For completeness, we mention that the \textbf{discrete spectrum} is defined via $
	\omega \in \sigma_d(\cL_{nk})  \ :\Leftrightarrow \ 0\in \sigma_d(\cL_{nk}(\omega)),$ 
	i.e., $\lambda=0$ is an isolated eigenvalue of finite algebraic multiplicity of the standard eigenvalue problem \eqref{sec_Lpb} (with $\omega\in \C$ fixed). Also, the \textbf{essential spectrum}  $\sigma_{\mathrm{ess}}(\cL_{nk})$ is defined by $\sigma_{\mathrm{ess}}(\cL_{nk})=\sigma(\cL_{nk})\setminus \sigma_d(\cL_{nk})$.

	Note that in \cite{BDPW25} the spectrum of the Maxwell operator pencil (in the second order formulation) for a related problem was defined using an auxiliary parameter $\lambda$ multiplying the operator $B$ rather than the identity $I$. This is due to the fact that in \cite{BDPW25} the divergence condition $\nabla \cdot D =0 $ was included in the definition of the domain of the pencil, which makes this choice of $\lambda$ necessary when defining the discrete spectrum.

	\subsection{Main Results}
	\label{S:main-result}
	As mentioned above, the domain $D_\cL$ is not suitable as the working space for the construction of solutions to the nonlinear problem since $u\in D_\cL$ does not imply $\cD \in D_\cL$. For example, nonlinearities involving the factor $u_1$ in $Q_2^{n,\nu}$ do not lie in $H^1(\R)$. A suitable choice of a function space is one, such that its elements are $H^1$ functions on each half line. In detail, we first define the piece-wise $H^1$ space
	$$
	\begin{aligned}
		\cH^1&:=\{f\in L^2(\R,\C^3): f|_{\R_+}\in H^1(\R_+,\C^3), f|_{\R_-}\in H^1(\R_-,\C^3)\},\\
		\|f\|_{\cH^1}&:=\sum_\pm \left(\sum_{j=1}^3 \|f_j\|^2_{H^1(\R_\pm)}\right)^{1/2},
	\end{aligned}
	$$
	and note that both $H^1(\R_\pm)$ are Banach algebras in the sense that
	\beq \label{eqn:algebra}
	\Vert f\cdot g \Vert_{H^1(\R_\pm)}\leq c_A\Vert f\Vert_{H^1(\R_\pm)}\Vert g\Vert_{H^1(\R_\pm)} \quad \forall f,g\in H^1(\R_\pm,\C^3),
	\eeq
	with some constant $c_A>0$, see Thm. 4.39 in \cite{AF2003}. For vector valued functions $f\in H^1(\R_\pm,\C^m)$ we define $$\|f\|_{H^1(\R_\pm)}:=\left(\sum_{j=1}^m\|f_j\|_{H^1(\R_\pm)}^2\right)^{1/2}.$$
	
	Our working space for the construction of the solution sequence $(u^{n,\nu})$ is 
	$$D_\cL\cap \cH^1.$$
	
	Our main general result holds under the following assumptions. Recall that $\cD_c'(\R^n)$ denotes the space of distributions with compact support (see Appendix \ref{S:distrib} for a detailed definition).
	\begin{enumerate}[label=(A\arabic*),ref=(A\arabic*)]
		\item \label{ass:A-cpctSupport}
		$\chi^{(1)}_{\pm,j,p}(x) \in\cD'_c(\R)$, $\chi^{(2)}_{\pm,j,p,q}(x)\in \cD'_c(\R^2)$, and $\chi^{(3)}_{\pm,j,p,q,r}(x)\in \cD'_c(\R^3)$ for all $x\in \R_\pm$, and $j,p,q,r\in\{1,2\}$. There are $T,T_N>0$ such that
		$$\supp \left(\chi^{(1)}_\pm(x) \right) \subset [0,T],\quad \supp \left( \chi^{(2)}_\pm(x)\right)\subset [0,T_N]^2 ,\quad \supp \left( \chi^{(3)}_\pm(x)\right)\subset [0,T_N]^3 ~~\forall x\in \R_\pm .$$
		
		\item \label{ass:A-TM}
		$\chi^{(1)}_{3,1}=\chi^{(1)}_{3,2}=0, \ \chi^{(2)}_{3,p,q}=0  \ \text{if} \ p,q\in \{1,2\}, \ \text{and} \   \chi^{(3)}_{3,p,q,r}=0  \ \text{if} \ p,q,r\in \{1,2\}.$
		\item \label{ass:A-cont}
		$\chi^{(1)}_{\pm,j,p}\in C^1(\R_\pm,\cD'_c(\R))$, $\chi^{(2)}_{\pm,j,p,q}\in C^1(\R_\pm ,\cD'_c(\R^2))$, and $\chi^{(3)}_{\pm,j,p,q,r}\in C^1(\R_\pm ,\cD'_c(\R^3))$ for all $j,p,q,r\in\{1,2\}.$ 
		
		\item \label{ass:A-periodic}
		$\chi^{(1)}_\pm$, $\chi^{(2)}_\pm$, and $\chi^{(3)}_\pm$ are eventually periodic, i.e., there exist $L>0$ and $p^{(1)}_\pm$, $p^{(2)}_\pm$, $p^{(3)}_\pm>0$ such that
		$$\chi^{(j)}_+(x+p^{(j)}_+)= \chi^{(j)}_+(x), \quad  \text{for all} \quad j\in \{1,2,3\}, x>L$$ and
		$$\chi^{(j)}_-(x-p^{(j)}_-)= \chi^{(j)}_-(x), \quad  \text{for all} \quad j\in \{1,2,3\}, x<-L.$$
		\item \label{ass:A-eval}
		$\omega_0=\omega_R+\ri\omega_I$ with $\omega_I<0$ is an eigenvalue of the operator pencil $\cL_k$, i.e., $\cL_k(\omega_0)\varphi_0$ $=0$,  where $\varphi_0\in \cH^1\cap D_{\cL}\setminus \{0\}$ is normalized to $\Vert \varphi_0\Vert_{\cH^1}=1$.
		\item \label{ass:A-resolvent}
		$\omega_0^{(n,\nu )}=n\omega_R+\ri\nu\omega_I \in \rho(\cL_{nk})$ for all  $(n,\nu) \in \mathbf{I}\setminus \{(1,1),(-1,1)\}$.
		\item \label{ass:A-resolvEst}
		For any $(n,\nu)\in\mathbf{I}\setminus \{(1,1),(-1,1)\}$ and $r=(r_1,r_2,r_3)^\top\in \cH^1$ the equation
		\beq\label{E:res-eq}
		\cL_{nk}(\omega_0^{(n,\nu)} )u=r
		\eeq
		admits a unique solution $u=(u_1,u_2,u_3)^\top\in D_\cL\cap \cH^1$ fulfilling
		$$\Vert (u_1,u_2) \Vert_{H^1(\R_\pm)}\leq f_1^-(\nu ) \Vert (r_1,r_2)\Vert_{H^1(\R_-)} + f_1^+(\nu ) \Vert (r_1,r_2)\Vert_{H^1(\R_+)} +f_3(\nu)  \Vert r_3\Vert_{H^1(\R_+ )+H^1(\R_- )}$$
		and
		$$\Vert u_3 \Vert_{H^1(\R_\pm)} \leq f_2^-(\nu ) \Vert (r_1,r_2)\Vert_{H^1(\R_-)} + f_2^+(\nu ) \Vert (r_1,r_2)\Vert_{H^1(\R_+)} +f_3(\nu)  \Vert r_3\Vert_{H^1(\R_+ )+H^1(\R_- )},$$
		where $f_1^\pm,f_2^\pm,f_3:\N\to \R_+$. If $\chi^{(2)}_+\neq 0$ or $\chi^{(3)}_+\neq 0$, then the functions $f_{1,2}^+$  satisfy the following conditions
		\begin{itemize}
			\item there exists $M>0$ such that 
			$$f_1^+(\nu )\leq M\nu^{-5}e^{-\sqrt{3}\nu |\omega_I|T_N} \ \text{ for all } \ \nu\in\N,$$
			\item $f_2^+$ grows at most exponentially, i.e., there is $\alpha\in\R$ such that 
			$$f_2^+(\nu )\leq Me^{\alpha\nu} \ \text{ for all } \ \nu\in\N.$$
		\end{itemize}
		If $\chi^{(2)}_-\neq 0$ or $\chi^{(3)}_-\neq 0$, then the functions $f_{1,2}^-$  satisfy the same conditions. No assumption on $f_3:\N\to \R_+$ is needed.
	\end{enumerate}
	\brem
	See Appendix \ref{S:distrib} for the definition of the support of a distribution used in assumption \ref{ass:A-cpctSupport}. Note that this assumption implies that the polarization model is causal and has a finite memory in time. 
	
	One observes in \ref{ass:A-resolvEst} drastically different assumptions on the estimates of the first two components and the third component of the solution of \eqref{E:res-eq}. This is due to the fact that in the nonlinear problem \eqref{eqn:operatorFormulation} only the equations for $\vec{u}_1$ and $\vec{u}_2$ are nonlinear and the nonlinearity does not depend on $\vec{u}_3$. Hence, in the first two components the decay rate of the resolvent must compensate for the potentially exponential growth of the Fourier transforms $\hat{\chi}^{(2)}$ and $\hat{\chi}^{(3)}$ in the variables $\nu$ and $T_N$, see Lemma \ref{lem:gamma_est}. The additional factor $\nu^{-5}$ in the resolvent estimate compensates for the number of terms in the sum with respect to $m,l,\lambda$, and $\mu$ in $h^{n,\nu}_j$ for $j=1,2$, see \eqref{eqn:estimate_r}.
	\erem

	The following is our general existence result for a polychromatic solution of the cubically nonlinear Maxwell equations \eqref{eqn:MaxwellEq}--\eqref{E:P}.
	\begin{maintheorem}\label{T:main}
		For a fixed $k\in\R$, assume \ref{ass:A-cpctSupport}--\ref{ass:A-resolvEst}, and let $\eps>0$ be arbitrary. Then there exists a sequence $\vec{u}:=(u^{n,\nu})_{(n,\nu)\in \Z\times\N} \subset D_\cL\cap \cH^1$ with
		$$u^{1,1} = \eps \varphi_0, \quad u^{n,\nu} =0 \ \text{ if }\ |n|>\nu, \quad u^{-n,\nu}=\overline{u^{n,\nu}} \quad \forall\ (n,\nu)\in \Z\times\N,$$
		satisfying equation \eqref{eqn:operatorFormulation} for each $(n,\nu)\in\Z\times\N$ and for some $c,\widetilde{M}>0$ the estimate
		\beq\label{E:un-nu-est}
		\|u^{n,\nu}\|_{\cH^1}\leq c\widetilde{M}^{\nu}\eps^\nu  \quad  \forall\ (n,\nu) \in \Z\times \N
		\eeq
		holds.
		
		If $\eps< e^{\omega_I T_\mathrm{max}}\widetilde{M}^{-1}$, then the series 
		\beq\label{E:Maxw-sol}
		\psi(x,y,t):=\sum_{\nu \in \N}\sum_{n\in \Z}u^{n,\nu}(x)e^{-\ri n(\omega_R t -ky)}e^{\nu \omega_I t}
		\eeq
		converges for every $t\in [-T_{\max},\infty )$ where $T_{\max}:=\max\{T,T_N\}$ and decays exponentially in time. Furthermore, the electromagnetic field $\cE:=(\boldsymbol{\mu}_0\psi_1,\boldsymbol{\mu}_0\psi_2,0)^\top,$ $\cH:=(0,0,\psi_3)^\top$
		is a solution of the Maxwell equations \eqref{eqn:MaxwellEq}--\eqref{E:P} on the time interval $[0,\infty )$, satisfying the regularity  
		$$\cE, \cH \in C^\infty([-T_\mathrm{max},\infty),C_\text{per}^\infty(\R,\cH^1))$$ with the period $[0,\tfrac{2\pi}{|k|}]$. Moreover, $\cE_2(\cdot,y,t),\cH_3(\cdot ,y,t) \in H^1(\R)$ for all $(y,t)\in \R\times [-T_\mathrm{max},\infty)$. 
	\end{maintheorem}
	\bigskip
	\begin{rem}
		The regularity $\cE, \cH \in C^\infty([-T_\mathrm{max},\infty),C_\text{per}^\infty(\R,\cH^1))$ means that the map $t\mapsto (\cE, \cH)(\cdot,\cdot,t)$ is in $C^\infty([-T_\mathrm{max},\infty))$, and for each $t\in [-T_\mathrm{max},\infty)$ the map $y\mapsto (\cE,\cH)(\cdot,y,t) \in \cH^1$ is in $C^\infty(\R)$ and periodic with period $[0,\tfrac{2\pi}{|k|}]$. Due to the additional $H^1(\R)$ regularity of $\cE_2(\cdot,y,t)$ and $\cH_3(\cdot ,y,t)$ we see that the constructed Maxwell solution satisfies the natural spatial regularity in \eqref{E:fn-sp-Maxw}.
	\end{rem}
	\begin{rem}\label{R:chihat-well-defined}
		Note that assumption \ref{ass:A-cpctSupport} implies that the Fourier transforms $\hat{\chi}^{(1)}_\pm(\cdot, \omega)$, $\hat{\chi}^{(2)}_\pm(\cdot ,\omega_1,\omega_2)$, $\hat{\chi}^{(3)}_\pm(\cdot,\omega_1,\omega_2,\omega_3)$ are defined (as functions) for all $\omega,\omega_1,\omega_2,\omega_3\in \C$; in particular for all $\omega,\omega_1,\omega_2,\omega_3\in \mathbf{S}$. This follows by the theorem of Paley-Wiener, see Appendix \ref{S:distrib}. Hence, the evaluation $\hat{\chi}_\pm^{(1)}(\cdot,\omega)$ with $\omega \in \C$ makes sense.
	\end{rem}
	\begin{rem}\label{R:chihat-cont}
		Theorem \ref{T:main} holds also if $\chi_\pm^{(j)}\in W^{1,\infty}(\R_\pm,\cD_c'(\R^j))$ for all $j=1,2,3$, i.e., under weaker regularity. However, in order for the model to only feature one interface (at $x=0$), we assume $\chi_\pm^{(j)}\in C(\R_\pm,\cD'_c(\R^j))$ and the differentiability of these mappings (with respect to $x$), which directly implies $\hat{\chi}^{(j)}_\pm(\cdot,\omega )\in C^1(\R_\pm ,\C)$ for all $\omega\in\C^j$, see Appendix \ref{S:distrib}. 
		The boundedness of the mappings $\hat{\chi}^{(j)}(\cdot,\omega )$ and $\partial_x \hat{\chi}^{(j)}(\cdot,\omega )$ for $x\to\pm\infty$ is then a direct consequence of assumption \ref{ass:A-periodic}. Recall that  $\R_\pm=\pm [0,\infty)$; hence we conclude that $\hat{\chi}^{(j)}_\pm(\cdot,\omega )\in C^1_b(\R_\pm ,\C)$.
	\end{rem}
	
	\begin{rem} The discussion in the introduction about the advantage of non-self-adjoint problems relates to the  non-resonance condition \ref{ass:A-resolvent}. As the spectrum of $\cL_{nk}$ occupies only a small subset of $\C$, this assumption is typically easily satisfied. 
	\end{rem}
	\begin{rem}
		Theorem \ref{T:main} can be directly generalized to finitely many interfaces, i.e., at $x=x_j, j=1,\dots, m$ with $m\in \N$ and $x_j <x_{j+1}$ for all $j=1,\dots m-1$. The space $\cH^1$ then needs to be replaced by $\{u\in L^2(\R,\C^3):u|_{(x_j,x_{j+1})}\in H^1((x_j,x_{j+1}),\C^3), \ \forall\ j\in \{1,\dots,m-1\}\}$ and the assumptions need to be adapted accordingly.
	\end{rem}
	
	\medskip
	
	Next we apply Theorem \ref{T:main} to a concrete case of physical materials. In detail, we study the interface of two homogeneous layers, at $x<0$ and $x>0$. One of the layers is a nonlinear dispersive material described by the Lorentz model with a finite memory and the other layer is a linear non-dispersive material. The finiteness of the memory of the Lorentz model is achieved by a simple truncation of the support of the time dependent Lorentz susceptibility. The following theorem shows that under certain conditions on the spectrum of the interface problem with the untruncated Lorentz model and one condition on the eigenvalue $\omega_0$ of the interface with the truncated Lorentz model the assumptions \ref{ass:A-cpctSupport}--\ref{ass:A-resolvEst}  are satisfied for large enough truncation parameters $T$.
	\begin{maintheorem}\label{T:Lor-example}
		Fix  $k\in \R$ and consider the interface problem with $\chi^{(1)}$ given by 
		$$
		\chi^{(1)}(x,t):=\begin{cases}
			\frac{c_L}{c_*} e^{-\gamma t}\sin\left(t\sqrt{\omega_*^2-\gamma^2}\right)\mathbbm{1}_{[0,T]}(t)~I_{3\times 3}, & x<0,\\
			\alpha \delta~I_{3\times 3}, & x>0,
		\end{cases}
		$$
		where $\delta \in \cD'(\R)$ is the Dirac distribution and $c_*=\sqrt{\omega_*^2-\gamma^2}$ with $c_L,\alpha>0, \omega_*>\gamma>0$. Also let
		$$
		\chi^{(2)}(x,t_1,t_2):=\begin{cases} \chi^{(2)}_{T_N}(t_1,t_2), & x<0,\\
			0, & x>0,
		\end{cases} 
		\qquad 
		\chi^{(3)}(x,t_1,t_2,t_3):=\begin{cases} \chi^{(3)}_{T_N}(t_1,t_2,t_3), & x<0,\\
			0, & x>0
		\end{cases}
		$$
		with $\chi^{(2,3)}_{T_N}$ defined in \eqref{E:chi23_trunc} and $0<T_N< \frac{1}{2\sqrt{3}}T$. If assumptions \ref{ass:B-signOfOmega}--\ref{ass:B-rationality}  (see Section \ref{sec:proof-Thm-2}) are satisfied, then a polychromatic solution of \eqref{eqn:MaxwellEq}--\eqref{E:P} of the form given by Theorem \ref{T:main} exists.	
	\end{maintheorem}

	\section{Proof of Theorem \ref{T:main}} \label{Sec_main_res_proof}

	We first prove an estimate for the coefficients $\beta^{n,m,\nu,\mu}_{j,p,q}$ and $\gamma^{n,m,l,\nu,\mu,\lambda}_{j,p,q,r}$ of the nonlinearity defined in \eqref{E:betaAndGam} that is based on the Paley-Wiener Theorem, see for example Theorem 7.23 in \cite{rudinFA}.
	As the norm $\Vert\cdot\Vert_{W^{1,\infty}(\R_\pm)}$ in this lemma we choose $\Vert g\Vert_{W^{1,\infty}(\R_\pm)}:= \Vert g\Vert_{L^\infty(\R_\pm)} + \Vert g'\Vert_{L^\infty(\R_\pm)}$.
	\begin{lemma}\label{lem:gamma_est}
		Let $\omega_0=\omega_R+\ri\omega_I\in\C$ and assume \ref{ass:A-cpctSupport}, \ref{ass:A-cont}, and \ref{ass:A-periodic}. Then there are constants $c_\beta^\pm>0$ and $c_\gamma^\pm>0$ such that
		\beq \label{eqn:estimate-Gamma}
		\begin{aligned} 
			\Vert \beta^{n,m,\nu,\mu}_{j,p,q}(\cdot,\omega_0)\Vert_{W^{1,\infty}\left(\R_\pm\right)} &\leq c_\beta^\pm \nu e^{\sqrt{2}T_N|\omega_I|\nu},\\
			\Vert \gamma^{n,m,l,\nu,\mu ,\lambda}_{j,p,q,r}(\cdot ,\omega_0)\Vert_{W^{1,\infty }\left(\R_{\pm}\right)}&\leq c_\gamma^\pm\nu e^{\sqrt{3}T_N|\omega_I|\nu},
		\end{aligned}
		\eeq
		for all $j,p,q,r\in\{1,2\}$ and $(n,\nu ),(m,\mu ),(l,\lambda )\in\mathbf{I}$ with $\mu<\nu$ and $\mu+\lambda <\nu$.
	\end{lemma}
	
	\begin{proof}
		We prove the claim only for $\gamma^{n,m,l,\nu,\mu,\lambda}_{j,p,q,r}$ as the estimate for $\beta^{n,m,\nu,\mu}_{j,p,q}$ follows exactly the same way. For any $x\in\R_{\pm}$, the distributions $\chi^{(3)}_{\pm}(x)$ have support in $[0,T_N]^3$ by assumption \ref{ass:A-cpctSupport}. Hence, the Paley-Wiener theorem \cite{rudinFA} implies the existence of constants $c^\pm>0$ such that, for any $\vec{\omega} = (\omega_1,\omega_2,\omega_3)\in\C^3$ and $x\in\R_\pm$
		\begin{equation*} \label{eqn:paley-wiener}
			\left| \hat{\chi}^{(3)}_{\pm ,j,p,q,r}(x,\vec{\omega} )\right| \leq c^\pm e^{\sqrt{3}T_N |\Imag (\vec{\omega})|} \leq c^\pm e^{\sqrt{3}T_N (|\Imag (\omega_1)| +|\Imag (\omega_2)| + |\Imag (\omega_3)|)}
		\end{equation*}
		because $\mathrm{diam}\left([0,T_N]^3\right)=\sqrt{3}T_N$. The fact that the constants $c^\pm$ can be chosen independent of $x$ follows from assumption \ref{ass:A-periodic} and the continuity of $\chi^{(3)}$ from \ref{ass:A-cont}.
		With that, we estimate for any $x\in\R_\pm$,
		\begin{align*}
			|\gamma^{n,m,l,\nu,\mu,\lambda}_{j,p,q,r}(x,\omega_0)| &= \left|-\omega_0^{(n,\nu)}\perm_0\boldsymbol{\mu}_0^3 \hat{\chi}^{(3)}_{\pm ,j,p,q,r}\left(x,\omega_0^{(l,\lambda)},\omega_0^{(m,\mu)},\omega_0^{(n-m-l,\nu-\mu-\lambda)}\right)\right|\\
			&\leq c^\pm\perm_0\boldsymbol{\mu}_0^3(|n|+\nu )|\omega_0|e^{\sqrt{3}T_N|\omega_I|(\lambda + \mu +\nu -\mu -\lambda)}\\
			&\leq c_\gamma^\pm \nu e^{\sqrt{3}T_N|\omega_I|\nu},
		\end{align*}
		where we used that $\nu -\mu -\lambda >0$ and $|n|\leq\nu$ by assumption, and set $c_\gamma^\pm:=2|\omega_0|c^\pm\perm_0\boldsymbol{\mu}_0^3$.
		
		A completely analogous estimate follows for $|\pa_x\gamma^{n,m,l,\nu,\mu,\lambda}_{j,p,q,r}(x,\omega_0)|$. This is ensured by the differentiability in  \ref{ass:A-cont}. In particular, it implies that $\pa_x \chi^{(3)}_{\pm ,j,p,q,r}(x)$ $ \in \cD_c'(\R^3)$ with the same support as $\chi^{(3)}_{\pm ,j,p,q,r}(x)$.
	\end{proof}
	
	\begin{rem}
		Note that this lemma is the main reason for assumption \ref{ass:A-periodic} and there is no connection between the eventual spatial periodicity in $x$ of the material and the $y-$periodicity of the solution we construct.
	\end{rem}
	
	\medskip
	\begin{proof}[Proof of Theorem \ref{T:main}]
		~\\
		
		First, we recall again (as explained in Remark \ref{R:chihat-well-defined}) that $\beta^{n,m,\nu,\mu}_{j,p,q}(x,\omega_0)$, $\gamma^{n,m,l,\nu,\mu,\lambda}_{j,p,q,r}(x,\omega_0)$, and $\hat{\chi}^{(1)}_{j,p}(x,\omega_0^{(n,\nu)})$ are well defined functions of $x$ for all $(n,\nu)\in\mathbf{I}$ due to the compact support assumption \ref{ass:A-cpctSupport}.
		
		We start the construction of the sequence $\vec{u}$ with $\nu=1$. Clearly, the formulation \eqref{eqn:nonlinearity} of the nonlinearity $h^{n,\nu}$, in particular
		$h^{n,1}=0$ for all $n\in \Z$,  implies
		\beq\label{lin_eq_n=1}\cL_{nk}(\omega_0^{(n,1)})u^{n,1}=0 \quad \forall\ n \in \Z.\eeq
		For $|n|\neq 1$ we may choose $u^{n,1}=0$ as a trivial solution of \eqref{lin_eq_n=1}.
		Considering $n=1$, due to \ref{ass:A-eval}, it follows that \eqref{lin_eq_n=1} has a non-trivial solution
		$u^{1,1}=\eps \varphi_0\in D_{\cL}$, where $\eps >0$ is arbitrary.
		For any real valued distribution $\Lambda \in \cD_c'(\R^n)$ we have the symmetry \beq\label{conjugate}\hat{\Lambda}(-\bar{\omega})=\overline{\hat{\Lambda}(\omega)},\quad \forall \ \omega \in\C^n,\eeq
		as a simple calculation using the definition \eqref{E:FT-distrib-comp} shows. 
		Hence,we get $$\cL_{-nk}(\omega_0^{(-n,1)})=-\overline{\cL_{nk}(\omega_0^{(n,1)})}.$$ So for $n=-1$ equation \eqref{lin_eq_n=1} has a non-trivial solution $u^{-1,1}$ given by
		$u^{-1,1} = \overline{u^{1,1}}=\eps\overline{\varphi_0}.$

		Next, we apply an induction argument to construct the components $u^{n,\nu}$ with $\nu>1$. Let $\nu \in \N$ and assume that
		\bi
		\item[(i)] $u^{m,\mu}=0$ for all $|m|>\mu$, $\mu \leq \nu$;
		\item[(ii)] $u^{m,\mu}\in D_\cL\cap \cH^1$ solves \eqref{eqn:operatorFormulation} for each $m \in \Z$ and $\mu \leq \nu$;
		\item[(iii)] $u^{-m,\mu} = \overline{u^{m,\mu}}$ for all $|m|\leq\mu, \mu \leq \nu$;
		\item[(iv)] $\|u_E^{m,\mu}\|_{\cH^1}\leq c_u^{\mu-1}\eps^\mu$ for all $|m|\leq \mu, \mu \leq \nu$, where $c_u>0$ is the constant
		\beq \label{eqn:cuDef} c_u:=\max_{\pm} 2^{5/2}c_A\left( c_\beta^\pm \frac M2 + \sqrt{(c_\beta^\pm)^2 M^2 +c_\gamma^\pm M}\right) \eeq 
		with the constants $c_A ,M$ and $c_\beta^\pm ,c_\gamma^\pm$ given by  \eqref{eqn:algebra}, \ref{ass:A-resolvEst} and \eqref{eqn:estimate-Gamma}, respectively.
		\ei
		We aim to show that there are $u^{n,\nu+1}$ with the same properties as in (i)--(iv) for all $n \in \Z$.
		The estimate of $\Vert u_H^{m,\mu}\Vert_{\cH^1}$ will then be shown directly without induction.
		
		Firstly, we prove that $u^{n,\nu+1}=0$ can be chosen for all $|n|>\nu+1$. From \eqref{eqn:hClosedForm} and the induction hypothesis (i) we get for any $j=1,2,3$
		\begin{align*}
			h^{n,\nu+1}_j(x,\omega,\vec{u})&=\sum_{m=-\nu}^\nu\sum\limits_{\mu=1}^{\nu}\sum_{p,q=1}^2 \beta^{n,m,\nu+1,\mu}_{j,p,q}(x,\omega)u_p^{m,\mu}(x)u_q^{n-m,\nu+1-\mu}(x) \\
			&\quad + \sum_{l=-\nu}^\nu\sum_{m=-\nu}^\nu\sum_{\mu=1}^{\nu-1}\sum_{\lambda=1}^{\nu-\mu}\sum_{p,q,r=1}^2 \gamma^{n,m,l,\nu+1,\mu,\lambda}_{j,p,q,r}(x,\omega) u^{m,\mu}_p(x) u^{l,\lambda}_q(x)u^{n-m-l,\nu+1-\mu-\lambda}_r(x)\\
			&=\sum\limits_{\mu=1}^{\nu}\sum_{m=-\mu}^\mu\sum_{p,q=1}^2 \beta^{n,m,\nu+1,\mu}_{j,p,q}(x,\omega)u_p^{m,\mu}(x)u_q^{n-m,\nu+1-\mu}(x) \\
			&\quad + \sum_{\mu=1}^{\nu-1}\sum_{\lambda=1}^{\nu-\mu}\sum_{l=-\lambda}^\lambda\sum_{m=-\mu}^\mu\sum_{p,q,r=1}^2 \gamma^{n,m,l,\nu+1,\mu,\lambda}_{j,p,q,r}(x,\omega) u^{m,\mu}_p(x) u^{l,\lambda}_q(x)u^{n-m-l,\nu+1-\mu-\lambda}_r(x),
		\end{align*}
		meaning, in particular, that $h$ includes only finite sums.
		If $n>\nu +1$, then also
		\begin{equation*}
			n-m> \nu+1 -\mu \quad \text{ and } \quad n-m-l > \nu +1 - \mu - \lambda
		\end{equation*}
		in the above two sums respectively. The induction hypothesis implies $u^{n-m,\nu-\mu}_q,$ $u^{n-m-l,\nu +1-\mu -\lambda}_r\equiv 0$ for any such $n$ and all $m,l,\mu,\lambda$ appearing in the sums. Thus $h^{n,\nu +1}\equiv 0$ for all $n>\nu +1$, and a similar argument shows $h^{n,\nu +1}\equiv 0$ if $n<-\nu -1$.
		We may thus choose  $u^{n,\nu+1}=0$ as the trivial solution to 
		$$\cL_{nk}(\omega_0^{(n,\nu +1)})u^{n,\nu +1}=0$$ 
		for all $|n|>\nu+1$, irrespective of whether the point $\omega_0^{(n,\nu+1)}$ is in the resolvent set of $\cL_{nk}$ or not.
		
		Next, we consider $0\leq n \leq \nu +1$.
		Using the algebra property \eqref{eqn:algebra} of $H^1$, the fact that $\beta^{n,m,\nu,\mu}_{j,p,q}(\cdot,\omega_0)$, $\gamma^{n,m,l,\nu,\mu,\lambda}_{j,p,q,r}(\cdot ,\omega_0)$ are in $W^{1,\infty}(\R_\pm )$ on each half line $\R_\pm$ by Lemma \ref{lem:gamma_est}, and that $h^{n,\nu+1}$ depends only on $u^{j,\mu}$ with $\mu \leq \nu$, $j\in \Z$, where $u^{j,\mu}\in \cH^1$ by assumption (ii), we have
		$h^{n,\nu+1}\in \cH^1$ for all $n\in \Z$. Due to \ref{ass:A-resolvent}, there exists a unique solution $u^{n,\nu +1}\in D_\cL\cap \cH^1$  of
		$$\cL_{nk}(\omega_0^{(n,\nu+1)})u^{n,\nu +1}=h^{n,\nu +1},$$
		which shows (ii) with $\nu$ replaced by $\nu +1$ for $n\in\N_0$.
		
		Next we verify the symmetry condition (iii) and show (ii) (with $\nu$ replaced by $\nu+1$) for $n <0$. We define $u^{-n,\nu+1}:=\overline{u^{n,\nu+1}}$ and show that $u^{-n,\nu+1}$ solves \eqref{eqn:operatorFormulation}.
		First, we deduce from \eqref{conjugate} and the fact $\omega_0^{(-j,l)}=-\overline{\omega_0^{(j,l)}}$ that
		$$\beta^{-n,-m,\nu,\mu}_{j,p,q}(x,\omega_0)=-\omega^{(-n,\nu)}\perm_0\boldsymbol{\mu}_0^2\hat{\chi}^{(2)}_{j,p,q}(x,\omega^{(-m,\mu)}_0,\omega_0^{(-n+m,\nu-\mu)})= -\overline{\beta^{n,m,\nu,\mu}_{j,p,q}}(x,\omega_0)\quad \forall n,m,\nu,\mu,j,p,q,$$
		and similarly $$\gamma^{-n,-m,-l,\nu,\mu,\lambda}_{j,p,q,r}(x,\omega_0)=-\overline{\gamma^{n,m,l,\nu,\mu,\lambda}_{j,p,q,r}}(x,\omega_0),\quad \forall n,m,l,\nu,\mu,\lambda,j,p,q,r.$$
		Hence, by replacing the summation indices $l$ and $m$ in $h^{n,\nu+1}$ by $-l$ and $-m$, we have
		\begin{align*}
			h^{-n,\nu +1}(x,\omega_0,\vec{u})&=\sum\limits_{\mu=1}^{\nu}\sum_{m=-\mu}^\mu\sum_{p,q=1}^2 \beta^{-n,-m,\nu+1,\mu}_{j,p,q}(x,\omega_0)u_p^{-m,\mu}(x)u_q^{-n+m,\nu+1-\mu}(x) \\
			&\quad + \sum_{\mu=1}^{\nu-1}\sum_{\lambda=1}^{\nu-\mu}\sum_{l=-\lambda}^\lambda\sum_{m=-\mu}^\mu\sum_{p,q,r=1}^2 \gamma^{-n,-m,-l,\nu+1,\mu,\lambda}_{j,p,q,r}(x,\omega_0) u^{-m,\mu}_p(x) u^{-l,\lambda}_q(x)u^{-n+m+l,\nu+1-\mu-\lambda}_r(x)\\
			&=-\overline{h^{n,\nu +1}(x,\omega_0,\vec{u})},
		\end{align*}
		where the induction hypothesis (iii) has been used. Due to $$\cL_{-nk}(\omega_0^{(-n,\nu)})=-\overline{\cL_{nk}(\omega_0^{(n,\nu)})}$$ for all $(n,\nu)\in\mathbf{I}$, which is again a consequence of \eqref{conjugate} applied to $\hat{\chi}^{(1)}$, we conclude that $u^{-n,\nu+1}=\overline{u^{n,\nu+1}}$ is a solution to \eqref{eqn:operatorFormulation}. Hence, 
		(ii) and (iii) are shown for all $|m|\leq\nu+1$.
		
		It remains to show the $\cH^1$-estimate (iv) with $\nu$ replaced by $\nu +1$. We will employ \ref{ass:A-resolvEst} with $r:=h^{n,\nu+1}(\cdot ,\omega_0, \vec{u})$. Note that then $r_3=0$ and for $j=1,2$ we get
		\beq \label{eqn:form-of-r}
		\begin{aligned} 
			r_j=& \sum\limits_{\mu=1}^{\nu}\sum_{m=-\mu}^\mu\sum_{p,q=1}^2 \beta^{n,m,\nu+1,\mu}_{j,p,q}(\cdot,\omega_0)u_p^{m,\mu}u_q^{n-m,\nu+1-\mu} \\
			&\quad + \sum_{\mu=1}^{\nu-1}\sum_{\lambda=1}^{\nu-\mu}\sum_{l=-\lambda}^\lambda\sum_{m=-\mu}^\mu\sum_{p,q,r=1}^2 \gamma^{n,m,l,\nu+1,\mu,\lambda}_{j,p,q,r}(\cdot,\omega_0) u^{m,\mu}_p u^{l,\lambda}_q u^{n-m-l,\nu+1-\mu-\lambda}_r,
		\end{aligned}
		\eeq
		where the summands  with  $|n-m|>\nu-\mu$ or $|n-m-l|>\nu-\mu-\lambda$ vanish.
		To estimate $\Vert r_j\Vert_{H^1(\R_\pm)}$, we first note that 
		$$\sum_{p,q=1}^2\Vert u^{m,\mu}_p \Vert_{H^1(\R_\pm)} \Vert u^{l,\lambda}_q\Vert_{H^1(\R_\pm)}\leq \sqrt{2}\Vert u_E^{l,\lambda}\Vert_{H^1(\R_\pm ,\C^3)} \sum_{p=1}^2 \Vert u^{m,\mu}_p\Vert_{H^1(\R_\pm)} \leq 2\Vert u_E^{m,\mu}\Vert_{H^1(\R_\pm,\C^3)} \Vert u_E^{l,\lambda}\Vert_{H^1(\R_\pm,\C^3)},$$
		for any $\mu,\lambda\in\N$, $l,m\in\Z$, with an analogous result and the constant $2^{3/2}$ if the sum is taken over the product of three functions. 
		Using additionally the fact that each sum in $m$ and $l$ in \eqref{eqn:form-of-r} has less than $2\nu +1$ terms, as well as the induction hypothesis (iv), the algebra property \eqref{eqn:algebra} and the bound \eqref{eqn:estimate-Gamma} from Lemma \ref{lem:gamma_est}, we further estimate for $j=1,2$
		\begin{align}
			\Vert r_j \Vert_{H^1(\R_-)}&\leq c_A \sum\limits_{\mu=1 }^{\nu} \sum\limits_{m=-\mu }^{\mu} \Vert \beta^{n,m,\nu +1,\mu}(\cdot ,\omega_0)\Vert_{W^{1,\infty}(\R_-)} 2\Vert u_E^{m,\mu}\Vert_{H^1(\R_-,\C^3)}\Vert u_E^{n-m,\nu +1-\mu}\Vert_{H^1(\R_-,\C^3)}\notag\\
			&\quad + c_A^2\sum_{\mu=1}^{\nu -1}\sum_{\lambda=1}^{\nu-\mu}\sum_{l=-\lambda}^\lambda\sum_{m=-\mu}^\mu \Vert \gamma^{n,m,l,\nu+1,\mu,\lambda}(\cdot ,\omega_0)\Vert_{W^{1,\infty}(\R_-)} 2^{3/2} \Vert u_E^{m,\mu}\Vert_{H^1(\R_-,\C^3)}\\
			& \hspace{4cm} \times \Vert u_E^{l,\lambda}\Vert_{H^1(\R_-,\C^3)}\Vert u_E^{n-m-l,\nu +1 -\mu-\lambda}\Vert_{H^1(\R_-,\C^3)}\notag\\
			&\leq 2 c_A c_\beta^-(\nu +1)e^{\sqrt{2}T_N|\omega_I|(\nu +1)} \sum\limits_{\mu=1 }^{\nu} \sum\limits_{m=-\mu }^{\mu} c_u^{\mu-1} \varepsilon^{\mu}c_u^{\nu+1-\mu-1}\varepsilon^{\nu+1-\mu}\notag\\
			&\quad + 2^{3/2} c_A^2 c_\gamma^- (\nu+1)e^{\sqrt{3}T_N|\omega_I|(\nu+1)}\sum_{\mu=1}^{\nu -1}\sum_{\lambda=1}^{\nu-\mu}\sum_{l=-\lambda}^\lambda\sum_{m=-\mu}^\mu c_u^{\mu-1}\varepsilon^{\mu}c_u^{\lambda-1}\varepsilon^{\lambda}c_u^{\nu +1-\mu-\lambda-1}\varepsilon^{\nu+1 -\mu-\lambda}\notag\\
			&\leq 4 c_A c_\beta^- (\nu+1)^3 e^{\sqrt{2}T_N|\omega_I|(\nu +1)} c_u^{\nu-1}\varepsilon^{\nu+1} + 2^{7/2} c_A^2 c_\gamma^- (\nu+1)^5 e^{\sqrt{3}T_N|\omega_I|(\nu+1)} c_u^{\nu-2}\varepsilon^{\nu+1}\notag\\
			&\leq  (\nu+1)^5 e^{\sqrt{3}T_N|\omega_I|(\nu +1)} c_u^{\nu-1}\varepsilon^{\nu+1} \left(4 c_A c_\beta^- + 2^{7/2} \frac{c_A^2 c_\gamma^-}{c_u} \right) \label{eqn:estimate_r}
		\end{align}
		The analogous result on $\R_+$ follows in the same way, hence
		\[\Vert (r_1, r_2) \Vert_{H^1(\R_\pm,\C^2)}\leq \sqrt{2}(\nu+1)^5 e^{\sqrt{3}T_N|\omega_I|(\nu +1)} c_u^{\nu-1}\varepsilon^{\nu+1} \left(4 c_A c_\beta^\pm + 2^{7/2} \frac{c_A^2 c_\gamma^\pm}{c_u} \right).\]
		From \ref{ass:A-resolvEst} it now follows that the unique solution $u^{n,\nu+1}$ of $\cL_{nk}(\omega_0^{(n,\nu+1 )})u^{n,\nu+1}=r$ fulfills
		\beq \label{eqn:simpleEst}
		\Vert u_E^{n,\nu+1} \Vert_{H^1\left(\R_\pm,\C^3\right)}\leq \sqrt{2}M c_u^{\nu-1}\varepsilon^{\nu+1}\left( \left(4 c_A c_\beta^- + 2^{7/2} \frac{c_A^2 c_\gamma^-}{c_u} \right) +\left(4 c_A c_\beta^+ + 2^{7/2} \frac{c_A^2 c_\gamma^+}{c_u} \right)\right) .
		\eeq
		The choice of $c_u$ in \eqref{eqn:cuDef} guarantees that the term in parentheses in \eqref{eqn:simpleEst} can be bounded by $c_u/(\sqrt{2}M)$, because $c_u$ is the maximum of all solutions to the two equations
		$$\left(4 c_A c_\beta^\pm + 2^{7/2} \frac{c_A^2 c_\gamma^\pm}{c_u} \right) =\frac{c_u}{2^{3/2}M},$$
		which are equivalent to quadratic equations in $c_u$. We thus arrive at $$\Vert u_E^{n,\nu+1} \Vert_{H^1\left(\R_\pm,\C^3\right)}\leq c_u^{\nu}\varepsilon^{\nu+1},$$
		proving (iv) for $\nu+1$ instead of $\nu$ and finishing the induction.
		
		Estimates of the remaining part of $u^{n,\nu}$, namely of $\Vert u_3^{n,\nu}\Vert_{H^1(\R_\pm)}$, and thus also of $\|u_H^{n,\nu}\|_{\cH^1}$, can now simply be obtained by \ref{ass:A-resolvEst}. In more detail, denote again for any fixed $(n,\nu)\in\mathbf{I}\setminus \{ (1,1),(-1,1)\}$ the right hand side by
		$r:=h^{n,\nu}(\cdot ,\omega_0,\vec{u})$, then there exists a unique solution $u^{n,\nu}$  to
		$$\cL_{nk}(\omega_0^{(n,\nu)})u^{n,\nu}(x)=r(x)$$
		with $r_3=0$. The estimates for $r_1$ and $r_2$ in \eqref{eqn:estimate_r} (with $\nu +1$ replaced by $\nu$), the property (iv) and assumption \ref{ass:A-resolvEst} imply
		\begin{equation*} \label{eqn:estimate-u3-minus}
			\begin{aligned} 
				\Vert u_3^{n,\nu} \Vert_{H^1\left(\R_\pm\right)}&\leq \sqrt{2}M\nu^5e^{(\alpha +\sqrt{3}T_N|\omega_I|)\nu}c_u^{\nu-2}\varepsilon^\nu \left( \left(4 c_A c_\beta^- + 2^{7/2} \frac{c_A^2 c_\gamma^-}{c_u} \right) +\left(4 c_A c_\beta^+ + 2^{7/2} \frac{c_A^2 c_\gamma^+}{c_u} \right)\right)\\
				&\leq \nu^5e^{(\alpha +\sqrt{3}T_N|\omega_I|)\nu}c_u^{\nu-1}\varepsilon^\nu .
			\end{aligned}  
		\end{equation*}
		Thus
		\beq \label{eqn:estimate_u3_and_uplus}
		\begin{aligned}
			\Vert u^{n,\nu}\Vert_{H^1\left(\R_\pm,\C^3\right)} &\leq c_u^{\nu-1}\eps^{\nu}\left( 1+\nu^{10}e^{2(\alpha +\sqrt{3}T_N|\omega_I|)\nu}\right)^{1/2} \\
			&\leq c_u^{-1} \eps^{\nu} \widetilde{M}^\nu,
		\end{aligned}
		\eeq
		where $\widetilde{M}= 10\sqrt{2}c_u\max\left\{ 1,e^{\alpha +\sqrt{3}|\omega_I|T_N}\right\}$ can be obtained by a simple calculation using the fact,
		that $\nu^5 \leq 10^\nu$, $\forall \nu\in\N$. This shows \eqref{E:un-nu-est} with $c=c_u^{-1}$.

		Let us now study the solution of the Maxwell equations \eqref{eqn:MaxwellEq} generated by $\psi$ in \eqref{E:Maxw-sol}. Formally, (without checking the convergence of the series in $\psi$) the fields $\cE:= \boldsymbol{\mu}_0(\psi_1,\psi_2,0)^\top, \cH:=(0,0,\psi_3)^\top$ satisfy Maxwell equations by construction of the vector $\vec{u}$. This is clear for the first two equations in \eqref{eqn:MaxwellEq}. The divergence condition $\nabla \cdot \cH=0$ is trivial since $\cH_{1,2}\equiv 0$ and $\cH$ is independent of $z$. For the divergence condition on $\cD$ note that \eqref{eqn:DField} can be rewritten as
		$$
		\begin{aligned}
			&\cD (x,y,t) = -\sum\limits_{n\in \Z} \sum\limits_{\nu\in \N} e^{\ri n(ky-\omega_R t) +\nu\omega_It}  \frac{1}{\omega_0^{(n,\nu)}}B(x,\omega_0^{(n,\nu)})u_E^{n,\nu}(x)  \\
			&- \sum\limits_{n\in \Z} \sum\limits_{\nu\in \N}e^{\ri n(ky-\omega_R) t +\nu\omega_It} \frac{1}{\omega_0^{(n,\nu)}}h^{n,\nu}(x,\omega_0,\vec{u})
		\end{aligned}
		$$
		with $B$ defined in \eqref{E:op-def}.
		
		Because of the exponential dependence on $y$, the $y$-derivative acts on each $n-$summand via multiplication by $\ri nk$. Hence, it is sufficient to show for each $n\in \Z$ and $\nu \in \N$ 
		$$ \nabla_{nk}\cdot (B(\cdot, \omega_0^{(n,\nu)}) u_E^{n,\nu}+h^{n,\nu})=0.$$
		This follows from \eqref{eqn:operatorFormulation} because $B u_E^{n,\nu}+h^{n,\nu}=A_{nk}u^{n,\nu}$ and
		$$\nabla_{nk}\cdot A_{nk} u^{n,\nu}=-\ri \nabla_{nk}\cdot (\nabla_{nk}\times u^{n,\nu})=0.$$
		
		The temporal exponential decay of $\psi$ is clear because $\psi(x,y,t) = e^{\omega_I t}\varphi(x,y,t)$ with a bounded $\varphi$ and $\omega_I<0$.
		
		To show that $\psi(\cdot,\cdot,t)\in H^m([0,\frac{2\pi}{|k|}],\cH^1)$ for all $t\geq -T_{\max}$, we first note that for any $m\in \N_0$
		$$
		\begin{aligned}
			\|\pa_y^m\psi(\cdot,y,t)\|_{\cH^1} &\leq \sum_{\nu \in \N} \sum_{|n|\leq \nu}|n|^m|k|^m\|u^{n,\nu}\|_{\cH^1}e^{\nu \omega_I t}\\
			&\leq\sum_{\nu \in \N} \sum_{|n|\leq \nu}|n|^m|k|^m e^{\nu\omega_I t}c_u^{-1}\widetilde{M}^\nu\varepsilon^\nu\\
			& \leq 2c_u^{-1}|k|^m \sum_{\nu \in \N}\nu^{m+1} \left(e^{\omega_I t}\widetilde{M}\eps\right)^\nu
		\end{aligned}
		$$
		by estimates (iv) and \eqref{eqn:estimate_u3_and_uplus}. This series converges geometrically and uniformly in $t\in [-T_{\max},\infty )$ if
		$$\eps<e^{\omega_I T_\text{max}}\widetilde{M}^{-1}$$
		to some value $N(k,t)<\infty$. Here we have used $\omega_I< 0$. Since, for each $m \in \N$ and $y_0\in \R$,
		$$\int_{y_0}^{y_0+2\pi/|k|}\|\pa_y^m\psi(\cdot,y,t)\|_{\cH^1}^2\dd y \leq \frac{2\pi}{|k|}N(k,t)^2  <\infty,$$
		we have, indeed, $y\mapsto \psi(\cdot,y,t)$ is $C^\infty_\mathrm{per}(\R)$ with the $y$-period $[0,\frac{2\pi}{|k|}]$.
		
		The same way one obtains  $\pa_t^p\psi(\cdot,\cdot,t)\in C_\mathrm{per}^\infty(\R,\cH^1)$ for all $p\in\N_0$. The convergence is uniform with respect to $t$ on $[-T_\mathrm{max},\infty)$ such that $$\cE, \cH \in C^\infty([-T_\mathrm{max},\infty),C^\infty_\text{per}(\R,\cH^1)).$$
		
		Lastly, notice that all integrals in the constitutive equation \eqref{eqn:Displacement} for the electric displacement field are now well defined and finite for all $t\geq 0$, because the electric field $\cE$ is only evaluated at times $s\geq -T$ or $s\geq - T_N$
		due to the finite memory of the system.
		
		This concludes the proof of Theorem \ref{T:main}.
	\end{proof}

	\section{Spectrum for an Interface of Two Homogeneous Layers}\label{sec:spec}
	
	It is the aim of the remaining sections to show that even in the case of two spatially homogeneous materials on each side of the interface at $x=0$ the assumptions of Theorem \ref{T:main} can be satisfied. We work with the physically relevant Lorentz material model. Although the well known Drude model is generally simpler, it is not suitable for our construction, see Appendix \ref{app:Drude} for more details. 
	
	It is expected that also in the case of non-homogeneous materials on either side of the interface assumptions of Theorem \ref{T:main} can be satisfied. However, checking this analytically will be more difficult and typically impossible.

	\subsection{Physical Models for $\chi^{(1)}$, $\chi^{(2)}$, and $\chi^{(3)}$}\label{PhyModel}
	
	\noindent\textbf{Standard Model for $\chi^{(1)}$}\\
	We will consider only the case of the interface of two homogeneous isotropic materials, for which the linear susceptibility is given by $\chi^{(1)}= \chi^{(1)}_\mathrm{sc}I_{3\times 3}$ with some scalar distribution $\chi^{(1)}_\mathrm{sc}\in\cD'_c(\R)$. Note that this diagonal form of $\chi^{(1)}$ satisfies the compatibility assumption \ref{ass:A-TM}  for the TM polarization. In the context of metallic structures and surface plasmons the Lorentz model for $\chi^{(1)}_\mathrm{sc}$ is commonly used. In the frequency domain it has the form
	
	\beq\label{E:Lorentz}
	\hat{\chi}^{(1)}_L(\omega)=-\frac{c_L}{\omega^2+2\ri\gamma\omega-\omega_*^2}
	\eeq
	with $c_L,\gamma,\omega_*>0$, see e.g. \cite{Pitarke_2007,ACL2018,CJM2023}. In the time domain we have
	
	$$\chi^{(1)}_L(t):=\frac{c_L}{\sqrt{\omega_*^2-\gamma^2}} e^{-\gamma t}\sin\left(t\sqrt{\omega_*^2-\gamma^2}\right)\theta(t),$$
	where $\theta$ is the Heaviside function and $\omega_*>\gamma$, see \cite{PBK2011}.
	
	Note that \eqref{E:Lorentz} is the Fourier-Laplace transform of $\chi^{(1)}_L$ only for $\Imag(\omega)>-\gamma$ as only for these values of $\omega$ the integral in \eqref{E:FT-L1} converges. In particular, $\hat{\chi}^{(1)}_L$  is not defined at all the points of $\mathbf{S}$. However, our construction uses all points in $\mathbf{S}$. That leads us to consider the truncated susceptibility $\chi^{(1)}_L \mathbbm{1}_{[0,T]}$. By the theorem of Paley-Wiener, see Theorem 7.22 in \cite{rudinFA}, we obtain that the Fourier-Laplace transform of such $\chi^{(1)}$ is analytic on the whole $\C$. Note that the temporal truncation results in a finite memory model. We define
	
	\beq\label{E:Lorentz-T}
	\chi^{(1)}_{L,T}(t):=\frac{c_L}{c_*} e^{-\gamma t}\sin\left(c_*t\right)\mathbbm{1}_{[0,T]}(t), \quad c_*:=\sqrt{\omega_*^2-\gamma^2}>0
	\eeq
	with Fourier-Laplace transform
	
	\beq\label{E:Lorentz-T-FT}
	\hat{\chi}^{(1)}_{L,T}(\omega)=-\frac{c_L}{\omega^2+2\ri\gamma\omega-\omega_*^2}\left[1+e^{(\ri \omega-\gamma )T}\left(\frac{\ri\omega-\gamma}{c_*}\sin\left(c_*T\right)-\cos\left(c_*T\right)\right)\right].
	\eeq
	
	The medium on the left side of the interface ($x<0$) will be described by the Lorentz model and on the right side we assume a non-dispersive medium with a positive and constant susceptibility. This means (in the notation \eqref{eqn:D-for-fns}) that $\chi^{(1)}_+=\alpha \delta I_{3\times 3}$ with $\alpha>0$ and $\delta$ being the (temporal) Dirac delta distribution in $\R$. Hence, we set
	\beq\label{E:Lorentz-interf}
	\hat{\chi}^{(1)}_\mathrm{sc}(x,\omega):=\begin{cases}
		\hat{\chi}^{(1)}_{L,T}(\omega), & x<0,\\
		\alpha>0, & x>0.
	\end{cases}
	\eeq
	The permittivity has the form $\perm_0(1+\chi^{(1)}_\mathrm{sc})I_{3\times 3}$ and the operator $B$ in \eqref{E:op-def} becomes
	\beq\label{E:B-diag}
	B(x,\omega)= -\omega \begin{pmatrix} \mu_0 \perm_0 (1+\hat\chi^{(1)}_\mathrm{sc}(x,\omega)) & 0 &0 \\
		0 & \mu_0 \perm_0 (1+\hat\chi^{(1)}_\mathrm{sc}(x,\omega)) & 0\\
		0 & 0 &1
	\end{pmatrix}.
	\eeq
	
	To simplify the notation, we  denote by $\perm$ the scalar factor of the permittivity, i.e.,
	$$\perm:=\perm_0(1+\chi^{(1)}_\mathrm{sc})$$
	and define also $\perm_\pm$ to be the restriction of $\perm$ to $\R_\pm$.
	With that we have
	\beq\label{E:perm-LT}
	\perm(x,\omega)=\begin{cases} \perm_-(\omega):= \perm_0(1+\hat{\chi}^{(1)}_{L,T}(\omega)), & x<0,\\
		\perm_+(\omega):= \perm_0(1+\alpha), & x>0.
	\end{cases}
	\eeq
	We denote the operators corresponding to the truncated Lorentz model by $\cL_{nk}^T$, i.e., with the operators $A_{nk}$ and $B$ in \eqref{E:op-def},
	\beq\label{E:Lnk-DL}
	\begin{aligned}
		\cL_{nk}^T(\omega)&:=A_{nk}-B(x,\omega) \quad \text{with} \ B \ \text{in} \ \eqref{E:B-diag},
	\end{aligned}
	\eeq
	where we again suppress the $x-$dependence in the operator notation.
	
	\medskip
	\noindent\textbf{Standard Model for $\chi^{(2)}$ and $\chi^{(3)}$}\\
	A canonical model for the quadratic and cubic susceptibilities $\hat{\chi}^{(2,3)}$ based on the Lorentz model of a homogeneous medium is
	\begin{equation*}\label{E:chi2}
		\hat{\chi}^{(2)}_{j,p,q}(\omega_1,\omega_2)=c^{(2)}_{j,p,q} \hat D(\omega_1)\hat D(\omega_2)\hat D(\omega_1+\omega_2),
	\end{equation*}
	\begin{equation*}\label{E:chi3}
		\hat{\chi}^{(3)}_{j,p,q,r}(\omega_1,\omega_2,\omega_3)=c^{(3)}_{j,p,q,r} \hat D(\omega_1)\hat D(\omega_2)\hat D(\omega_3)\hat D(\omega_1+\omega_2+\omega_3),
	\end{equation*}
	where $c^{(2)}_{j,p,q},c^{(3)}_{j,p,q,r}>0$ are material constants for ech $j,p,q,r\in \{1,2,3\}$ and
	$$\hat D(\omega)=-\frac{1}{\omega^2+2\ri \tilde\gamma \omega-\tilde\omega_*^2 }$$
	with $\tilde\omega_*>0$ being a ``resonant frequency" and $\tilde\gamma>0$, see \cite{boyd2008}, Sec. 1.4. 
	
	In the time domain these functions are given by
	$$
	\begin{aligned}
		\chi^{(2)}_{j,p,q}(t_1,t_2)&=c^{(2)}_{j,p,q} \int_\R D(s)D(t_1-s)D(t_2-s)\,\mathrm{d}s,\\
		\chi^{(3)}_{j,p,q,r}(t_1,t_2,t_3)&=c^{(3)}_{j,p,q,r} \int_\R D(s)D(t_1-s)D(t_2-s)D(t_3-s)\,\mathrm{d}s	\end{aligned}
	$$
	with $$D(t)=\frac{1}{\sqrt{\tilde\omega_*^2-\tilde{\gamma}^2}}e^{-\tilde{\gamma}t}\sin (t\sqrt{\omega_*^2-\tilde{\gamma}^2})\theta (t).$$
	Similarly to $\chi^{(1)}$, we truncate $\chi^{(2,3)}$ to satisfy \ref{ass:A-cpctSupport}. This ensures convergence of the convolutions in \eqref{eqn:Displacement}.  Choosing $T_N>0$, we work with
	\beq \label{E:chi23_trunc}
	\begin{aligned}
		\chi^{(2)}_{{T}_{j,p,q}}(t_1,t_2)&:= c^{(2)}_{j,p,q} \int_\R D(s)D(t_1-s)\mathbbm{1}_{[0,T_N]}(t_1)D(t_2-s)\mathbbm{1}_{[0,T_N]}(t_2)\,\mathrm{d}s,\\
		\chi^{(3)}_{{T}_{j,p,q,r}}(t_1,t_2,t_3)&:= c^{(3)}_{j,p,q,r} \int_\R D(s)D(t_1-s)\mathbbm{1}_{[0,T_N]}(t_1)D(t_2-s)\mathbbm{1}_{[0,T_N]}(t_2)D(t_3-s)\mathbbm{1}_{[0,T_N]}(t_3)\,\mathrm{d}s .
	\end{aligned}
	\eeq
	
	In the studied example of homogeneous media we set, for $m=2,3,$
	\beq\label{E:chi23-ex}
	\hat{\chi}^{(m)}(x,\omega_1,\dots,\omega_m)=\begin{cases} \hat{\chi}^{(m)}_-(\omega_1,\dots,\omega_m):= \hat{\chi}^{(m)}_{T}(\omega_1,\dots,\omega_m), & x<0,\\
		\hat{\chi}^{(m)}_+(\omega_1,\dots,\omega_m):=0, & x>0,
	\end{cases}
	\eeq
	and in order to satisfy \ref{ass:A-TM} we impose the conditions
	\beq \label{E:chi23-TM}
	c^{(2)}_{3,p,q}=0  \ \text{if} \ p,q\in \{1,2\}, \ \text{and} \   c^{(3)}_{3,p,q,r}=0  \ \text{if} \ p,q,r\in \{1,2\}.
	\eeq
	The choice $\hat{\chi}^{(2,3)}_+=0$ will be necessary to satisfy assumption \ref{ass:A-resolvEst} of Theorem \ref{T:main}. This is due to estimates which we prove for the resolvent in Section \ref{S:resolv}, see Corollary \ref{cor:summary_example}.
	
	The following Proposition \ref{verify_A1-A4} summarizes which assumptions of Theorem \ref{T:main} are satisfied so far.
	\bprop\label{verify_A1-A4}
	Let $\perm$, $\hat{\chi}^{(2)}$, and $\hat{\chi}^{(3)}$ be given by \eqref{E:perm-LT} and \eqref{E:chi23-ex} and let \eqref{E:chi23-TM} be satisfied. Then assumptions \ref{ass:A-cpctSupport}--\ref{ass:A-periodic} hold.
	\eprop
	\begin{proof}
		Both \ref{ass:A-cpctSupport} and \ref{ass:A-cont} are satisfied by the choice of $\chi^{(1)}_{L,T}$ and $\chi^{(2,3)}_T$ because each is either a Dirac delta or an $L^2$ function with the compact support $[0,T],$ $ [0,T_N]^2$, and $[0,T_N]^3$ respectively. Such functions induce distributions of compact support. Assumption \ref{ass:A-TM} follows from \eqref{E:chi23-TM} and assumption \ref{ass:A-periodic} holds trivially because $\chi^{(1)}_{L,T}$ and $\chi^{(2,3)}_T$ are independent of $x$ on each half-line $\R_\pm$.
	\end{proof}
	

	\subsection{Spectrum of the Interface Problem}\label{S:spec-interf}
	
	This section studies the spectrum of the interface problem for the Lorentz interface \eqref{E:Lorentz-interf}. Note that this is an operator pencil problem for which the spectrum was defined in Sec. \ref{sec.spectrum}. After discussing the existence of an eigenvalue $\omega_0$ as in \ref{ass:A-eval}, the aim is to check when \ref{ass:A-resolvent} holds.
	
	The spectrum for interfaces with spatially homogeneous layers in $\R_\pm$, as in example \eqref{E:Lorentz-interf}, was studied in \cite{BDPW25}. It was shown that irrespectively of the form of $\omega \mapsto \hat{\chi}_\mathrm{sc}^{(1)}(\omega)$ it consists of the point spectrum and the essential spectrum and has the following structure. Let $D(\perm)\subset \C$ be the intersection of the $\omega$-domains of $\perm_+$ and $\perm_-$, i.e., the set in which both $\perm_+$ and $\perm_-$ are defined. Outside the set
	\beq \label{eqn:Omega0} \Omega_0:=\{\omega\in D(\perm):\omega=0 \ \text{or} \ \perm_+(\omega)=0 \ \text{or} \ \perm_-(\omega)=0\}\eeq the eigenvalues are simple and for the wavenumber $nk$ they are given by solutions of the \textit{dispersion relation}
	\beq\label{E:ev.cond}
	n^2k^2(\perm_+(\omega)+\perm_-(\omega))=\omega^2\boldsymbol{\mu}_0\perm_+(\omega)\perm_-(\omega).
	\eeq
	Hence, for the operator $\cL^T_{nk}$ defined in \eqref{E:Lnk-DL} we have
	\begin{equation*}\label{E:pt-spec-red}
		\sigma_p(\cL^T_{nk})\setminus \Omega_0 = \{\omega\in D(\perm)\setminus\Omega_0: \eqref{E:ev.cond} \ \text{holds}\}.
	\end{equation*}
	The essential spectrum outside $\Omega_0$ consists of curves given via the conditions
	\beq\label{E:ess.cond.plus}
	\omega^2\boldsymbol{\mu}_0\perm_+ (\omega)\in
	\begin{cases}
		[n^2k^2,\infty) \ & \text{if} \ k\neq 0,\\
		(0,\infty) \ &\text{if} \ k=0,
	\end{cases}
	\eeq
	and
	\beq\label{E:ess.cond.minus}
	\omega^2\boldsymbol{\mu}_0\perm_- (\omega)\in
	\begin{cases}
		[n^2k^2,\infty) \ & \text{if} \ k\neq 0,\\
		(0,\infty) \ &\text{if} \ k=0,
	\end{cases}
	\eeq
	i.e.,
	\beq\label{E:ess-spec-red}
	\sigma_\text{ess}(\cL^T_{nk})\setminus \Omega_0 = \{\omega\in D(\perm)\setminus\Omega_0: \eqref{E:ess.cond.plus} \ \text{or} \  \eqref{E:ess.cond.minus} \ \text{holds}\}.
	\eeq
	Inside $\Omega_0$ eigenvalues of both finite and infinite multiplicity are possible.

	Interfaces with the untruncated model $\hat{\chi}^{(1)}_{L}$  (instead of $\hat{\chi}^{(1)}_{L,T}$)  were considered as examples in \cite{BDPW25}.
	We denote the operator in the untruncated case by $\cL^\infty_{nk}$, i.e., similarly to \eqref{E:Lnk-DL} we define for $\alpha>0$
	\begin{equation*}\label{E:Lnk-DL-infty}
		\begin{aligned}
			\cL^\infty_{nk}(\omega)&:=A_{nk}-B(x,\omega) \qquad  \text{with} \quad \hat\chi^{(1)} =  \left( \mathbbm{1}_{\{x<0\}}\hat{\chi}^{(1)}_L + \mathbbm{1}_{\{x>0\}}\alpha\right)I_{3\times3}.
		\end{aligned}
	\end{equation*}
	Note that the corresponding permittivity is given by
	$$
	\perm(x,\omega)=\begin{cases} \perm_-^\infty(\omega):= \perm_0(1+\hat{\chi}^{(1)}_{L}(\omega)), & x<0,\\
		\perm_+(\omega)= \perm_0(1+\alpha), & x>0.
	\end{cases}
	$$
	
	In the untruncated case determining the spectrum is straightforward because equation \eqref{E:ev.cond} and the relations \eqref{E:ess.cond.plus} and \eqref{E:ess.cond.minus} can be reduced to polynomial equations of degree at most four, which are solvable explicitly using computer algebra. In particular, there  are four eigenvalues and finitely many curves (four corresponding to $\perm_-$ and two corresponding to $\perm_+$) of the essential spectrum for $\cL^\infty_{nk}$. It was observed in \cite{BDPW25} that  the whole point spectrum of $\cL_{nk}^\infty$ outside of $\Omega_0$ is contained in the open strip $S_\gamma:=\{\omega \in \C: \Imag (\omega )\in (-\gamma ,0)\}$, i.e.,
	\beq\label{E:pt-spec-bds}
	\sigma_p(\cL^\infty_{nk}) \setminus \Omega_0 \subset S_\gamma=\{\omega \in \C: \Imag (\omega)\in (-\gamma,0)\}
	\eeq
	and the whole spectrum lies in the closed strip, i.e.,
	\beq\label{E:spec-bds}
	\sigma(\cL^\infty_{nk})\subset \overline{S_\gamma}= \{\omega\in \C: \Imag(\omega)\in [-\gamma,0]\}.
	\eeq
	These inclusions hold for all $n\in \N$ and $k \in \R$.
	In Figure \ref{fig:spectrum_lorentz}\subref{fig:rectangle} we plot the spectrum of $\cL_{k}^\infty$ for a concrete set of parameters.
	In the truncated case, where $\hat{\chi}^{(1)}_{L,T}$ involves exponential and trigonometric functions, equation \eqref{E:ev.cond} and the equations corresponding to \eqref{E:ess.cond.plus} and \eqref{E:ess.cond.minus} are transcendental and explicit  solutions are typically not available.
	
	We study the spectrum mainly in order to check assumption \ref{ass:A-resolvent}. The aim is to show that if \ref{ass:A-resolvent} holds in the untruncated case, i.e., if an eigenvalue $\omega_0^\infty = \omega_{R,\infty} + \ri \omega_{I,\infty}$ of $\cL_{k}^\infty$ satisfies
	\beq\label{E:NR-infty}
	n\omega_{R,\infty}+\ri\nu\omega_{I,\infty} \in \rho(\cL^\infty_{nk}) \quad \text{for all} \quad  (n,\nu) \in \mathbf{I}\setminus \{(1,1),(-1,1)\},
	\eeq
	then there is an eigenvalue $\omega_0$ of the truncated operator $\cL_{nk}^T$ such that \ref{ass:A-resolvent} holds for the truncation $T$ large enough. Due to \eqref{E:spec-bds} only finitely many indices $(n,\nu)$ need to be tested in assumption \eqref{E:NR-infty}. As explained above, assumption \eqref{E:NR-infty} can be checked easily using computer algebra. 
	
	In this section we use a combination of analytical and numerical methods to show hat this strategy works in the case of the above Lorentz interface. In fact, in the rest of this section we study only whether $\omega_0^{(n,\nu)}\notin \sigma_p(\cL_{nk}^T)$. For a rigorous proof of this property see Lemma \ref{L:ptsp-belowgam} and for a rigorous verification of \ref{ass:A-resolvent} see Prop. \ref{resol_set}.
	
	
	As we show here, the Lorentz model is well behaved under truncation, in the sense that the point spectrum of the interface problem perturbs only slightly due to a truncation with large enough $T$. The dispersion relation
	\eqref{E:ev.cond} for an interface between a dielectric and a truncated Lorentz material, i.e., the permittivity given by \eqref{E:perm-LT}, is equivalent to the equation
	\beq \label{E:dispRel-Lorenz}
	\begin{split}
		G_n(\omega,T):= \left((\perm_0^{-1}n^2k^2-\omega^2\boldsymbol{\mu}_0)\perm_+ +n^2k^2\right)\left( \omega^2 +2\ri \gamma\omega - \omega_*^2\right) - c_L(n^2k^2-\omega^2\boldsymbol{\mu}_0\perm_+) \\
		-c_L e^{(\ri\omega-\gamma )T}(n^2k^2-\omega^2\boldsymbol{\mu}_0\perm_+)\left( \frac{\ri\omega -\gamma}{c_*}\sin (c_*T)-\cos(c_*T)\right) =0,
	\end{split}
	\eeq
	where we recall that $c_*=\sqrt{\omega_*^2-\gamma^2} \in \R$, $n \in \N$, and $k \in\R$.
	Similarly, we define the function 
	\begin{equation*} \label{E:dispRel-Lorenz-untruncated}
		G_n^\infty(\omega):= \left((\perm_0^{-1}n^2k^2-\omega^2\boldsymbol{\mu}_0)\perm_+ +n^2k^2\right)\left( \omega^2 +2\ri \gamma\omega - \omega_*^2\right) - c_L(n^2k^2-\omega^2\boldsymbol{\mu}_0\perm_+) 
	\end{equation*}
	which encodes the dispersion relation of the untruncated operator $\cL^\infty_{nk}$.
	
	As discussed in \eqref{E:pt-spec-bds}, the point spectrum of $\cL_{nk}^\infty$ outside $\Omega_0$ is included in the strip $S_\gamma =\{\Imag (\omega )\in (-\gamma ,0)\}$
	for each $n\in \Z, k \in \R$. 
	We demonstrate that the point spectrum perturbs only slightly under the truncation with $T$ large. We start by employing the argument principle from complex analysis, see e.g. Theorem 4.18 in \cite{Ahlfors1979}, to determine
	the number of solutions to \eqref{E:dispRel-Lorenz} inside a large rectangle contained in $S_\gamma$. Namely, as $G_n(\cdot,T)$ is an entire function, the integral
	$$I_T:=\frac{1}{2\pi\ri}\int_\Gamma \frac{\pa_\omega G_n(\omega,T)}{G_n(\omega,T)}\,\mathrm{d}\omega$$
	along the boundary $\Gamma$ of some simply connected domain is an integer and gives precisely the number of zeros of $G_T$ contained in that domain. If we choose $\Gamma$ to be the edges of the rectangle $R_{a,\delta}\subset \C$ with vertices $\pm a$ and $\pm a +\ri (-\gamma +\delta)$ with some large $a>0$ and small $\delta >0$, as shown in Figure \ref{fig:spectrum_lorentz}\subref{fig:rectangle}, then the  rectangle includes all four eigenvalues of the untruncated operator. We calculate the
	integral $I_T$ numerically and obtain that $I_T=4$ if $T$ is large enough and $\delta$ is not too small. On the other hand, for any fixed $T$ we get  that $I_T>4$ if $\delta$ is chosen too small, i.e., additional eigenvalues are detected near $\Imag(\omega)=-\gamma$. Next, we check that for any $\delta>0$ (arbitrarily small) no additional eigenvalues appear in $R_{a,\delta}$, i.e., $I_T=4$ if $T=T(\delta)$ is chosen large enough.
	Hence, we define $\delta_0(T)>0$ to be the minimal value
	$\delta$ such that $I_T=4$ on $\Gamma:=\partial R_{a,\delta}$ (with a fixed $a$). Figure \ref{fig:spectrum_lorentz}\subref{fig:deltasLorenz} shows an almost reciprocal dependence between $T$ and $\delta_0$ for $a=1000$, supporting our claim.
	
	Next, we show analytically by the implicit function theorem that the four eigenvalues detected by the argument principle are close to the four eigenvalues in the untruncated model.
	\begin{figure}
		\centering
		\begin{subfigure}[t]{0.4\textwidth}
			\includegraphics[width=\textwidth]{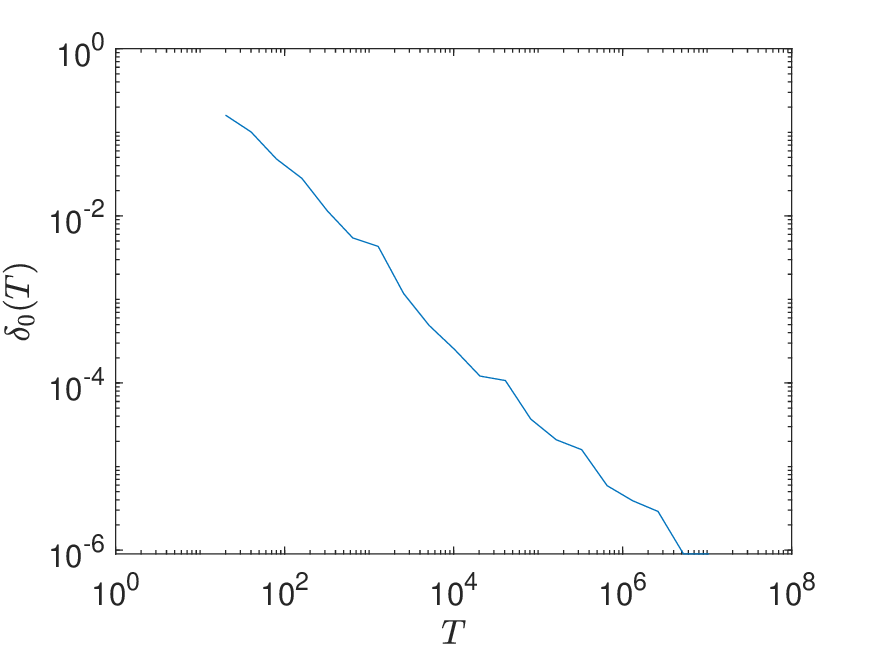}
			\caption{The minimal value $\delta_0 (T)$ such that $G_n(\cdot,T)$ has exactly four zeros with imaginary part greater than $-\gamma +\delta_0$. The zeros were detected using the numerical argument principle on $\Gamma = \pa R_{a,\delta_0}$, with $a=1000$. The logarithmic plot against $T$ shows that roughly $\delta_0(T)=O(T^{-0.96})$. }
			\label{fig:deltasLorenz}
		\end{subfigure}
		\hspace{0.03\textwidth}
		\begin{subfigure}[t]{0.4\textwidth}
			\includegraphics[width=\textwidth]{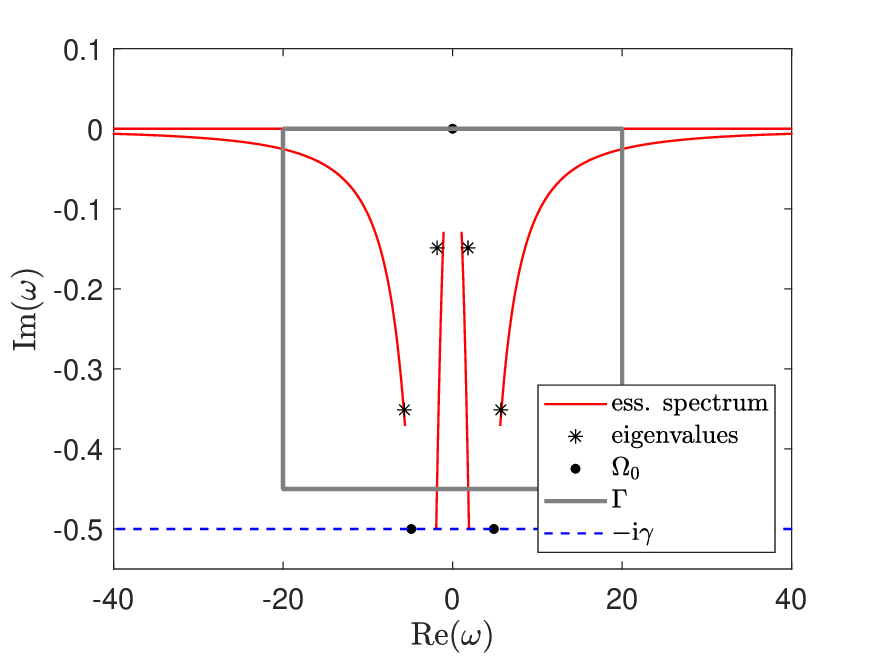}
			\caption{The spectrum of the untruncated operator $\cL_{3}^\infty$ with the parameters given by \eqref{eqn:parameters} and the rectangle $\Gamma$ used for the argument principle calculation. The rectangle $\Gamma$ has a width of $2a=40$ and a distance of $\delta=0.05$ to the blue dashed line marking all complex points with imaginary part $-\gamma$.
			}
			\label{fig:rectangle}
		\end{subfigure}
		\caption{\label{fig:spectrum_lorentz}%
			Numerical computations for the spectrum of the truncated Lorentz model. }
	\end{figure}
	
	\begin{lemma}
		\label{lem:evals_lorentz_close}
		Let $k \in \R, n\in \N,$ and $\omega_0^\infty\in S_\gamma$ be an eigenvalue of $\cL^\infty_{k}$ and assume that $\omega_0^\infty$ is a simple zero of the dispersion relation function $G^\infty_1$. Then, for each $\eps>0$ there exists $T_* >0$ such that $\cL_{k}^T$ with the truncation given by $T\geq T_*$ has an eigenvalue $\omega \in B_\eps(\omega_0^\infty)\subset \C$. Moreover, the eigenvalues of the truncated problem generate a continuous curve parameterized by $\tau:=1/T$ such that $\omega(\tau) \to \omega_0^\infty$ as $\tau\to 0$.
	\end{lemma}
	\begin{proof}
		Set $\tau:=1/T$ and $\phi :S_\gamma\times [0,\infty )\to \C$ with
		$$\phi (\omega,\tau ):=\begin{cases}
			G_1(\omega,1/\tau),~~\tau>0,\\
			G_1^\infty(\omega), ~~\tau =0.
		\end{cases}$$
		
		Clearly, $\phi$ is continuous and an easy calculation, using the fact that $e^{(\ri\omega -\gamma )T}$ decays for $T\to\infty$ if $\Imag(\omega) > -\gamma$, shows that $\phi\in C^1(S_\gamma\times [0,\infty ))$ with
		$$\partial_\omega \phi(\omega ,0)=\lim_{T\to\infty}\pa_\omega G_1(\omega,T).$$
		
		Now, by assumption, $\omega_0^\infty$ is a simple zero of $G_1^\infty$, thus
		$$\phi(\omega_0^\infty, 0)=0,~~\partial_\omega \phi (\omega_0^\infty ,0)\neq 0.$$
		The implicit function theorem then guarantees a continuous curve $\omega:\ [0,\eta )\to \C$, such that
		$\omega(\tau )$ is a solution of $G_1(\omega (\tau),1/\tau)=0$, i.e., an eigenvalue of $\cL^T_{k}$ with truncation given by $T=\frac{1}{\tau}$. In particular, for any $\eps>0$, one has $\omega (\tau)\in B_\eps (\omega_0^\infty)$
		if $\tau\leq \tau_0$ with $\tau_0$ small enough, by the continuity of the curve. We set $T_*:=\frac{1}{\tau_0}$.
	\end{proof}
	\begin{rem}
		The existence of an eigenvalue $\omega_0\in\C$ of the operator pencil $\cL_{k}^T$ is guaranteed by Lemma \ref{lem:evals_lorentz_close} because $\cL_{k}^\infty$ has eigenvalues. Hence, assumption  \ref{ass:A-eval} is satisfied.
	\end{rem}
	
	In order to compute the four eigenvalues of the truncated model numerically, we use a standard Newton's method applied to the function $G_1$. Lemma \ref{lem:evals_lorentz_close} guarantees that, at least for $T$ large, we can use the eigenvalues of the untruncated model as a starting guess for the Newton iteration. 
	Figure \ref{fig:omega_convergences} shows the convergence of an eigenvalue calculated that way. We choose $k=3$ in this test and the eigenvalue $\omega_0^\infty\approx 1.8179-0.1488 \ri$ of $\cL_{3}^\infty$.
	
	\begin{figure}
		\centering
		\begin{subfigure}[t]{0.4\textwidth}
			\includegraphics[width=\textwidth]{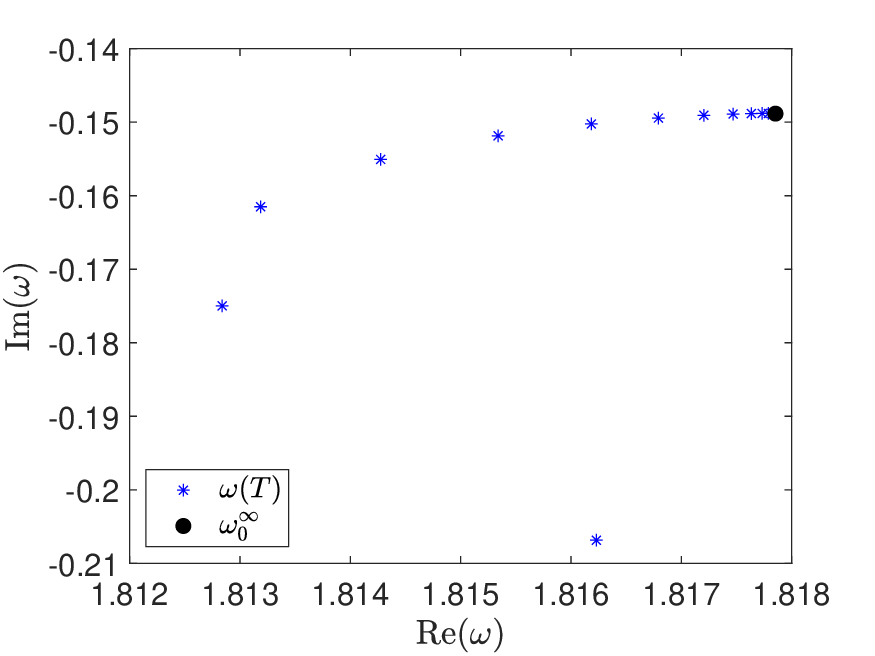}
			\caption{Eigenvalues $\omega (T)$ of the truncated operators $\cL_{3}^T$ with parameters specified in \eqref{eqn:parameters} and several values of $T$, converging to the eigenvalue $\omega_0^\infty\approx 1.8179 - 0.1488\ri$ of $\cL^\infty_{3}$.}
			\label{fig:omega_convergence}
		\end{subfigure}
		\hspace{0.03\textwidth}
		\begin{subfigure}[t]{0.4\textwidth}
			\includegraphics[width=\textwidth]{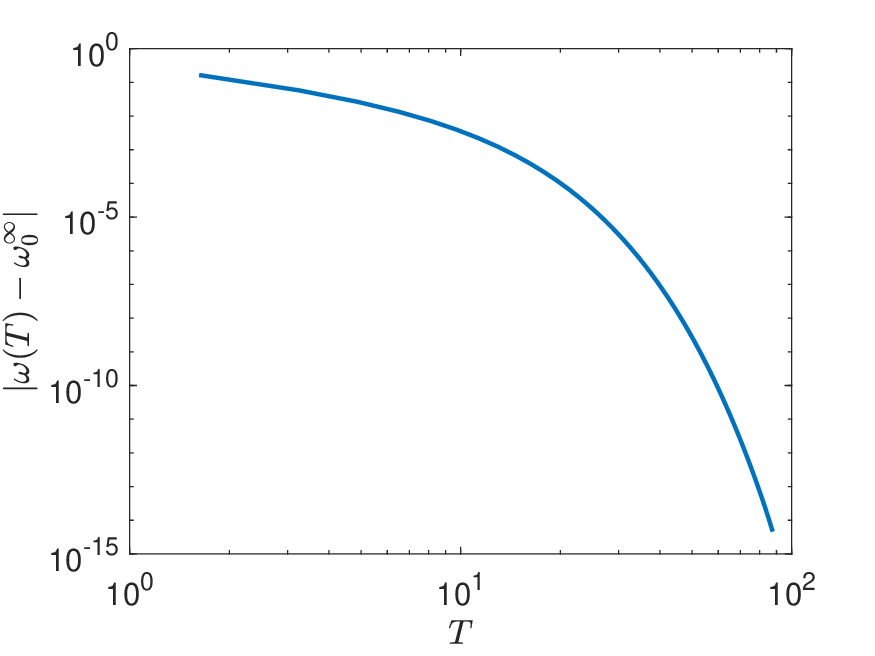}
			\caption{The rate of convergence for the same eigenvalues as in plot (a) in a log-log-plot shows super-polynomial convergence.}
			\label{fig:convergence_rate}
		\end{subfigure}
		\caption{\label{fig:omega_convergences}%
			Convergence of the numerically computed eigenvalues for the truncated Lorentz model.}
	\end{figure}
	
	The next step in checking assumption \ref{ass:A-resolvent} is to ensure that there is no intersection of $\mathbf{S}=\{n\omega_R +\ri\nu\omega_I:~(n,\nu)\in\mathbf{I}\}$ and the eigenvalues of the truncated Lorentz model below the line $\Imag(\omega)=-\gamma$. More precisely, we show that if the truncation $T$ is large enough and chosen as an integer multiple of $\pi/c_*$, then eigenvalues of $\cL_{nk}^T$ in the half plane 
	$$H_\delta := \{ \omega\in\C :~\Imag (\omega )\leq-\gamma -\delta\}$$ 
	with $\delta>0$ must lie outside the cone 
	$$C_l:=\{ \omega\in\C :~|\Real (\omega )|\leq l |\Imag (\omega )|\}$$ 
	with $l>0$. Clearly, if $l$ is chosen large enough, i.e., $l>\vert \frac{\omega_R}{\omega_I}|$, we get the inclusion $\mathbf{S}\subset C_l$.
	\begin{lemma}\label{L:evals-deep-down}
		For each $l,\delta>0$ there exists $J\in \N$ such that $\cL_{nk}^T$ with $T=T_j:=j\frac{\pi}{c_*}$ has no eigenvalues in the set $\left( H_\delta\cap C_l\right)\setminus \Omega_0$ for all $n \in \Z$ if $j\in\N, j \geq J.$
	\end{lemma}
	
	\begin{proof}
		Eigenvalues of $\cL_{nk}^T$ outside the set $\Omega_0$ are given as zeros of $G_n(\cdot,T)$ in \eqref{E:dispRel-Lorenz}.				
		Since $T_j=j\frac{\pi}{c_*}$, we have $\sin (c_*T_j)=0$ and $\cos(c_* T_j)=(-1)^j$. By multiplying \eqref{E:dispRel-Lorenz} with $e^{-(\ri\omega -\gamma)T_j}$ we obtain the equation
		\beq \label{E:dispRel-Lorenz_mod}
		\begin{aligned}
			&\left( \left( (\perm_0^{-1}n^2k^2-\omega^2\boldsymbol{\mu}_0)\perm_+ +n^2k^2\right)\left( \omega^2 +2\ri \gamma\omega - \omega_*^2\right) - c_L(n^2k^2-\omega^2\boldsymbol{\mu}_0\perm_+)\right) e^{-(\ri\omega-\gamma )T_j}\\
			&\quad + c_L (-1)^j(n^2k^2-\omega^2\boldsymbol{\mu}_0\perm_+) =0
		\end{aligned}
		\eeq
		for an eigenvalue $\omega$. Let $\delta, l>0$. Choosing now $\omega\in C_l$, we may estimate the absolute value of each $\omega$ appearing in the first term of \eqref{E:dispRel-Lorenz_mod} by $(1+l)|\Imag (\omega )|$, such that
		$$\left|\left( \left( (\perm_0^{-1}n^2k^2-\omega^2\boldsymbol{\mu}_0)\perm_+ +n^2k^2\right)\left( \omega^2 +2\ri \gamma\omega - \omega_*^2\right) - c_L(n^2k^2-\omega^2\boldsymbol{\mu}_0\perm_+)\right) e^{-(\ri\omega-\gamma )T_j}\right| \leq Q_l(|\Imag (\omega )|) e^{(\Imag (\omega )+\gamma )T_j},$$
		with some polynomial $Q_l$. Clearly, for $\omega\in H_\delta$ the term $Q_l(|\Imag (\omega )|)e^{(\Imag (\omega )+\gamma )T_j/2}$ is bounded in $\Imag (\omega )$ by some constant $c(l)$ independent of $T_j$.
		
		We thus arrive at
		$$
		\begin{aligned}
			&\left|\left( \left( (\perm_0^{-1}n^2k^2-\omega^2\boldsymbol{\mu}_0)\perm_+ +n^2k^2\right)\left( \omega^2 +2\ri \gamma\omega - \omega_*^2\right) - c_L(n^2k^2-\omega^2\boldsymbol{\mu}_0\perm_+)\right) e^{-(\ri\omega-\gamma )T_j}\right|\\
			& \leq c(l) e^{(\Imag (\omega )+\gamma )T_j/2}\leq c(l) e^{-\delta T_j/2}
		\end{aligned}
		$$
		for all $\omega\in C_l\cap H_\delta$. Finally, there exists $\varepsilon >0$ such that
		$$\left|c_L (-1)^j(n^2k^2-\omega^2\boldsymbol{\mu}_0 \perm_+)\right|\geq c_L\boldsymbol{\mu}_0 \perm_+|\Imag (\omega^2)|>\eps$$
		for all $\omega \in C_l\cap H_\delta$. Choosing now $j$ large enough, we get $c(l)e^{-\delta T_j/2}<\eps$ and hence \eqref{E:dispRel-Lorenz_mod} cannot hold inside $H_\delta\cap C_l$.
	\end{proof}
	
	To summarize, let $\omega^\infty_0 = \omega_{R,\infty} + \ri \omega_{I,\infty}$ be a simple eigenvalue of $\cL^\infty_{k}$ and $\omega_0=\omega_R+\ri\omega_I$ the eigenvalue of $\cL_{k}^T$ that lies on the curve through $\omega_0^\infty$ given by Lemma \ref{lem:evals_lorentz_close}.
	Lemma \ref{L:evals-deep-down} shows that, choosing in particular $l> |\omega_{R,\infty}/\omega_{I,\infty}|$, for  $T=j\frac{\pi}{c_*}$ with $j$ large enough, eigenvalues of $\cL_{nk}^T$ below the strip $S_\gamma$ lie either asymptotically (as $T\to\infty$) close to the line $\Imag(\omega)=-\gamma$ or they satisfy $|\Real (\omega )|> l|\Imag (\omega )|$. Since $\omega_0$ converges to $\omega_0^\infty$ for $T\to\infty$, also $l>|\omega_R /\omega_I|$ will be satisfied for large enough $T$. Hence, for such values of $T$ we have proved 
	$$
	\omega_0^{(n,\nu)} \notin \sigma_p(\cL_{nk,L}) \cap H_\delta \quad \forall (n,\nu)\in \mathbf{I}.
	$$
	Inside $S_\gamma$, where the eigenvalues of $\cL_{nk}^\infty$ lie, Lemma \ref{lem:evals_lorentz_close} shows that for $T$ large there are eigenvalues of $\cL_{nk}^T$ close to those of $\cL_{nk}^\infty$. Additional eigenvalues are possible but our winding number computations suggest that these lie either close to the line $\Imag(\omega)=-\gamma$ or have a large real part. In the latter case these eigenvalues do not intersect  the discrete cone $\mathbf{S}$. In the former case the intersection with $\mathbf{S}$ is also empty if 
	\beq \label{E:res-im-as}
	\nu\omega_{I,\infty} \neq -\gamma \quad \forall \nu \in \N.
	\eeq
	Indeed, the discreteness of $\mathbf{S}$ and Lemma \ref{lem:evals_lorentz_close} then imply
	$$
	d:=\dist\left(\mathbf{S},\left\{\omega\in \C:\Imag(\omega)=-\gamma\right\}\right)>0.
	$$
	
	Hence, relying on the winding number computations within $S_\gamma$, we expect 
	\beq\label{E:ptspec-inters}
	\omega_0^{(n,\nu)} \notin  \sigma_p(\cL_{nk}^T) \quad \forall (n,\nu)\in \mathbf{I}\setminus\{(1,1),(-1,1)\}
	\eeq
	if $T$ is large enough, if \eqref{E:res-im-as} holds, and if 
	\beq\label{E:pt-cpes-cond-untrunc}
	n\omega_{R,\infty}+\ri\nu \omega_{I,\infty} \notin \sigma_p(\cL_{nk,L}^\infty)
	\quad \forall (n,\nu)\in \mathbf{I}\setminus\{(1,1),(-1,1)\}.
	\eeq
	We note that assumption \eqref{E:pt-cpes-cond-untrunc} needs to be checked only for finitely many values of $(n,\nu)$ because  $\sigma(\cL_{\kappa}^\infty)\subset \overline{S_\gamma}$ for each $\kappa\in \R$. This can be done using computer algebra.
	
	A rigorous proof of \eqref{E:ptspec-inters} can be found in Lemma \ref{L:ptsp-belowgam}, which uses \eqref{E:res-im-as} and \eqref{E:pt-cpes-cond-untrunc} as assumptions.
	
	%
	
	\section{Assumptions and Proof of Theorem \ref{T:Lor-example}}\label{sec:proof-Thm-2}
	In this section we study the existence of polychromatic solutions, as given by Theorem \ref{T:Lor-example}, for the interface problem with $\perm$ given by \eqref{E:perm-LT} and $\chi^{(2)}, \chi^{(3)}$ by \eqref{E:chi23-ex} and \eqref{E:chi23-TM}.	
	
	\subsection{Assumptions of Theorem \ref{T:Lor-example}}

	First, we introduce our assumptions \ref{ass:B-signOfOmega}--\ref{ass:B-rationality} used in Theorem \ref{T:Lor-example}. Mainly, they are assumptions on a selected eigenvalue $\omega_0^\infty = \omega_{R,\infty} +\ri  \omega_{I,\infty}$ of the untruncated operator $\cL_{k}^\infty$, but \ref{ass:B-rationality} is an assumption on a corresponding eigenvalue $\omega_0=\omega_R+\ri \omega_I$ of the truncated $\cL_{k}^T$ at some $T$ large. Here and in the following, we will denote by $$c_0:=1/\sqrt{\perm_0\mu_0}$$ 
	the speed of light in vacuum.
	\begin{enumerate}[label=(B\arabic*),ref=(B\arabic*)]
		\item \label{ass:B-signOfOmega}
		$\omega_0^\infty= \omega_{R,\infty} +\ri  \omega_{I,\infty}$ is an eigenvalue of $\cL_{k}^\infty$ and a simple zero of $G_1^\infty$ with 
		$\omega_{I,\infty}<0$, $\omega_{R,\infty}>0$.
		\item \label{ass:B-omIfarFromGamma}
		$\nu\omega_{I,\infty}+\ga\neq0$ for all $\nu\in\N$.
		\item \label{ass:B-Vmin}
		For each $(n,\nu)\in\mathbf{I}$ with $\nu<\frac{\ga}{|\omega_{I,\infty}|}$,
		\[V_{-,\infty}^{(a)}(n,\nu):=-c_0^{-2}\left(n\omega_{R,\infty}+\ri\nu\omega_{I,\infty}+\frac{c_L\left(\alpha_\infty+\ri\beta_\infty\right)}{\left(c_*^2+(\gamma+\nu\omega_{I,\infty})^2-n^2\omega_{R,\infty}^2\right)^2+4n^2\omega_{R,\infty}^2(\ga+\nu\omega_{I,\infty})^2}\right)\neq0, \]
		where
		\[\alpha_\infty:=n\omega_{R,\infty}\left(\omega_*^2-n^2\omega_{R,\infty}^2-\nu^2\omega_{I,\infty}^2\right),\quad \beta_\infty:=\left(n^2\omega_{R,\infty}^2+\nu^2\omega_{I,\infty}^2\right)\left(2\gamma+\nu\omega_{I,\infty}\right)+\omega_{*}^2\nu\omega_{I,\infty}.\]
		\item \label{ass:B-muMin}
		For each $(n,\nu)\in \mathbf{I}$ with $\nu<\frac{\ga}{|\omega_{I,\infty}|}$, \[\begin{split}
			\mu_{-,\infty}^2(n,\nu):=&~n^2k^2-c_0^{-2}\left(n^2\omega_{R,\infty}^2-\nu^2\omega_{I,\infty}^2+2\ri n\omega_{R,\infty}\nu\omega_{I,\infty}\right)\\
			&-c_0^{-2}c_L\frac{\tilde{\alpha}_\infty(n,\nu)+\ri\tilde{\beta}_\infty(n,\nu)}{\left(c_*^2+(\ga+\nu\omega_{I,\infty})^2-n^2\omega_{R,\infty}^2\right)^2+4n^2\omega_{R,\infty}^2(\gamma+\nu\omega_{I,\infty})^2}\neq0,
		\end{split}\]
		where 
		$$
		\begin{aligned}
			\tilde{\alpha}_\infty(n,\nu)&:=-(n^2\omega_{R,\infty}^2 + \nu^2\omega_{I,\infty}^2)(n^2\omega_{R,\infty}^2 + \nu^2\omega_{I,\infty}^2 + 2\gamma\nu\omega_{I,\infty}) - \omega_*^2(\nu^2\omega_{I,\infty}^2-n^2\omega_{R,\infty}^2),\\
			\tilde{\beta}_\infty(n,\nu)&:=2n\omega_{R,\infty}\left[\gamma(n^2\omega_{R,\infty}^2 + \nu^2\omega_{I,\infty}^2) + \omega_*^2\nu \omega_{I,\infty}\right].
		\end{aligned}
		$$
\item \label{ass:B-betatil-phi-inf}
		$\tilde{\beta}_\infty(n,\nu)\neq 0$ for each $(n,\nu)\in  \mathbf{I}$ with $\nu>\frac{\ga}{|\omega_{I,\infty}|}$ and $n \neq 0$
		and 
			$$\varphi_\infty(n,\nu):=-\tilde{\beta}_\infty(n,\nu)-2\frac{n\omega_{R,\infty}\nu\omega_{I,\infty}}{c_L}\left[(c_*^2+(\gamma+\nu\omega_{I,\infty})^2-n^2\omega_{R,\infty}^2)^2+4n^2\omega_{R,\infty}^2(\ga+\nu\omega_{I,\infty})^2\right]\neq 0$$
		for each $(n,\nu)\in  \mathbf{I}$ with $\nu<\frac{\ga}{|\omega_{I,\infty}|}$ and $n \neq 0$.
		\item \label{ass:B-nonResonance}
		$n\omega_{R,\infty}+\ri\nu\omega_{I,\infty}\notin \sigma_p(\cL^\infty_{nk,L}) \ \forall (n,\nu )\in \mathbf{I}\setminus \{(1,1),(-1,1)\}$ with $\nu < \frac{\gamma}{|\omega_{I,\infty}|}$.
		\item \label{ass:B-formOfT}
		$T=T_j:=j\frac{\pi}{c_*}, j \in \N_\mathrm{odd}$.
		\item \label{ass:B-rationality}
		There exists $j_*\in \N_{\mathrm{odd}}, j_*\geq \max\{\lceil T_+\frac{c_*}{\pi}\rceil ,j_1,j_2,j_3,j_4\}$ with $T_+, j_1, j_2, j_3$, and $j_4$ from Lemmas \ref{bds_V_pm}, \ref{L:ptsp-belowgam}, \ref{L:j1-phi-beta}, \ref{bds_V_pm_L}, and \ref{bd_abs_mu_minu_L}, respectively, such that the following holds. The eigenvalue $\omega_0 =\omega_R+\ri \omega_I$ of $\cL_{k}^T$ for the truncation $T=T_{j_*}=j_*\frac{\pi}{c_*},$ on the curve through $\omega_0^\infty$, given by Lemma \ref{lem:evals_lorentz_close}, satisfies
		\beq\label{E:omR-form}
		\frac{\omega_R}{c_*}=\frac{2m}{j_*}, \ \text{for some} \ m\in \Z.
		\eeq
	\end{enumerate}
	
		\brem \label{R:beta-phi-inf-asymp} Although assumption \ref{ass:B-betatil-phi-inf} for $\tilde{\beta}_\infty$ is formulated for all $(n,\nu)\in \mathbf{I}$ with $\nu>\frac{\ga}{|\omega_{I,\infty}|}, n\neq 0$, it is effectively an assumption on only finitely many $(n,\nu)$. This is because of the large $\nu$ asymptotics of $\tilde{\beta}_\infty$ for $n\neq 0$. Indeed, it is easy to see that there is $c>0$ such that for all $(n,\nu)\in \mathbf{I}$ with $\nu$ large enough and $n \neq 0$ we have 
		$$
		|\tilde{\beta}_\infty(n,\nu)|>c\nu^2.
		$$
		\erem
		\brem
	Clearly, the eigenvalue $\omega_0$ of $\cL^T_{k}$ depends on the truncation parameter $T$ (which determines $\perm_-$). Hence (if $T=T_j$ is chosen as in \ref{ass:B-formOfT}) $\omega_0$ depends on $j$. Assumption \ref{ass:B-rationality} then prescribes a certain dependence between $\omega_R$ and $T=T_{j_*}$. This can be seen because 
	$$\frac{\omega_R}{c_*}=\frac{2m}{j_*} \ \Leftrightarrow \ \frac{\omega_R}{c_*}=\frac{2m\pi}{c_*T}  \ \Leftrightarrow  \ \frac{\omega_R}{\pi}=\frac{2m}{T}.$$
	Hence condition \eqref{E:omR-form} is equivalent to 
	$$\omega_R\in \frac \pi T \Z_{\mathrm{even}}.$$
	Note that because $T$ can be chosen arbitrarily large, also small values of $\omega_R$ are, in principle, possible.
	\erem	
	\brem   
	Assumptions \ref{ass:B-formOfT} and \ref{ass:B-rationality} are merely to simplify the susceptibility in \eqref{E:Lorentz-T-FT} in a way that lets us prove that the interface given by \eqref{E:perm-LT} satisfies \ref{ass:A-resolvEst}. 
	Namely, the sine and cosine terms and the factor $e^{\ri \omega_R T}$ in \eqref{E:Lorentz-T-FT} are replaced by $0, -1$ and $1$ respectively. In particular, the oddness assumption in \ref{ass:B-formOfT} simplifies  $\hat{\chi}^{(1)}_{L,T}$ because $\cos(c_*T)=-1$ and $\sin (c_*T)=0$, see the proof of Lemma \ref{bds_V_pm_L}. Assumption \ref{ass:B-formOfT} simply reduces the choice of the truncation parameter but assumption \ref{ass:B-rationality} is rather restrictive and probably impossible to be checked analytically or numerically.	
	We expect, though, that one can prove the validity of \ref{ass:A-resolvEst} for the interface \eqref{E:perm-LT}  also with generic values of $T$ and $\omega_R$.
	\erem

	In the following, we choose $\omega_0=\omega_R+\ri\omega_I\in \sigma_p(\cL_{k}^T)$ as in assumption \ref{ass:A-eval} to be again such that it lies on the curve through the eigenvalue $\omega_0^\infty$ as provided by Lemma \ref{lem:evals_lorentz_close}.
	
	For the spectrum and resolvent estimates it turns out to be crucial to control the quantities
	$$V_\pm(n,\nu):=-\omega_0^{(n,\nu)}\boldsymbol{\mu}_0\perm_\pm(\omega_0^{(n,\nu)}) \quad \text{and} \quad \mu_\pm(n,\nu):=\sqrt{n^2k^2+\omega_0^{(n,\nu)} V_\pm(n,\nu)}, \quad (n,\nu)\in \mathbf{I}.$$
	
	Note that we define for any complex number $z=|z|e^{\ri \varphi}\in\C$ with $\varphi \in (-\pi,\pi]$ the principal square root as
	$$\sqrt{z}:=|z|^{1/2}e^{\ri \varphi /2}.$$
	Hence, $\Real (\sqrt{z})\geq 0$ always holds. In particular, $\Real (\mu_\pm)\geq 0.$
	
	It is important to realize that $\omega_0,V_\pm$, and $\mu_\pm$ depend on the truncation parameter $T$.  For $\omega_0$, $V_-$ and $\mu_-$ this fact is immediate from their definition and the $T$-dependence of $V_+$ and $\mu_+$ is inherited from their dependence on $\omega_0$.

	\brem  
	Assumptions \ref{ass:B-Vmin} and \ref{ass:B-muMin} are used to estimate $|V_-|, |\mu_-|$ and $\Real(\mu_-)$ in Lemmas \ref{bds_V_pm_L}, \ref{bd_abs_mu_minu_L} and \ref{bd_re_mu_minu_L}. These are important both in showing that $\omega_0^{(n,\nu)}$ does not lie in the essential spectrum  $\sigma_\mathrm{ess}(\cL_{nk}^T)$ and in the resolvent estimates, i.e., in verifying \ref{ass:A-resolvEst}, see Corollary \ref{cor:summary_example}. 
	\erem 
	\medskip
	
	We proceed with the proof of  Theorem \ref{T:Lor-example}. The proof consists of verifying that assumptions 	 \ref{ass:A-cpctSupport}--\ref{ass:A-resolvEst} of Theorem \ref{T:main} hold for the selected interface problem under assumptions \ref{ass:B-signOfOmega}--\ref{ass:B-rationality}. 
	Assumptions \ref{ass:A-cpctSupport}--\ref{ass:A-periodic} have already been verified in Proposition \ref{verify_A1-A4} and the existence of eigenvalues for $\cL_{k}^T$, i.e., \ref{ass:A-eval}, is guaranteed by the discussion in Section \ref{S:spec-interf}. It thus remains to only check assumptions \ref{ass:A-resolvent} and \ref{ass:A-resolvEst}.
	
	\subsection{Proof of Theorem \ref{T:Lor-example}: Estimates for a Non-dispersive Material}
	We first consider the right side of the interface, i.e., the non-dispersive layer in $x>0$. Here the permittivity is given by $\perm_+(\omega)=\perm_0(1+\alpha)>0$, see \eqref{E:perm-LT}. We have the following lower and upper bounds on $V_+$ and $\mu_+$.
	\begin{lemma}\label{bds_V_pm}
		There are constants $T_+$, $c_1,c_2, a_+,b_+,c_+$ and $d_+$ such that
		\begin{subequations}\label{est-abs_V_pm}
			\beq\label{est-abs_V_plus}
			c_1\nu\leq\left|V_+(n,\nu)\right|\leq c_2\nu \\
			\eeq
			\beq\label{est-re_mu_plus}
			a_+ \nu \leq \Real\mu_+(n,\nu) \leq b_+ \nu
			\eeq
			\beq\label{est-abs-mu-plus}
			c_+\nu \leq |\mu_+(n,\nu)|\leq d_+\nu
			\eeq
		\end{subequations}
		for all $T>T_+$ and $(n,\nu)\in  \mathbf{I}=\{(n,\nu)\in \Z \times \N: |n|\leq \nu\}.$
	\end{lemma}
	\begin{proof}
		We define $\tau^2:=c_0^{-2}(1+\alpha)$.
		
		Clearly, for the constants $\tilde{c}_1=\tau^2|\omega_I|>0$, $\tilde{c}_2=\tau^2\sqrt{\omega_R^2 +\omega_I^2}>0$ one has
		$$\tilde{c}_1\nu\leq \left|V_+(n,\nu)\right|=\tau^2\left|n\omega_R+\mathrm{i}\nu\omega_I\right|=\tau^2\sqrt{n^2\omega_R^2+\nu^2\omega_I^2}\leq \tilde{c}_2\nu,\quad \forall \ \nu\in\mathbb{N},\ |n|\leq \nu.$$
		Then, due to the convergence $\omega_0\to \omega_0^\infty$, $T\to\infty$, we obtain the bounds $\tilde{c_1}\geq \frac{1}{2}\tau^2|\omega_{\infty,I}|=:c_1$ and $\tilde{c_2}\leq 2\tau^2\sqrt{\omega_{\infty ,R}^2 +\omega_{\infty ,I}^2} =:c_2$ for all $T>T_0$. These $T$-independent constants conclude estimate \eqref{est-abs_V_plus}.
		
		Next, we have
		\begin{equation*}\label{mu-plus}
			\mu_+(n,\nu) = \sqrt{n^2k^2 - \tau^2 (n\omega_R+\ri \nu \omega_I)^2} = r_+^{1/2}e^{\ri \frac{\theta_+}{2}},
		\end{equation*}
		where $r_+=|\mu_+|^2=\sqrt{(n^2(k^2-\tau^2\omega_R^2) + \nu^2\tau^2\omega_I^2)^2 +4n^2\nu^2\tau^4\omega_R^2\omega_I^2}$ and $\theta_+ = \text{Arg}(n^2k^2 - \tau^2 (n\omega_R+\ri \nu \omega_I)^2) \in (-\pi,\pi]$.
		
		To estimate $\Real \mu_+ := r_+^{1/2} \cos \left( \frac{\theta_+}{2}\right)$ we give upper and lower bounds on $|\mu_+|$ and show that $\theta_+\in (-\pi +\delta ,\pi -\delta )$ for some fixed $\delta >0$, as this assures that $\cos \left( \frac{\theta_+}{2}\right)$ is bounded away from zero.
		
		A simple calculation produces
		$$r_+=(n^4(k^2-\tau^2\omega_R^2)^2 +2n^2\nu^2\tau^2\omega_I^2(k^2+\tau^2\omega_R^2)+\nu^4\tau^4\omega_I^4)^{1/2}.$$
		Thus, for all $(n,\nu)\in \mathbf{I}$,
		$$r_+ \in (\tilde{c}_+^2 \nu^2,\  \tilde{d}_+^2\nu^2),$$
		where $\tilde{c}_+^2:=\tau^2\omega_I^2$ and $\tilde{d}_+^2:=((k^2-\tau^2\omega_R^2)^2+2\tau^2\omega_I^2(k^2+\tau^2\omega_R^2)+\tau^4\omega_I^4)^{1/2}$. Since both quantities depend continuously on $\omega_R,\omega_I$, we can again employ the convergence of $\omega_0$ to $\omega_0^\infty$ and thus set 
		$$c_+^2:=\frac{1}{2}\tau^2\omega_{\infty ,I}^2 ~\text{   and   } ~d_+^2:=2((k^2-\tau^2\omega_{\infty ,R}^2)^2+2\tau^2\omega_{\infty ,I}^2(k^2+\tau^2\omega_{\infty ,R}^2)+\tau^4\omega_{\infty ,I}^4)^{1/2}.$$ Hence, for all $T$ large enough
		$$c_+\nu\leq \left|\mu_+\right|\leq d_+\nu.$$
		For the analogous estimates of $\Real \mu_+$ it remains to show that $\cos\left(\frac{\theta_+}{2}\right)\geq c>0.$
		If $k^2\geq \tau^2\omega_R^2$, then it follows that
		$\Real\mu_+^2=n^2(k^2-\tau^2\omega_R^2) + \nu^2\tau^2\omega_I^2\geq c_I\nu^2 \geq 0$
		for all $(n,\nu)\in \mathbf{I}$,
		where $c_I:=\tau^2\omega_I^2>0$,  hence $\theta_+\in\left[-\frac{\pi}{2},\frac{\pi}{2}\right]$, hereby $\cos\left(\frac{\theta_+}{2}\right)\geq c>0$.
		
		If  $k^2< \tau^2\omega_R^2$ with $n^2\geq \frac{\tau^2\omega_I^2}{\tau^2\omega_R^2-k^2}\nu^2$ and $|n|\leq\nu$, then it follows that $-\tilde{c}_0\nu^2\leq\Real\mu_+^2\leq0,$
		where $\tilde{c}_0:=\tau^2\omega_R^2-k^2>0.$ In addition, we observe
		$\left|\Imag\mu_+^2\right|\geq \tilde{c}_1\nu^2$ with $\tilde{c}_1:=\frac{2\tau^3\omega_I^2|\omega_R|}{\sqrt{\tau^2\omega_R^2-k^2}}$. Hence $$\frac{\left|\Imag \mu_+^2\right|}{\left|\Real \mu_+^2\right|}\geq\frac{\tilde{c}_1}{\tilde{c}_0}>0$$
		for all $\nu\in\N$ and $|n|\leq\nu$ with $n^2\geq \frac{\tau^2\omega_I^2}{\tau^2\omega_R^2-k^2}\nu^2$. This implies $\theta_+\in [-\pi+\delta, \pi-\delta]$ with $\delta =\arctan\left(\frac{\tilde{c}_1}{\tilde{c}_0}\right)>0$, and thus  $\cos\left(\frac{\theta_+}{2}\right)\geq c>0$.
		
		Finally, if  $k^2< \tau^2\omega_R^2$ with $n^2< \frac{\tau^2\omega_I^2}{\tau^2\omega_R^2-k^2}\nu^2$ and $|n|\leq\nu$, then it follows that $\Real\mu_+^2>0$, hence $\theta_+\in\left(-\frac{\pi}{2},\frac{\pi}{2}\right)$, hereby $\cos\left(\frac{\theta_+}{2}\right)\geq c>0$. Again, in all previous estimates $\omega_R,\omega_I$ can be replaced 
		by $\omega_{\infty ,R},\omega_{\infty ,I}$, scaling the occurring constants by $\frac{1}{2}$ or $2$. Hence, $\cos\left( \frac{\theta_+}{2}\right)$ is bounded away from $0$ uniformly in $T$ and \eqref{est-re_mu_plus} holds.
	\end{proof}
	
	\subsection{Proof of Theorem \ref{T:Lor-example}: Estimates for a Lorentz Material with Memory Truncation}
	
	Next we study the material in \eqref{E:perm-LT} in $x<0$ described by the Lorentz model with the memory truncation and
	obtain lower and upper bounds for $V_-$,  $|\mu_-|$, and $\Real (\mu_-)$.
	For $\nu <\gamma/|\omega_I|$ we bound these quantities analogously to Lemma \ref{bds_V_pm}. For $\nu > \gamma/|\omega_I|$, i.e., for large $\nu$, we show that $V_-$, $|\mu_-|$, and  $\Real(\mu_-)$ grow exponentially fast in $\nu$ (and in $T$).
	We are able to provide such bounds only under assumptions \ref{ass:B-signOfOmega}--\ref{ass:B-rationality}.
	
	Let us recall the definitions $c_0=1/\sqrt{\perm_0\boldsymbol{\mu}_0}$ and $c_*=\sqrt{\omega_*^2-\gamma^2}$ and introduce the concept of asymptotic equivalence. For two terms $A(t)$ and $B(t)$ we denote their asymptotic equivalence by the $\sim$ symbol in the sense that 
	$$ A(t) \sim B(t)\, (t\to t_0)~~\Leftrightarrow ~~ \lim_{t\to t_0} \frac{A(t)}{B(t)}=1.$$
	
	\blem\label{L:j0}
	Assume \ref{ass:B-signOfOmega}--\ref{ass:B-muMin}. Then there exists $T_0>0$ such that assumptions \ref{ass:B-signOfOmega}--\ref{ass:B-muMin} hold with $\omega_0^\infty=\omega_{R,\infty}+\ri \omega_{I,\infty}$ replaced by $\omega_0=\omega_R+\ri\omega_I$, where $\omega_0$ is the eigenvalue for any truncation $T\geq T_0$ on the curve through $\omega_0^\infty$ given by Lemma \ref{lem:evals_lorentz_close}. Also the sets $\{(n,\nu)\in\mathbf{I}: \nu<\gamma/|\omega_{I,\infty}|\}$ and $\{(n,\nu)\in\mathbf{I}: \nu<\gamma/|\omega_I|\}$ are identical for $T\geq T_0$.
	\elem
	\bpf
	Lemma \ref{lem:evals_lorentz_close} guarantees that for any $\delta>0$ and $T=T(\delta)$ large enough there is an eigenvalue $\omega_0$ in the $\delta$-vicinity of $\omega_0^\infty$ and $\omega_0(T)\to \omega_0^\infty$ as $T\to\infty$ (after reparametrising the curve from Lemma \ref{lem:evals_lorentz_close} by $T$ instead of $1/T$). This and the fact that all quantities in \ref{ass:B-signOfOmega}--\ref{ass:B-muMin} are continuous in $\omega_R$ and $\omega_I$ in a vicinity of $\omega_0^\infty$ imply the statement.
	\epf
	
	In the proofs of the remaining statements the sets 
	\begin{equation*}
		\label{eqn:setsIaIb}
		\mathbf{I}_a:=\left\{(n,\nu)\in \mathbf{I}:\ \nu <\frac{\gamma}{|\omega_{I,\infty}|}\right\} \quad \text{and} \quad \mathbf{I}_b:=\left\{(n,\nu)\in \mathbf{I}:\ \nu >\frac{\gamma}{|\omega_{I,\infty}|}\right\}
	\end{equation*}
	are used often since many of the quantities appearing include the factor $e^{(\nu\omega_I +\gamma)T}$, which behaves differently depending on whether $(n,\nu)\in \mathbf{I}_a$ or $(n,\nu)\in \mathbf{I}_b$. Due to assumption \ref{ass:B-omIfarFromGamma} we get $\mathbf{I}=\mathbf{I}_a\cup\mathbf{I}_b$. Clearly, the set $\mathbf{I}_a$ is finite, while $\mathbf{I}_b$ is infinite.
	
	\blem\label{L:ptsp-belowgam}
	Let $\omega_0$ and $\omega_0^\infty$ be connected  by Lemma \ref{lem:evals_lorentz_close}. Assume \ref{ass:B-signOfOmega}, \ref{ass:B-omIfarFromGamma}, \ref{ass:B-nonResonance}, and \ref{ass:B-formOfT}.  Then there exists $j_1\in \N_{\mathrm{odd}}$ such that
	$$\omega_0^{(n,\nu)}\notin \sigma_p(\cL_{nk}^T)\setminus \Omega_0 \quad \forall (n,\nu)\in \mathbf{I}\setminus  \{(1,1),(-1,1)\}  $$
	if $j\geq j_1$.
	\elem
	\bpf
	First we choose $j$ in \ref{ass:B-formOfT} so large ($j\geq j_0$ with some $j_0\in \N$) that \ref{ass:B-signOfOmega}
	and \ref{ass:B-omIfarFromGamma} hold with $\omega_0$ instead of $\omega_0^\infty$, i.e., $j$ so large, that $T_j>T_0$ from Lemma \ref{L:j0}. 
	
	Let
	$$d:=\dist\left(\{n\omega_{R,\infty}+\ri \nu\omega_{I,\infty}:(n,\nu)\in \mathbf{I} \}, \{\Imag(\omega)=-\gamma\}\right).$$
	Clearly $d=\min_{\nu \in\N}|\nu \omega_{I,\infty}+\gamma|$. By \ref{ass:B-omIfarFromGamma} we know that $d>0$. By Lemma \ref{lem:evals_lorentz_close} there is $j_d\in \N_{\mathrm{odd}}$ such that if $T=T_j$ with $j\geq j_d$, then 
	$$\dist\left(\{\omega_{0}^{(n,\nu)}:(n,\nu)\in \mathbf{I} \}, \{\Imag(\omega)=-\gamma\}\right)>\frac d2.$$
	Hence, for $T=T_j$ with $j\geq j_d$ we have $\omega_{0}^{(n,\nu)} \notin \{\Imag(\omega)\in[-\gamma-\tfrac{d}{2},-\gamma]\}$ and choosing in Lemma \ref{L:evals-deep-down} $J=J(l=\frac{\omega_R}{|\omega_I|}+1,\delta=\tfrac d2)$, we get $\omega_{0}^{(n,\nu)} \notin \sigma_p(\cL_{nk}^T)\cap \{\Imag(\omega)\leq -\gamma-\tfrac d2\}\setminus\Omega_0$ for $j\geq J$. This shows that $\omega_0^{(n,\nu)} \notin \sigma_p(\cL_{nk}^T)\setminus \Omega_0$ for all $(n,\nu) \in \mathbf{I}_b$ if $T=T_j, j \geq \max\{j_0, j_d, J\}$.
	
	Let us now show  that 
	\beq\label{E:pt-spec-above-gam}
	\omega_{0}^{(n,\nu)} \notin \sigma_p(\cL_{nk}^T)\setminus\Omega_0 \ \forall (n,\nu)\in\mathbf{I}_a \setminus \{(1,1), (-1,1)\}
	\eeq 
	for $j$ large enough.
	
	Recall the definition 
	\[\begin{split}
		G_n(\omega,T):= \left((\perm_0^{-1}n^2k^2-\omega^2\boldsymbol{\mu}_0)\perm_+ +n^2k^2\right)\left( \omega^2 +2\ri \gamma\omega - \omega_*^2\right) - c_L(n^2k^2-\omega^2\boldsymbol{\mu}_0\perm_+) \\
		-c_L e^{(\ri\omega-\gamma )T}(n^2k^2-\omega^2\boldsymbol{\mu}_0\perm_+)\left( \frac{\ri\omega -\gamma}{c_*}\sin (c_*T)-\cos(c_*T)\right)
	\end{split}\]
	from \eqref{E:dispRel-Lorenz}, which is a reformulation of the dispersion relation of $\cL_{nk}^T$, i.e., $\omega_0^{(n,\nu)}\in \sigma_p(\cL_{nk}^T)\setminus \Omega_0$ if and only if $G_n(\omega_0^{(n,\nu)},T)=0$. Further, recall that for $\omega\in \{\Imag (\omega)>-\gamma\}$
	$$G^\infty_n(\omega ):=\lim_{T\to\infty}G_n(\omega ,T)=\left((\perm_0^{-1}n^2k^2-\omega^2\boldsymbol{\mu}_0)\perm_+ +n^2k^2\right)\left( \omega^2 +2\ri \gamma\omega - \omega_*^2\right) - c_L(n^2k^2-\omega^2\boldsymbol{\mu}_0\perm_+)=0$$
	is the dispersion relation for $\cL^\infty_{nk}$. From \ref{ass:B-nonResonance} we obtain $G^\infty_n\left(n\omega_{R,\infty} +\ri\nu\omega_{I,\infty}\right)\neq 0$ for all $(n,\nu)\in \mathbf{I}_a\setminus \{(1,1),(-1,1)\}$.
	
	Fix now some $(n,\nu)\in \mathbf{I}_a\setminus \{(1,1),(-1,1)\}$. By continuity of $G^\infty_n$ there is $R>0$ such that $G^\infty_n(\omega)\neq 0$ for all $\omega\in \overline{B_R\left(\omega_\infty^{(n,\nu)}\right)}$. On this compact set we also have $G_n(\omega,T_j)\neq 0$ for all $j>j'(n,\nu)$. Then by Lemma \ref{lem:evals_lorentz_close}, after possibly making $j'(n,\nu)$ larger, we have $\omega_0^{n,\nu}\in B_R\left(\omega_\infty^{(n,\nu)}\right)$ for all $j>j'(n,\nu)$.
	Using the finiteness of the set 
	$\mathbf{I}_a$, we get \eqref{E:pt-spec-above-gam} for $j\geq J'$, where
	$$J':=\max\left\{j'(n,\nu ):~(n,\nu)\in\mathbf{I}_a\setminus \{(1,1),(-1,1)\} \right\}.$$
	
	Finally, we choose $j_1:=\max\{j_0,j_d,J,J'\}$.
	\epf
	
	\blem\label{L:j1-phi-beta}
	Assume \ref{ass:B-signOfOmega}, \ref{ass:B-omIfarFromGamma}, \ref{ass:B-betatil-phi-inf}, and \ref{ass:B-formOfT} and let  $\omega_0=\omega_R+\ri\omega_I$ be the eigenvalue for the truncation $T=T_j, j \in \N,$ on the curve through $\omega_0^\infty$ given by Lemma \ref{lem:evals_lorentz_close}.
	Define
	\[\tilde{\beta}(n,\nu):=2n\omega_R\left(\gamma\left(n^2\omega_R^2+\nu^2\omega_I^2\right)+\omega_*^2\nu\omega_I\right)\]
	and
	\[\begin{split}     \varphi(n,\nu)&:=-\tilde{\beta}(n,\nu) -2\frac{n\omega_{R}\nu\omega_{I}}{c_L}\left[(c_*^2+(\gamma+\nu\omega_{I})^2-n^2\omega_{R}^2)^2+4n^2\omega_{R}^2(\ga+\nu\omega_{I})^2\right].
	\end{split}\]
	There exists $j_2\in \N_{\mathrm{odd}}$ such that if $j\geq j_2$, then
	\beq\label{E:phi-beta}
	|\tilde{\beta}(n,\nu)|< |\varphi(n,\nu)|e^{(\nu\omega_I+\gamma)T_j} \quad \forall (n,\nu)\in \{(n,\nu)\in \mathbf{I}_a:n\neq 0\}
	\eeq
	and
	\beq\label{E:phi-beta2}
	|\tilde{\beta}(n,\nu)|> |\varphi(n,\nu)|e^{(\nu\omega_I+\gamma)T_j} \quad \forall (n,\nu)\in \{(n,\nu)\in \mathbf{I}_b: n\neq 0\}.
	\eeq
	\elem
	\bpf 	Again, we first employ Lemma \ref{L:j0} to choose $j$ in \ref{ass:B-formOfT} so large ($j\geq j_0$ with some $j_0\in \N$) that \ref{ass:B-signOfOmega}
		and \ref{ass:B-omIfarFromGamma} hold with $\omega_0$ instead of $\omega_0^\infty$.  
		
		Note that $\tilde{\beta}$ as well as $\varphi$, in fact, depend on $j$. This is because $\omega_R$ and $\omega_I$ depend on $T_j$.  We suppress this dependence in our notation. In the proof we use a short notation for the index sets in \eqref{E:phi-beta} and \eqref{E:phi-beta2},
		$$
		\mathbf{I}_a^{(0)}:=\{(n,\nu)\in \mathbf{I}_a: n\neq 0\} \quad \text{and} \quad \mathbf{I}_b^{(0)}:=\{(n,\nu)\in \mathbf{I}_b: n\neq 0\}.
		$$
	 We show that if $j$ is large enough, then $\varphi$ is bounded away from zero on $\mathbf{I}_a^{(0)}$ and $\tilde{\beta}$ is bounded away from zero on $\mathbf{I}_b^{(0)}$. For that we use $\tilde{\beta}_\infty$ and $\varphi_\infty$ defined in  assumptions \ref{ass:B-muMin} and \ref{ass:B-betatil-phi-inf}. These are the same as $\tilde{\beta}$ and $\varphi$ but with $(\omega_R,\omega_I)$ replaced by $(\omega_{R,\infty},\omega_{I,\infty})$. By these assumptions and the large $\nu$ asymptotics discussed in Remark \ref{R:beta-phi-inf-asymp}, there is $\delta>0$ such that  
		\beq\label{E:beta-phi-delta}
		|\varphi_\infty(n,\nu)|>\delta \quad \forall (n,\nu)\in \mathbf{I}_a^{(0)} \quad \text{and} \quad		
		|\tilde{\beta}_\infty(n,\nu)| >\delta \quad \forall (n,\nu)\in \mathbf{I}_b^{(0)}.  		
		\eeq
		Since $\tilde{\beta}(n,\nu)$ is a polynomial in $(n,\nu)$ with coefficients that depend continuously on $(\omega_R,\omega_I)$, we deduce from $(\omega_R,\omega_I)\to (\omega_{R,\infty},\omega_{I,\infty})$ as $j\to \infty$ that the quadratic growth of $\tilde{\beta}(n,\cdot)$ is independent of $j$. In detail, there are $c>0,\nu_0\in\N$ such that 
		$$|\tilde{\beta}(n,\nu)|>c\nu^2>\delta/2 ~~\text{for all } j>j_0,\nu>\nu_0,$$
		after possibly increasing $j_0$. For the remaining finitely many values $(n,\nu)\in \mathbf{I}_b^{(0)}, \nu \leq \nu_0$ the continuity of the mapping $(\omega_R,\omega_I)\mapsto \tilde{\beta}(n,\nu)$ and \eqref{E:beta-phi-delta} imply, after increasing (if necessary) $j_0$ even further, that for all $j\geq j_0$ 
		$$|\tilde{\beta}(n,\nu)|>\delta/2 \quad \forall (n,\nu)\in \mathbf{I}_b^{(0)}.$$
		For $\varphi(n,\nu)$ on the finite set $\mathbf{I}_a^{(0)}$ we use again 
the continuity of $(\omega_R,\omega_I)\mapsto \varphi(n,\nu)$ and
\eqref{E:beta-phi-delta} to conclude that for all $j\geq j_0$  (with a possibly larger $j_0$) 
		$$|\varphi(n,\nu)|>\delta/2 \quad \forall (n,\nu)\in \mathbf{I}_a^{(0)}.$$

		Let us now prove \eqref{E:phi-beta}. For each $(n,\nu)\in \mathbf{I}_a^{(0)}$ we define
		$$
		j'(n,\nu):=\min\{\tilde{j}\in \N_{\mathrm{odd}}: \tilde{j} \geq j_0, |\tilde{\beta}(n,\nu)|< |\varphi(n,\nu)|e^{(\nu\omega_I+\gamma)T_j} \quad \forall j\geq \tilde{j}\},$$
		where the index set being minimized is non-empty because $|\tilde{\beta}(n,\nu)|$ is bounded in $j$,  $e^{(\nu\omega_I+\gamma)T_j}\to \infty$ as $j\to\infty$ and  $|\varphi(n,\nu)|>\delta/2$ if $(n,\nu)\in \mathbf{I}_a^{(0)}$. 
		Finally, we set
		$$j_a:=\max_{(n,\nu)\in \mathbf{I}_a^{(0)}} j'(n,\nu).$$
		Hence, \eqref{E:phi-beta} holds if $j_2\geq j_a$.

		For the proof of \eqref{E:phi-beta2} we define 
		$$j''(n,\nu):=\min \{\tilde{j}\in \N_{\mathrm{odd}}: \tilde{j}\geq j_0, \ |\tilde{\beta}(n,\nu)|> |\varphi(n,\nu)|e^{(\nu\omega_I+\gamma)T_j} \quad \forall j\geq \tilde{j}\}$$
		and
		$$j_b:=\sup_{(n,\nu)\in \mathbf{I}_b^{(0)}} j''(n,\nu).$$
		We have $j_b<\infty$ because $j''$ is bounded in $(n,\nu)$, which follows from the exponential decay in $\nu$ on the right hand side and the growth on the left hand side of the inequality in the definition of $j''$. Hence, \eqref{E:phi-beta2} holds if $j_2\geq j_b$.

		The index $j_2:=\max\{j_a,j_b\}$ satisfies the lemma.
	\epf
	
	\begin{lemma}\label{bds_V_pm_L}
		Assume \ref{ass:B-signOfOmega}--\ref{ass:B-Vmin}, and \ref{ass:B-formOfT} and let  $\omega_0=\omega_R+\ri\omega_I$ be the eigenvalue for the truncation $T=T_j$ on the curve through $\omega_0^\infty$ given by Lemma \ref{lem:evals_lorentz_close}.
		Then there exist constants $j_3\in\mathbb{N}_\mathrm{odd}$ and $c_1$, $c_2$, $c_3$, $c_4>0$ such that if $j\geq j_3$, then 
		\beq\label{E:est-Vm.nu-small}
		c_1\leq |V_-(n,\nu)|\leq c_2
		\eeq
		holds for  all $(n,\nu) \in \mathbf{I}_a$ and
		\beq\label{est-abs_V_min_L}
		c_3\nu^{-1}e^{-(\nu\omega_I+\gamma)T_j}\leq\left|V_-(n,\nu)\right|\leq c_4\nu^{-1}e^{-(\nu\omega_I+\gamma)T_j}
		\eeq
		for all $(n,\nu) \in \mathbf{I}_b$.
	\end{lemma}
	
	\bpf Just like in the proof of Lemma \ref{L:j1-phi-beta} note that $V_-$ depends on $j$ since $\omega_R$ and $\omega_I$ depend on $T_j$. Also here we suppress this dependence.
	
	Using Lemma \ref{L:j0} we can choose $j$ in \ref{ass:B-formOfT} so large ($j\geq j_0$ with some $j_0\in \N$) that \ref{ass:B-signOfOmega}--\ref{ass:B-Vmin} hold with $\omega_0$ instead of $\omega_0^\infty$.  
	
	We recall that $V_-(n,\nu):=-(n\omega_R+\ri\nu\omega_I)\boldsymbol{\mu}_0\perm_0\left(1+\hat{\chi}^{(1)}_{L,T}(n\omega_R+\ri\nu\omega_I)\right)$, where $\hat{\chi}^{(1)}_{L,T}$ is given in \eqref{E:Lorentz-T-FT}. Because of assumption \ref{ass:B-formOfT} we have $\sin(c_*T)=0$ and $\cos(c_*T)=-1$ resulting in 
	$$\hat{\chi}^{(1)}_{L,T}(\omega)=-\frac{c_L}{\omega^2+2\ri\gamma\omega-\omega_*^2}\left[1+e^{(\ri \omega-\gamma )T}\right]$$
	and
	$$\hat{\chi}^{(1)}_{L,T}(n\omega_R+\ri\nu\omega_I)=c_L\frac{c_*^2+(\gamma+\nu\omega_I)^2-n^2\omega_R^2+2\ri n\omega_R(\gamma+\nu\omega_I)}{\left(c_*^2+(\gamma+\nu\omega_I)^2-n^2\omega_R^2\right)^2+4n^2\omega_R^2(\ga+\nu\omega_I)^2} \left[1+e^{-(\nu\omega_I+\gamma )T}e^{\ri n\omega_R T}\right].$$
	
	A simple calculation now yields
	\beq\label{Vmin}
	V_-(n,\nu)=-c_0^{-2}(n\omega_R+\ri\nu\omega_I)-\frac{c_0^{-2}c_L\left(\alpha +\ri\beta \right)\left(1+e^{-(\nu\omega_I+\gamma)T_j}e^{\ri n\omega_RT_j}\right)
	}{\left(c_*^2+(\gamma+\nu\omega_I)^2-n^2\omega_R^2\right)^2+4n^2\omega_R^2(\ga+\nu\omega_I)^2},
	\eeq
	where
	$$
	\begin{aligned}
		\alpha(n,\nu)&:=n\omega_R\left(\omega_*^2-n^2\omega_R^2-\nu^2\omega_I^2\right),\quad\beta(n,\nu):=\left(n^2\omega_R^2+\nu^2\omega_I^2\right)\left(2\gamma+\nu\omega_I\right)+\omega_{*}^2\nu\omega_I.	
	\end{aligned}
	$$
	
	As in the statement of the lemma, we distinguish the two cases $(n,\nu)\in \mathbf{I}_a$ or $(n,\nu)\in \mathbf{I}_b$.
	
	Note again that $\mathbf{I}_a$ has only finitely many elements. For each $(n,\nu)\in\mathbf{I}_a$ we have $e^{-(\nu\omega_I+\gamma)T_j} \to 0 \  (j \to \infty)$ and $V_-(n,\nu)\to V_{-,\infty}^{(a)}(n,\nu)$ as $j\to \infty$ with $ V_{-,\infty}^{(a)}$ given in assumption \ref{ass:B-Vmin}. This assumption and the above convergence imply that there is $j'\in \N_{\mathrm{odd}}$ such that  $V_-(n,\nu)\neq 0$ for all $(n,\nu)\in \mathbf{I}_a$ and $j\geq j'$. This proves the first claim of the lemma, i.e., \eqref{E:est-Vm.nu-small}.
	
	For each $(n,\nu)\in\mathbf{I}_b$ it follows that $ e^{-(\nu\omega_I+\gamma)T_j}\to \infty \ (j\to\infty)$
	and hence
	\beq\label{as-Vminus}
	V_-(n,\nu)\sim V_{-,\infty}^{(b)}(n,\nu)\  (j\to\infty),~~\text{uniformly in }(n,\nu),
	\eeq
	where $$V_{-,\infty}^{(b)}(n,\nu):=-c_0^{-2}c_Le^{-(\nu\omega_I+\ga)T_j}\frac{(\alpha+\ri \beta) e^{\ri n\omega_RT_j}}{\left(c_*^2+(\gamma+\nu\omega_I)^2-n^2\omega_R^2\right)^2+4n^2\omega_R^2(\ga+\nu\omega_I)^2}.
	$$
	The uniformity of the equivalence follows from the fact, that 
	$$\frac{V_-(n,\nu)}{V_{-,\infty}^{(b)}(n,\nu)}=1 + r(n,\nu)e^{-\ri n\omega_RT_j}e^{(\nu\omega_I +\ga)T_j}$$
	where $r$ is a rational function with the denominator given by $\alpha + \ri \beta$, which is nonzero as we show next, such that $r$ grows at most polynomially in $\nu$ on $\mathbf{I}_b$. To prove that the denominator does not vanish, we proceed by contradiction and set  $\alpha(n,\nu)=\beta(n,\nu)=0$. Because $\omega_R\neq 0$ by \ref{ass:B-signOfOmega}, the equality $\alpha(n,\nu)=0$ holds if and only if one of the following two cases holds
		$$\text{(i)} \ n=0, \quad \text{(ii)} \ \omega_*^2=n^2\omega_R^2+\nu^2\omega_I^2.$$
		Case (i) implies $\beta = \nu\omega_I(\nu^2\omega_I^2 +2\gamma\nu\omega_I+\omega_*^2)$. Because $\omega_I\neq 0$, the condition $\beta=0$ is equivalent to
		$\nu\omega_I=-\gamma\pm\sqrt{\gamma^2-\omega_*^2}=-\gamma\pm \ri c_*$, which is a contradiction because $\nu\omega_I\in \R$. Note that $\gamma <\omega_*$ by assumptions of Theorem \ref{T:Lor-example}. Case (ii) results in $\beta = 2\omega_*^2(\gamma+\nu\omega_I)$, which does not vanish due to assumption \ref{ass:B-omIfarFromGamma} and because $\omega_*>0$.
	
	Our strategy is to first show \eqref{est-abs_V_min_L} for all $\nu$ large enough (i.e. $\nu \geq \nu_0$ with some $\nu_0\in \N$) and then prove that  $|V_-(n,\nu)|>\delta e^{-(\nu\omega_I+\ga)T_j}$ with $\delta>0$ for all $\gamma/|\omega_I|<\nu <\nu_0$ and all $j$ large enough.
	
	We see that $|V_{-,\infty}^{(b)}(n,\nu)|=c e^{-(\nu\omega_I+\ga)T_j} \frac{|\alpha(n,\nu)+\ri \beta(n,\nu)|}{d(n,\nu)}$, where $c>0$ and 
	\beq\label{E:d}
	d(n,\nu):= \left(c_*^2+(\gamma+\nu\omega_I)^2-n^2\omega_R^2\right)^2+4n^2\omega_R^2(\ga+\nu\omega_I)^2.
	\eeq
	In the following $c$ and $c'$ play the role of generic positive constants independent of $n,\nu$, and $j$ and their values change during the calculation. We have
	\beq\label{al-be-est}
	c'\nu^3\geq |\alpha+\ri\beta|\geq c (|\alpha|+|\beta|)\geq c(|n|\nu^2+\nu^3)\geq c\nu^3
	\eeq
	for all $(n,\nu)$ and $j$ with $\nu$ and $j$ large enough. For $d$ we have
	$$d(n,\nu)\sim (\nu^2\omega_I^2-n^2\omega_R^2)^2+4n^2\omega_R^2\nu^2\omega_I^2 =  (\nu^2\omega_I^2+n^2\omega_R^2)^2 \quad (\nu \to \infty)$$
	and  hence we get
	\beq\label{d-est}
	c\nu^4\leq d(n,\nu)\leq c'\nu^4
	\eeq
	for all $(n,\nu)$ and $j$  with $\nu$ and $j$ large enough. 
	
	Equations \eqref{al-be-est} and \eqref{d-est} prove that \eqref{est-abs_V_min_L} holds for all $(n,\nu)\in \mathbf{I}$ with $\nu\geq \nu_0$ and $j\in \N_{\mathrm{odd}}, j\geq j_3$ with some $\nu_0,j_3\in \N$. To have \eqref{est-abs_V_min_L} for all $(n,\nu)\in \mathbf{I}_b$ and $j\in \N_{\mathrm{odd}}$ large enough (with $c_1,\dots,c_4$ independent of $j$), we need to show that  $|V_-(n,\nu)|>\delta e^{-(\nu\omega_I+\ga)T_j}>0$ with $\delta>0$ for all $(n,\nu)\in \mathbf{I}_b$ and for all $j$ large enough. This will follow from $|V_{-,\infty}^{(b)}(n,\nu)|>\delta' e^{-(\nu\omega_I+\ga)T_j}$ with $\delta'>0$ and from the asymptotics \eqref{as-Vminus}. 
	
	We have $V_{-,\infty}^{(b)}(n,\nu)=f(n,\nu)e^{-(\nu\omega_I+\ga)T_j}e^{\ri n\omega_R T_j}$ with a rational function $f$. Hence it suffices to show that $f(n,\nu)\neq 0$ for all $(n,\nu)\in \mathbf{I}$ with $\gamma/|\omega_I|<\nu <\nu_0$. 
	We observe that $f(n,\nu)=0$ if and only if $\alpha(n,\nu)=\beta(n,\nu)=0$, which was disproved above.
	
	Note that, in fact, $f$ depends also on $j$ (via the dependence of $\omega_R$ and $\omega_I$ on $j$).  However, the above argument works also with $(\omega_{R,\infty},\omega_{I,\infty})$ replacing $(\omega_R,\omega_I)$ in $f$, so the convergence $\omega_0\to \omega_\infty, (j\to\infty),$ and the continuity of $f$ yield a constant $\delta'>0$ independent of $j$, for which
	$$|V_{-,\infty}^{(b)}(n,\nu)|>\delta' e^{-(\nu\omega_I+\ga)T_j}$$
	for all $j\in \N_{\mathrm{odd}}$ large enough and all $(n,\nu)\in \mathbf{I}_b$.
	
	This completes the proof of  \eqref{est-abs_V_min_L}, where the value of $j_3$ follows from the above arguments.
	\epf

	\blem\label{bd_abs_mu_minu_L}
	Assume \ref{ass:B-signOfOmega}, \ref{ass:B-omIfarFromGamma},  \ref{ass:B-muMin}, and \ref{ass:B-formOfT} and let $\omega_0=\omega_R+\ri\omega_I$ be the eigenvalue for the truncation $T=T_j$ on the curve through $\omega_0^\infty$ given by Lemma \ref{lem:evals_lorentz_close}. Then there are constants $j_4\in \mathbb{N}$ and $c_1,...,c_4>0$ such that if $j\geq j_4$, then
	\begin{equation}\label{est_abs_mu_min_L_sm_nu}
		c_1 \leq |\mu_-(n,\nu)| \leq c_2
	\end{equation}
	holds for each $(n,\nu) \in \mathbf{I}_a$ and
	\begin{equation}\label{est_abs_mu_min_L_lg_nu}
		c_3 e^{-(\nu\omega_I+\gamma)T_j/2} \leq |\mu_-(n,\nu)| \leq c_4 e^{-(\nu\omega_I+\gamma)T_j/2}	
	\end{equation}
	for all $(n,\nu) \in \mathbf{I}_b$.
	\elem
	\bpf Just like in the above proofs, also here $\mu_-$ depends on $j$ since $\omega_R$ and $\omega_I$ depend on $T_j$. This dependence is suppressed.
	
	Again, we first use Lemma \ref{L:j0} to choose  $j$ in \ref{ass:B-formOfT} so large ($j\geq j_0$ with some $j_0\in \N$) that \ref{ass:B-signOfOmega}, \ref{ass:B-omIfarFromGamma},  \ref{ass:B-muMin} hold with $\omega_0$ instead of $\omega_0^\infty$.

	Recall that $\mu_-(n,\nu)=\sqrt{n^2k^2+\omega_0^{(n,\nu)}V_-(n,\nu) }$ and $V_-$ is given by \eqref{Vmin}. We get
	\beq\label{musq}
	\mu_-^2(n,\nu)=n^2k^2-c_0^{-2}\left(n^2\omega_R^2-\nu^2\omega_I^2+2n\omega_R\nu\omega_I\ri+\frac{c_L(\tilde{\alpha}(n,\nu)+\ri\tilde{\beta}(n,\nu))\left(1+e^{-(\nu\omega_I+\ga)T_j}e^{\ri n\omega_RT_j}\right)}{d(n,\nu)}\right),
	\eeq
	where
	\beq\label{al-be-til}
	\begin{aligned}
		\tilde{\alpha}(n,\nu)&:=-\left(n^2\omega_R^2+\nu^2\omega_I^2\right)\left(n^2\omega_R^2+\nu^2\omega_I^2+2\gamma\nu\omega_I\right)-\omega_*^2\left(\nu^2\omega_I^2-n^2\omega_R^2\right),\\ \tilde{\beta}(n,\nu)&:=2n\omega_R\left(\gamma\left(n^2\omega_R^2+\nu^2\omega_I^2\right)+\omega_*^2\nu\omega_I\right),
	\end{aligned}
	\eeq
	and $d$ was defined in \eqref{E:d}.
	Note also that $\tilde{\beta}$ was defined already in Lemma \ref{L:j1-phi-beta}. 
	
	We show first \eqref{est_abs_mu_min_L_sm_nu} for $(n,\nu)\in\mathbf{I}_a$, in which case we have $e^{-(\nu\omega_I+\gamma)T_j} \to 0 \  (j \to \infty)$ and
	$$\mu_-^2(n,\nu)\to \mu_{-,\infty}^2(n,\nu) \quad (j\to\infty),$$
	where $\mu_{-,\infty}^2$ was defined in assumption \ref{ass:B-muMin}. By the same assumption we have $\mu_{-,\infty}(n,\nu)\neq 0$ for all $(n,\nu)\in\mathbf{I}_a$. This proves \eqref{est_abs_mu_min_L_sm_nu} for all $j\in \N_{\mathrm{odd}}, j\geq j'$ with some $j'\in \N$ as $\mathbf{I}_a$ is finite.
	
	Now, for any $(n,\nu)\in\mathbf{I}_b$ we have  $e^{-(\nu\omega_I+\gamma)T_j} \to \infty \  (j \to \infty)$
	and, analogously to \eqref{as-Vminus},
	\begin{equation*}\label{E:mu-min-as}
		\mu_-^2(n,\nu)\sim \frac{c_0^{-2}c_Le^{-(\nu\omega_I+\ga)T_j}(\tilde{\alpha}(n,\nu)+\ri \tilde{\beta}(n,\nu))e^{\ri n\omega_RT_j}}{d(n,\nu)} \quad (j\to\infty)\quad \text{ uniformly in }(n,\nu).
	\end{equation*}
	In the proof of Lemma \ref{bds_V_pm_L} we showed in 	\eqref{d-est} that $c\nu^4\leq d(n,\nu)\leq c'\nu^4$ for all $\nu$ large enough. And because $d>0$, we obtain \eqref{d-est} for all $(n,\nu)\in \mathbf{I}$.
	
	Next, we estimate $|\tilde{\alpha}+\ri \tilde{\beta}|$ from above and below. Clearly,
	\beq \label{E:alpha-beta-from-above}|\tilde{\alpha}(n,\nu)|\leq c\nu^4, \ |\tilde{\beta}(n,\nu)|\leq c\nu^3 \quad \text{for all} \ (n,\nu)\in \mathbf{I},\eeq
	which implies
	\beq\label{E:al-be-upper}
	|\tilde{\alpha}(n,\nu)+\ri \tilde{\beta}(n,\nu)|\leq c\nu^4  \quad \text{for all} \ (n,\nu)\in \mathbf{I}.
	\eeq
	
	For the lower bound we first observe that 
	\beq\label{E:alpha-from-below}|\tilde{\alpha}(n,\nu)|\geq c\nu^4 \quad \text{for all} \ (n,\nu)\in \mathbf{I} \ \text{with $\nu$ large enough}.\eeq
	This produces $c\nu^4$ as a lower bound of $|\tilde{\alpha}+\ri \tilde{\beta}|$ for $\nu$ large enough. To get such a lower bound for all $(n,\nu)\in\mathbf{I}_b$, we need to show that $\tilde{\alpha}(n,\nu)+\ri \tilde{\beta}(n,\nu)\neq 0$ for all $(n,\nu)\in \mathbf{I}_b$. 
	
	Of course, $\tilde{\alpha}+\ri\tilde{\beta}=0$ means $\tilde{\alpha}=\tilde{\beta}=0$. If $\tilde{\alpha}(n,\nu)=0$, then \[n^2\omega_R^2+\nu^2\omega_I^2=-\frac{\omega_*^2(\nu^2\omega_I^2-n^2\omega_R^2)}{n^2\omega_R^2+\nu^2\omega_I^2+2\ga\nu\omega_I}\]
	and
	\[\tilde{\beta}=2n\omega_R\omega_*^2\left(\nu\omega_I-\frac{\ga(\nu^2\omega_I^2-n^2\omega_R^2)}{n^2\omega_R^2+\nu^2\omega_I^2+2\ga\nu\omega_I}\right).\]
	Hence, $\tilde{\alpha}=\tilde{\beta} =0$ implies that one of the following three cases holds
	$$ \text{(i)}\, \omega_R=0,\qquad \text{(ii)}\, n=0,\qquad \text{(iii)}\, \nu\omega_I(n^2\omega_R^2+\nu^2\omega_I^2+2\ga\nu\omega_I)=\ga(\nu^2\omega_I^2-n^2\omega_R^2).$$
	
	Case (i) immediately contradicts \ref{ass:B-signOfOmega} and case (iii) is equivalent to 
	$\omega_I=-\frac{\gamma}{\nu}$
	and hence contradicts \ref{ass:B-omIfarFromGamma}. For $n=0$, the equation $\tilde{\alpha}=0$ reads 
	$$\nu^2\omega_I^2 +2\gamma\nu\omega_I=-\omega_*^2$$
	or equivalently $(\nu\omega_I+\gamma)^2=-c_*^2<0$, such that case (ii) is also impossible.
	
	We conclude
	\beq\label{E:al-be-lower}
	|\tilde{\alpha}(n,\nu)+\ri \tilde{\beta}(n,\nu)|\geq c\nu^4  \quad \text{for all} \ (n,\nu)\in \mathbf{I}_b.
	\eeq
	
	Estimate \eqref{d-est} for all $(n,\nu)\in \mathbf{I}$ and estimates  \eqref{E:al-be-upper} and \eqref{E:al-be-lower} imply \eqref{est_abs_mu_min_L_lg_nu}. This completes the proof.	The value of $j_4$ is given by the above arguments.\epf
	
	\blem\label{bd_re_mu_minu_L}
	Assume  \ref{ass:B-signOfOmega}, \ref{ass:B-omIfarFromGamma},  \ref{ass:B-muMin}, \ref{ass:B-betatil-phi-inf}, \ref{ass:B-formOfT},  and \ref{ass:B-rationality}. There are constants $c_1,c_2>0$ such that
	\begin{equation}\label{est_re_mu_minu_L_large_re_nu}
		c_1 e^{-(\nu\omega_I+\gamma)T/2} \leq \Real\mu_- (n,\nu)\leq c_2 e^{-(\nu\omega_I+\gamma)T/2}
	\end{equation}
	for each $(n,\nu) \in \mathbf{I}$.
	\elem
	\brem
	We want to emphasize that, according to assumption \ref{ass:B-rationality}, in Lemma \ref{bd_re_mu_minu_L} the truncation is given by $T=T_{j_*}$ and the eigenvalue $\omega_0=\omega_R+\ri\omega_I$ is for this truncation. Recall also that $\mu_-$ depends on $\omega_R$ and $\omega_I$.
	
	Also note that because $T$ is fixed in Lemma \ref{bd_re_mu_minu_L}, we do not need to distinguish between the two cases $(n,\nu)\in \mathbf{I}_a$ and $(n,\nu)\in \mathbf{I}_b$ like in Lemmas \ref{L:j1-phi-beta}-\ref{bd_abs_mu_minu_L}.
	It is not necessary to isolate the case $\nu<\frac{\ga}{|\omega_I|}$.	
	\erem
	\begin{proof}
		First, note that the upper bound in \eqref{est_re_mu_minu_L_large_re_nu} holds using Lemma \ref{bd_abs_mu_minu_L} and the fact $\Real\mu_-\leq |\mu_-|$. Also, $\Real \mu_-\geq 0$ follows automatically - from the definition of the square root.
		
		For the lower bounds we use $\mu_-^2$, given in \eqref{musq}. Due to assumption \ref{ass:B-rationality} we have, for some $m\in\Z$,
		$$n\omega_R T = n \omega_Rj_*\frac{\pi}{c_*}=2nm\pi\in \pi \Z_{\mathrm{even}}$$
		resulting in $e^{\ri n\omega_RT}=1$ for all $n\in \Z$ and we obtain the largely simplified form of $\mu_-^2$
		$$
		\begin{aligned}
			\Real\mu_-^2(n,\nu)&=n^2k^2-c_0^{-2}\left(n^2\omega_R^2-\nu^2\omega_I^2+\frac{c_L\tilde{\alpha}(n,\nu)\left(1+e^{-(\nu\omega_I+\ga)T}\right)}{d(n,\nu)}\right),\\
			\Imag\mu_-^2(n,\nu)&=-c_0^{-2}\left(2n\omega_R\nu\omega_I+\frac{c_L\tilde{\beta}(n,\nu)\left(1+e^{-(\nu\omega_I+\ga)T}\right)}{d(n,\nu)}\right)	
		\end{aligned}
		$$
		with $\tilde{\alpha}, \tilde{\beta}$, and $d$ defined in \eqref{al-be-til} and \eqref{E:d}.

		Our strategy is first to show that there is $\delta>0$ such that for $\nu$ large enough ($\nu\geq \nu_0$ and $|n|\leq \nu$ with some $\nu_0\in \N$)
		\beq\label{arg-cond}
		\mathrm{Arg}(\mu^2_-(n,\nu))\in[-\pi+\delta, \pi-\delta].\eeq
		Then $\cos (\mathrm{Arg}(\mu_-))$ is bounded away from zero and we have $\Real \mu_-(n,\nu)\geq c |\mu_- (n,\nu)|$ for all $\nu \geq \nu_0$. Lemma \ref{bd_abs_mu_minu_L} then ensures
		$$
		\begin{aligned}
			\Real \mu_-(n,\nu)&\geq c e^{-(\nu\omega_I+\gamma)T/2} \quad \forall \nu \geq \max \{\nu_0,\gamma/|\omega_I|\}, |n|\leq \nu,\\
			\Real \mu_-(n,\nu)&\geq c \quad \forall \nu \in[\nu_0,\gamma/|\omega_I|) \quad \text{if} \ \nu_0<\gamma/|\omega_I|, |n|\leq\nu 
		\end{aligned}
		$$ 
		for some $c>0$ independent of $n$ and $\nu$. Note that assumption \ref{ass:B-rationality} will play a central role in proving \eqref{arg-cond}.
		
		Second, we will show that $\Real \mu_-(n,\nu)\neq 0$ for all $\nu<\nu_0$. Then the above lower bounds on $\Real \mu_-(n,\nu)$ hold with $\nu_0$ set to $\nu_0=1$, i.e., the lower bound in  \eqref{est_re_mu_minu_L_large_re_nu} will hold.
		
		We rewrite $\Real \mu_-^2$ and $\Imag \mu_-^2$ as follows.
		$$
		\begin{aligned}
			\Real \mu_-^2(n,\nu)&=-c_0^{-2}c_Le^{-(\nu\omega_I+\ga)T}\frac{\tilde{\alpha}(n,\nu)-\psi(n,\nu)e^{(\nu\omega_I+\ga)T}}{d(n,\nu)},\\
			\Imag \mu_-^2(n,\nu)&=-c_0^{-2}c_Le^{-(\nu\omega_I+\ga)T}\frac{\tilde{\beta}(n,\nu)-\varphi(n,\nu)e^{(\nu\omega_I+\ga)T}}{d(n,\nu)},         	
		\end{aligned}
		$$
		where
		$$
		\psi(n,\nu):=-\tilde{\alpha}(n,\nu)+\frac{c_0^{2}}{c_L}\left[n^2k^2-c_0^{-2}(n^2\omega_{R}^2-\nu^2\omega_{I}^2)\right]d(n,\nu)
		$$
		and
		$$
		\varphi(n,\nu)=-\tilde{\beta}(n,\nu) -2\frac{n\omega_{R}\nu\omega_{I}}{c_L}d(n,\nu)
		$$
		was defined in Lemma \ref{L:j1-phi-beta}.
		
		Let us first consider the case of large $\nu$ and note that 
		$$\Real \mu_-^2(n,\nu) \sim -c_0^{-2}c_L e^{-(\nu\omega_I+\ga)T}\frac{\tilde{\alpha}(n,\nu)}{d(n,\nu)}, \qquad \Imag \mu_-^2(n,\nu) \sim -c_0^{-2}c_L e^{-(\nu\omega_I+\ga)T}\frac{\tilde{\beta}(n,\nu)}{d(n,\nu)} \qquad (\nu\to\infty),$$
		hold uniformly in $n$. As we explained in \eqref{E:alpha-beta-from-above} and \eqref{E:alpha-from-below} in the proof of Lemma \ref{bd_abs_mu_minu_L}, there are $c,c'>0$ such that           
		$$c\nu^4\leq |\tilde{\alpha}(n,\nu)|\leq c'\nu^4, \quad |\tilde{\beta}(n,\nu)|<c\nu^3 \quad  \forall (n,\nu)\in \mathbf{I}.$$
		This implies 
		$$\lim_{\nu \to \infty}\frac{\Imag \mu_-^2(n,\nu)}{\Real \mu_-^2(n,\nu)}=0  \quad \forall n \in\Z, |n|\leq \nu.$$
		Moreover, one observes that there is $c>0$ and $\nu_0\in \N$, such that
		$$\tilde{\alpha}(n,\nu)<-c\nu^4 \quad \forall (n,\nu)\in \mathbf{I}, \nu \geq \nu_0.$$
		The negativity of $\tilde{\alpha}$ then guarantees that $\Real \mu_-^2(n,\nu)>0$ and hence $\mathrm{Arg}(\mu^2_-(n,\nu))$ stays near zero for $\nu$ large enough, i.e., we have \eqref{arg-cond} for $\nu\geq \nu_0$.
		
		It remains to treat the case of small $\nu$, i.e., to show that $\Real \mu_-(n,\nu)\neq 0$ for all $\nu<\nu_0$. Clearly, $\Real \mu_-=0$ if and only if $\Imag \mu_-^2=0 $ and $\Real \mu_-^2\leq 0$. Assuming    $\Imag \mu_-^2(n,\nu)=0$ for some $(n,\nu)\in \mathbf{I}$, we get  
		\beq\label{est:phi}
		\tilde{\beta}(n,\nu)-\varphi(n,\nu)e^{(\nu\omega_I+\ga)T}= 0.
		\eeq
		We distinguish between the cases $n=0$ and $n\neq  0$. If $n=0$, then we have
		$$
		\tilde{\alpha}(0,\nu)=-\nu^2\omega_I^2(c_*^2+(\ga+\nu\omega_I)^2)<0, \qquad \psi(0,\nu)=-\tilde{\alpha}(0,\nu)+c_L^{-1}\nu^2\omega_I^2d(0,\nu)>0
		$$
		and hence
		$$\Real\mu_-^2(0,\nu)>0,$$
		which is a contradiction.
		
		In the case $n \neq 0$ we apply \eqref{E:phi-beta} and \eqref{E:phi-beta2} from Lemma \ref{L:j1-phi-beta} (using the fact that $j_*\geq j_2$) to conclude
		$$|\tilde{\beta}(n,\nu)|\neq |\varphi(n,\nu)|e^{(\nu\omega_I+\gamma)T},$$
		which contradicts \eqref{est:phi}.
	\end{proof}
	
	\subsection{Proof of Theorem \ref{T:Lor-example}: Non-Resonance Assumption \ref{ass:A-resolvent} for the Lorentz Interface \eqref{E:perm-LT}}\mbox{}\\
	Next, we show that the estimates of $\Real(\mu_\pm)$ and $V_\pm$ imply that $\omega_0^{(n,\nu)}$ lies outside the essential spectrum and outside the singular set $\Omega_0$ for all  $(n,\nu) \in \mathbf{I}\setminus \{(1,1),(-1,1)\}$.

	\blem\label{L:essspec-inters}
	Assume \ref{ass:B-signOfOmega}--\ref{ass:B-betatil-phi-inf}, \ref{ass:B-formOfT},  and \ref{ass:B-rationality}. Then 
	$$\omega_0^{(n,\nu)} \notin \sigma_\mathrm{ess}(\cL_{nk}^T)\cup \Omega_0 \quad \text{for all} \quad  (n,\nu) \in \mathbf{I}\setminus \{(1,1),(-1,1)\}.$$
	\elem
	\bpf
	Recall that $\omega_0^{(n,\nu)}=n\omega_R+\ri\nu \omega_I$. The essential spectrum of $\cL_{nk}^T$ outside $\Omega_0$ is given by \eqref{E:ess-spec-red}. This can be reformulated as the set of $\omega\in \C$ for which $n^2k^2 - \boldsymbol{\mu}_0 \omega^2 \perm_+(\omega)$ or $n^2k^2 - \boldsymbol{\mu}_0 \omega^2 \perm_-(\omega)$ is real and non-positive if $nk\neq 0$ and strictly negative if $nk=0$.
	Our specific choice of $T=T_{j_*}$ in \ref{ass:B-rationality} ensures that $T$ is large enough so we can use Lemmas \ref{bds_V_pm} -- \ref{bd_abs_mu_minu_L}.
	
	We proved in Lemma \ref{bds_V_pm} that $n^2k^2 - \boldsymbol{\mu}_0 (\omega_0^{(n,\nu)})^2 \perm_+(\omega_0^{(n,\nu)}) \notin (-\infty,0]$ holds for all $(n,\nu)\in\mathbf{I}$, see \eqref{est-re_mu_plus}. Lemma \ref{bd_re_mu_minu_L} shows the same result for $\perm_-$ under assumptions  \ref{ass:B-signOfOmega}, \ref{ass:B-omIfarFromGamma},  \ref{ass:B-muMin}, \ref{ass:B-betatil-phi-inf}, \ref{ass:B-formOfT},  and \ref{ass:B-rationality} on $\omega_R$, $\omega_I$, and $T$.
	
	Similarly, since $\omega_0^{(n,\nu)}\neq 0$ for all $(n,\nu )\in\mathbf{I}$, we only need to study that part of the set $\Omega_0$ from \eqref{eqn:Omega0} which is given as those $\omega\in\C$ for which $V_-(\omega )=0$ or $V_+(\omega)=0$. However, in Lemmas \ref{bds_V_pm} and \ref{bds_V_pm_L}, using also assumption \ref{ass:B-Vmin}, we proved positive lower bounds for the modulus of $V_\pm (\omega_0^{(n,\nu )})$, so $\omega_0^{(n,\nu)}\notin \Omega_0$.
	\epf
	
	In summary, the above spectral results for $\cL_{nk}^T$ show that \ref{ass:A-resolvent} is satisfied by the Lorentz interface.
	
	\begin{prop}\label{resol_set}
		Consider the Lorentz interface \eqref{E:perm-LT} under assumptions  \ref{ass:B-signOfOmega}--\ref{ass:B-rationality}. 
		Then
		$$\omega_0^{(n,\nu )} \notin  \sigma(\cL_{nk}^T) \quad \text{for all} \quad (n,\nu) \in \mathbf{I}\setminus \{(1,1),(-1,1)\},$$
		i.e., assumption \ref{ass:A-resolvent} is satisfied.
	\end{prop}

	\begin{proof}
		From Lemma \ref{L:ptsp-belowgam} we get
		$$\omega_0^{(n,\nu)}\notin \sigma_p(\cL_{nk}^T)\setminus\Omega_0 \quad \forall (n,\nu)\in \mathbf{I}\setminus \{(1,1),(-1,1)\}.$$
		Additionally using Lemma \ref{L:essspec-inters} and the fact that $\sigma(\cL_{nk}^T)=\sigma_p(\cL_{nk}^T)\cup \sigma_\mathrm{ess}(\cL_{nk}^T)$, we obtain the statement.
	\end{proof}

	\subsection{Proof of Theorem \ref{T:Lor-example}: Resolvent Estimate for the Lorentz Interface \eqref{E:perm-LT}}
	\label{S:resolv}
	In this section we verify assumption \ref{ass:A-resolvEst} for the case of 
	the interface of a non-dispersive material and a Lorentz material as given by  \eqref{E:perm-LT}.
	We  prove estimates of the operator norms  $\|\cdot\|_{L^2\to L^2}$ and $\|\cdot\|_{\cH^1\to \cH^1}$ for the resolvent operator. Note that the truncation parameter $T$ in \eqref{E:perm-LT} will now be fixed at $T=T_{j_*}$, according to our assumptions \ref{ass:B-formOfT} and \ref{ass:B-rationality}. With this the parameter $T$ satisfies the conditions of Lemmas \ref{bds_V_pm}--\ref{bd_re_mu_minu_L}. Hence, we can use the resulting estimates in the subsequent proofs. In Proposition \ref{resol_set} we showed that $n\omega_R+\ri\nu \omega_I$ lies in the resolvent set. Next, in Lemma \ref{resol_L2} we produce a decay rate of the $L^2$-norm of the resolvent in terms of the index $\nu$. 
	
	Below we use the same notation $\|\cdot \|_{L^2}$ for the norm of any $\C^m$-valued function with $m=1,2,3$.
	\begin{thm}\label{resol_L2}
		Consider the operator $\cL_{nk}^T$ given in \eqref{E:Lnk-DL} and assume \ref{ass:B-signOfOmega}--\ref{ass:B-rationality}. Then a unique solution $u$ of 
		$$\cL_{nk}^T(\cdot, n\omega_R+\ri \nu\omega_I)u=r$$ 
		exists for each $r\in L^2(\R,\C^3)$ and $(n,\nu)\in \mathbf{I}\setminus \{(1,1),(-1,1)\}$. It can be decomposed into a homogeneous part $u_h$ and a particular part $u_p$ satisfying
		\beq\label{est-uh-thm}
		\begin{split}
			\|u_{h}\|_{L^2(\R_+)}
			&\leq c\left[\|(r_1,r_2)\|_{L^2(\R_-)} \nu^{\frac{1}{2}}e^{\frac{3}{4}(\nu\omega_I+\ga)T} + \|r_3\|_{L^2(\R_-)}\nu^{-\frac{1}{2}}e^{\frac{1}{4}(\nu\omega_I+\gamma)T}+\nu^{-1}\|r\|_{L^2(\R_+)}\right],\\
			\|u_{h,1}\|_{L^2(\R_-)} &\leq c\left[\|(r_1,r_2)\|_{L^2(\R_-)}\nu^{2}e^{\frac{3}{2}(\nu\omega_I+\ga)T}
			+\|r_3\|_{L^2(\R_-)}\nu e^{(\nu\omega_I+\gamma)T}+\nu^{\frac{3}{2}}e^{\frac{5}{4}(\nu\omega_I+\gamma)T}\|r\|_{L^2(\R_+)}\right],\\
			\|u_{h,2}\|_{L^2(\R_-)} &\leq c\left[\|(r_1,r_2)\|_{L^2(\R_-)}\nu e^{(\nu\omega_I+\ga)T} + \|r_3\|_{L^2(\R_-)} e^{\frac12(\nu\omega_I+\gamma)T}+\nu^{\frac{1}{2}}e^{\frac{3}{4}(\nu\omega_I+\gamma)T}\|r\|_{L^2(\R_+)}\right],\\
			\|u_{h,3}\|_{L^2(\R_-)} &\leq c\left[\|(r_1,r_2)\|_{L^2(\R_-)}e^{\frac{1}{2}(\nu\omega_I+\ga)T}+ \|r_3\|_{L^2(\R_-)}\nu^{-1} +\nu^{-\frac{1}{2}}e^{\frac{1}{4}(\nu\omega_I+\gamma)T}\|r\|_{L^2(\R_+)}\right],\\
		\end{split}
		\eeq
		and
		\beq\label{Est_up-thm}\begin{split}
			\|u_{p}\|_{L^2(\R_+)}\leq& ~c\|r\|_{L^2(\R_+)}\nu^{-1},\\
			\|u_{p,1}\|_{L^2(\R_-)}\leq& ~c\left(\|(r_1,r_2)\|_{L^2(\R_-)}\nu e^{(\nu\omega_I+\ga)T}+\|r_3\|_{L^2(\R_-)}\nu e^{(\nu\omega_I+\ga)T}\right),\\
			\|u_{p,2}\|_{L^2(\R_-)}\leq& ~c\left(\|(r_1,r_2)\|_{L^2(\R_-)}\nu e^{(\nu\omega_I+\ga)T}+\|r_3\|_{L^2(\R_-)} e^{\frac12(\nu\omega_I+\ga)T}\right),\\
			\|u_{p,3}\|_{L^2(\R_-)}\leq& ~c\left(\|(r_1,r_2)\|_{L^2(\R_-)}e^{\frac12(\nu\omega_I+\ga)T}+\|r_3\|_{L^2(\R_-)}\nu^{-1}\right)
		\end{split}\eeq
		for all $(n,\nu)\in \mathbf{I}$ with $\nu>\gamma/|\omega_I|$, where $c>0$ is a constant independent of $\nu$ and $n$.
	\end{thm}
	\bpf
	Given $r\in L^2(\R, \C^3)$, the equation $\cL_{n k}^T(x, \omega)u=r$ reads
	\begin{equation}\label{Res_sys}
		\begin{rcases}
			nk u_3- V_{\pm}(n,\nu) u_1 = r_1 \\
			\ri u_3'-  V_{\pm}(n,\nu) u_2= r_2\\
			\ri u_2' + nku_1 +\omega_0^{(n,\nu)} u_3= r_3
		\end{rcases}
		\text{ on $\R_\pm$.}
	\end{equation}
	In most of the following calculations we suppress the explicit dependence of $V_\pm$ and $\mu_\pm$ on $(n,\nu)$.
	The first equation in \eqref{Res_sys} implies
	\begin{equation}\label{u_1}u_1=\frac{1}{ V_\pm }\left( nk u_3 -r_1\right) \  {\rm ~on~} \R_\pm.\end{equation}
	Plugging \eqref{u_1} into the third equation of \eqref{Res_sys}, one obtains
	\beq\label{u23}
	\left. \begin{array}{rcl}
		u_3 ' +\ri  V_\pm u_2  &=& -\ri r_2 \\
		u_2 '  - \ri  \frac{\omega_0^{(n,\nu )}V_\pm +n^2k^2}{ V_\pm } u_3     &=& -\ri\left( r_3 +\frac{nk}{ V_\pm }r_1\right)
	\end{array} \right\} {\rm ~on~} \R_\pm.
	\eeq
	A unique solution  $(u_2,u_3)^T$ of \eqref{u23} with $u\in D(\cL_{nk}^T)$ exists because $\omega_0^{(n,\nu)}\notin \sigma(\cL_{nk}^T)$ by Proposition \ref{resol_set}. In the proof of Prop. 2.3 in \cite{DH2024} it was shown that $\omega_0^{(n,\nu)}\in \sigma_p(\cL_{nk}^T)\setminus \Omega_0$ is equivalent to $\mu_- V_+ +\mu_+ V_-=0$. This quantity is important in the estimates below. We note that we have 
	\beq\label{DR}\mu_-(n,\nu) V_+(n,\nu) +\mu_+(n,\nu) V_-(n,\nu) \neq 0 \quad \forall (n,\nu)\in \mathbf{I}\setminus\{(1,1),(-1,1)\}.\eeq
	Moreover, using the reverse triangle inequality and the bounds for $V_+,\mu_+,V_-$ and $\mu_-$ from Lemmas \ref{bds_V_pm}, \ref{bds_V_pm_L} and \ref{bd_abs_mu_minu_L}, respectively, we get the following estimate for all $\nu>\frac{\gamma}{|\omega_I|}$ with $|n|\leq\nu$ 
	\begin{align}
		|\mu_- V_+ +\mu_+ V_- |&\geq |\mu_+ V_- |-|\mu_- V_+ |\notag\\
		&\geq c_1 e^{-(\nu\omega_I+\gamma)T}-c_2\nu e^{-(\nu\omega_I+\gamma)T/2}\notag\\
		&= e^{-(\nu\omega_I+\gamma)T}\left(c_1-c_2\nu e^{(\nu\omega_I+\gamma)T/2}\right).\notag
	\end{align}
	with some constants $c_1,c_>0.$ Since we have that $\lim\limits_{\nu\to\infty} \nu e^{(\nu\omega_I +\gamma)T/2}=0$, we can choose $\nu_0\in\N$ so large, that $$\left(c_1-c_2\nu e^{(\nu\omega_I+\gamma)T/2}\right)\geq c_1/2$$ for all $\nu\geq \nu_0$. This leads to the estimate 
	\beq \label{den-est} 
	|\mu_- V_+ +\mu_+ V_- |\geq \frac{c_1}{2}\, e^{-(\nu\omega_I+\gamma)T}
	\eeq 
	for all $\nu\geq\nu_0$. Because $T$ is fixed (by \ref{ass:B-rationality}), estimate \eqref{den-est} holds for all $(n,\nu)\in \mathbf{I}\setminus \{(1,1),(-1,1)\}$ with a possibly changed value of $c_1$.
	
	Note that, by variation of constants, the solution of \eqref{u23} is $(u_2,u_3)(x)=(u_{h,2},u_{h,3})(x)+(u_{p,2},u_{p,3})(x)$, where
	\begin{align*}
		\begin{pmatrix}
			u_{h,2}\\ u_{h,3}
		\end{pmatrix} (x)= \begin{cases}
			C_+e^{-\mu_+ x}\begin{pmatrix}\mu_+\\ \ri V_+
			\end{pmatrix}, & x\in \R_+,\\
			C_-e^{\mu_- x} \begin{pmatrix}	\mu_-\\ -\ri V_-
			\end{pmatrix}, & x\in\R_-,
		\end{cases}
	\end{align*}
	\begin{align*}
		\begin{pmatrix}
			u_{p,2}\\ u_{p,3}
		\end{pmatrix}(x) =  \begin{cases}
			\frac{\ri}{2}\begin{pmatrix}
				\mu_+\left(e^{\mu_+x}\int_{x}^{\infty}\rho_+^{(1)}(s)e^{-\mu_+s}\dd s+e^{-\mu_+x}\int_{0}^{x}\rho_+^{(2)}(s)e^{\mu_+s}\dd s\right)\\
				-\ri V_+\left(e^{\mu_+x}\int_{x}^{\infty}\rho_+^{(1)}(s)e^{-\mu_+s}\dd s-e^{-\mu_+x}\int_{0}^{x}\rho_+^{(2)}(s)e^{\mu_+s}\dd s\right)
			\end{pmatrix}, &x\in \R_+,\\
			\frac{\ri}{2}\begin{pmatrix}
				\mu_-\left(e^{\mu_-x}\int_{x}^{0}\rho_-^{(1)}(s)e^{-\mu_-s}\dd s+e^{-\mu_-x}\int_{-\infty}^{x}\rho_-^{(2)}(s)e^{\mu_-s}\dd s\right)\\
				-\ri V_-\left(e^{\mu_-x}	\int_{x}^{0}\rho_-^{(1)}(s)e^{-\mu_-s}\dd s-e^{-\mu_-x}\int_{-\infty}^{x}\rho_-^{(2)}(s)e^{\mu_-s}\dd s\right)
			\end{pmatrix}, & x\in\R_-,
		\end{cases}
	\end{align*}
	with
	$$\mu_\pm=\sqrt{n^2k^2 + (n\omega_R +\ri\nu \omega_I) V_\pm},$$
	\beq\label{E:rhos}
	\begin{split}
		\rho_\pm^{(1)}(x):=\ri\frac{r_2(x)}{V_\pm}+\frac{nkr_1(x)}{V_\pm\mu_\pm}+\frac{r_3(x)}{\mu_\pm},\\ \rho_\pm^{(2)}(x):=\ri\frac{r_2(x)}{V_\pm}-\frac{nkr_1(x)}{V_\pm\mu_\pm}-\frac{r_3(x)}{\mu_\pm},
	\end{split}
	\eeq
	\begin{align*}
		C_- & = \frac{\ri}{2\left(\mu_-V_++\mu_+V_-\right)}\left[\left(\mu_+V_--\mu_-V_+\right)\int_{-\infty}^{0}\rho_-^{(2)}(s)e^{\mu_-s}\dd s + 2\mu_+V_+\int_{0}^{\infty}\rho_+^{(1)}(s)e^{-\mu_+s}\dd s\right],\\
		C_+&=\frac{\mu_-}{\mu_+}C_-+\frac{\ri \mu_-}{2\mu_+}\int_{-\infty}^{0}\rho_-^{(2)}(s)e^{\mu_-s}\dd s-\frac{\ri}{2}\int_{0}^{\infty}\rho_+^{(1)}(s)e^{-\mu_+s}\dd s,
	\end{align*}
	where $\omega_0^{(n,\nu)}=n\omega_R+\ri\nu\omega_I$. The constants $C_\pm$ are chosen such hat $u_2$ and $u_3$ are continuous at $x=0$, i.e., such that the interface conditions $ \llbracket u_2\rrbracket=\llbracket u_3\rrbracket=0$. This is required for $u \in D_{\cL}$.
	Note that all the denominators in \eqref{E:rhos} and in $C_\pm$ are bounded away from zero by the estimates in Lemmas \ref{bds_V_pm}, \ref{bds_V_pm_L}, and \ref{bd_abs_mu_minu_L}. The denominator $\mu_-V_+ +\mu_+V_-$ is also non-vanishing, see \eqref{DR}.
	From \eqref{u_1} we also define $u_{p,1}$ and $u_{h,1}$ as
	\begin{equation*}\label{E:u1-expl}
		u_{p,1}=\frac{1}{ V_\pm}\left(nk u_{p,3} -r_1\right), \qquad u_{h,1}=\frac{nk}{ V_\pm}u_{h,3} \quad  \text{on} \ \R_\pm \ \text{respectively.}
	\end{equation*}
	
	For the $L^2$-estimate first note that Lemma 4.7. of \cite{BDPW25} provides an estimate for all the integral terms in  $u_p$:
	\beq\label{E:est-rho-int}
	\begin{aligned}
		\left\|e^{\mu_+ \cdot }\int_{\cdot}^{\infty}\rho_+^{(1)}(s)e^{-\mu_+s}\dd s\right\|_{L^2(\R_+)}&\leq \frac{c}{\Real(\mu_+)}\left\|\rho_+^{(1)}\right\|_{L^2(\R_+)}, \\
		\left\|e^{-\mu_+ \cdot }\int_{0}^{\cdot}\rho_+^{(2)}(s)e^{\mu_+s}\dd s\right\|_{L^2(\R_+)}&\leq \frac{c}{\Real(\mu_+)}\left\|\rho_+^{(2)}\right\|_{L^2(\R_+)}, \\
		\left\|e^{-\mu_- \cdot }\int_{-\infty}^{\cdot}\rho_-^{(2)}(s)e^{\mu_-s}\dd s\right\|_{L^2(\R_-)}&\leq \frac{c}{\Real(\mu_-)}\left\|\rho_-^{(2)}\right\|_{L^2(\R_-)}, \\
		\left\|e^{\mu_- \cdot }\int_{\cdot}^{0}\rho_-^{(1)}(s)e^{-\mu_-s}\dd s\right\|_{L^2(\R_-)}&\leq \frac{c}{\Real(\mu_-)}\left\|\rho_-^{(1)}\right\|_{L^2(\R_-)}.
	\end{aligned}
	\eeq
	The $\rho$-terms in \eqref{E:rhos} are estimated via
	\beq\label{rho-est}
	\|\rho^{(1,2)}_\pm\|_{L^2(\R_\pm)} \leq \frac{c}{|V_\pm|}\left(\|r_2\|_{L^2(\R_\pm)} + \frac{|n|}{|\mu_\pm|}\|r_1\|_{L^2(\R_\pm)}\right) + \frac{1}{|\mu_\pm|}\|r_3\|_{L^2(\R_\pm)},
	\eeq
	with a constant $c$ independent of $n$ and $\nu$. Thus, using the Cauchy-Schwarz inequality with the norm \beq \label{eqn:L2normExponential}\|e^{\mp \mu_\pm \cdot}\|_{L^2(\R_\pm)}=(2\Real (\mu_\pm))^{-1/2}\eeq and the estimates from Lemma \ref{bds_V_pm} and Lemmas \ref{bds_V_pm_L}--\ref{bd_re_mu_minu_L}, we have
	$$  \zeta:=\left|\int_{0}^{\infty}\rho_+^{(1)}(s)e^{-\mu_+s}\dd s\right|\leq c \nu^{-3/2}\|r\|_{L^2(\R_+)},$$
	$$ \xi:=\left|\int_{-\infty}^{0}\rho_-^{(2)}(s)e^{\mu_-s}\dd s\right|\leq c\left(\nu^{2}e^{\frac74(\nu\omega_I+\ga)T}\|r_1\|_{L^2(\R_-)}+\nu e^{\frac54(\nu\omega_I+\ga)T}\|r_2\|_{L^2(\R_-)}+e^{\frac34(\nu\omega_I+\ga)T}\|r_3\|_{L^2(\R_-)}\right).$$
	Consequently, using \eqref{den-est},
	\beq\label{C_minus}
	\begin{split}
		|C_-|\leq& ce^{(\nu\omega_I+\ga)T}\left(e^{-(\nu\omega_I+\ga)T}\xi+\nu^2\zeta\right)\\
		\leq& c\left[\|r_1\|_{L^2(\R_-)}\nu^2e^{\frac74(\nu\omega_I+\ga)T}+\|r_2\|_{L^2(\R_-)}\nu e^{\frac54(\nu\omega_I+\ga)T}+\|r_3\|_{L^2(\R_-)}e^{\frac34(\nu\omega_I+\ga)T}+\nu^{\frac{1}{2}}e^{(\nu\omega_I+\ga)T}\|r\|_{L^2(\R_+)}\right]
	\end{split}
	\eeq
	and
	\beq\label{C_plus}
	\begin{split}
		|C_+|\leq& c\left[\nu^{-1}e^{-\frac{1}{2}(\nu\omega_I+\ga)T}(|C_-|+\xi)+\zeta\right]\\
		\leq& c\left[\|r_1\|_{L^2(\R_-)}\nu e^{\frac54(\nu\omega_I+\ga)T}+\|r_2\|_{L^2(\R_-)} e^{\frac34(\nu\omega_I+\ga)T}+\|r_3\|_{L^2(\R_-)}\nu^{-1}e^{\frac14(\nu\omega_I+\ga)T}+\nu^{-\frac{3}{2}}\|r\|_{L^2(\R_+)}\right].
	\end{split}
	\eeq
	
	Next, we estimate the homogeneous solution $u_h$ and the particular solution $u_p$. Using again \eqref{eqn:L2normExponential}, the estimates \eqref{E:est-rho-int}, \eqref{rho-est}, \eqref{C_minus}, and \eqref{C_plus}, as well as Lemma \ref{bds_V_pm}, we obtain, for  all $\nu>\frac{\ga}{|\omega_I|}$ and $|n|\leq\nu$,
	\beq\label{est-uh}
	\begin{split}
		\|u_{h}\|_{L^2(\R_+)}
		&\leq c\left[\|r_1\|_{L^2(\R_-)}\nu^{\frac{3}{2}}e^{\frac{5}{4}(\nu\omega_I+\ga)T}+\|r_2\|_{L^2(\R_-)}\nu^{\frac{1}{2}}e^{\frac{3}{4}(\nu\omega_I+\ga)T}\right]\\
		&+c\left[\|r_3\|_{L^2(\R_-)}\nu^{-\frac{1}{2}}e^{\frac{1}{4}(\nu\omega_I+\gamma)T}+\nu^{-1}\|r\|_{L^2(\R_+)}\right],\\
		\|u_{h,1}\|_{L^2(\R_-)} &\leq c\left[\|r_1\|_{L^2(\R_-)}\nu^{3}e^{2(\nu\omega_I+\ga)T}+\|r_2\|_{L^2(\R_-)}\nu^{2}e^{\frac{3}{2}(\nu\omega_I+\ga)T}\right]\\
		&+c\left[\|r_3\|_{L^2(\R_-)}\nu e^{(\nu\omega_I+\gamma)T}+\nu^{\frac{3}{2}}e^{\frac{5}{4}(\nu\omega_I+\gamma)T}\|r\|_{L^2(\R_+)}\right],\\
		\|u_{h,2}\|_{L^2(\R_-)} &\leq c\left[\|r_1\|_{L^2(\R_-)}\nu^{2}e^{\frac32(\nu\omega_I+\ga)T}+\|r_2\|_{L^2(\R_-)}\nu e^{(\nu\omega_I+\ga)T}\right]\\
		&+c\left[\|r_3\|_{L^2(\R_-)} e^{\frac12(\nu\omega_I+\gamma)T}+\nu^{\frac{1}{2}}e^{\frac{3}{4}(\nu\omega_I+\gamma)T}\|r\|_{L^2(\R_+)}\right],\\
		\|u_{h,3}\|_{L^2(\R_-)} &\leq c\left[\|r_1\|_{L^2(\R_-)}\nu e^{(\nu\omega_I+\ga)T}+\|r_2\|_{L^2(\R_-)}e^{\frac{1}{2}(\nu\omega_I+\ga)T}\right]\\
		&+c\left[\|r_3\|_{L^2(\R_-)}\nu^{-1} +\nu^{-\frac{1}{2}}e^{\frac{1}{4}(\nu\omega_I+\gamma)T}\|r\|_{L^2(\R_+)}\right],\\
	\end{split}
	\eeq
	and
	\beq\label{Est_up}\begin{split}
		\|u_{p,j}\|_{L^2(\R_+)}\leq& ~c\|r\|_{L^2(\R_+)}\nu^{-1},\quad j=1,~2,~3,\\
		\|u_{p,1}\|_{L^2(\R_-)}\leq& ~c\left(2\|r_1\|_{L^2(\R_-)}\nu e^{(\nu\omega_I+\ga)T}+\|r_2\|_{L^2(\R_-)}\nu^2e^{\frac32(\nu\omega_I+\ga)T}+\|r_3\|_{L^2(\R_-)}\nu e^{(\nu\omega_I+\ga)T}\right),\\
		\|u_{p,2}\|_{L^2(\R_-)}\leq& ~c\left(\|r_1\|_{L^2(\R_-)}\nu^2 e^{\frac32(\nu\omega_I+\ga)T}+\|r_2\|_{L^2(\R_-)}\nu e^{(\nu\omega_I+\ga)T}+\|r_3\|_{L^2(\R_-)} e^{\frac12(\nu\omega_I+\ga)T}\right),\\
		\|u_{p,3}\|_{L^2(\R_-)}\leq& ~c\left(\|r_1\|_{L^2(\R_-)}\nu e^{(\nu\omega_I+\ga)T}+\|r_2\|_{L^2(\R_-)}e^{\frac12(\nu\omega_I+\ga)T}+\|r_3\|_{L^2(\R_-)}\nu^{-1}\right).
	\end{split}\eeq
	Estimates \eqref{est-uh} and \eqref{Est_up} yield \eqref{est-uh-thm} and \eqref{Est_up-thm}, respectively.
	\epf

	Next, we estimate the inverse of the operator pencil $\cL_{nk}^T(\cdot, n\omega_R+\ri \nu\omega_I)$ in the norm $\|\cdot\|_{\cH^1\to \cH^1}$ in order to check assumption \ref{ass:A-resolvEst}.
	\begin{thm}\label{resol_H1}
		Under the assumptions of Theorem \ref{resol_L2} the solution $u$ of $\cL_{nk}^T(\cdot, n\omega_R+\ri \nu\omega_I)u=r\in \cH^1$ can be decomposed into a homogeneous part $u_h$ and a particular part $u_p$ satisfying
		\beq\label{Est-uh-H1}
		\begin{split}
			\left\|u_{h}\right\|_{H^1(\R_+)}&\leq c\left(\|(r_1,r_2)\|_{L^2(\R_-)} \nu^{\frac32}e^{\frac34(\nu\omega_I+\gamma)T} +\|r_3\|_{L^2(\R_-)}\nu^{\frac12}e^{\frac14(\nu\omega_I+\gamma)T} +\|r\|_{L^2(\R_+)}\right), \\
			\left\|u_{h,1}\right\|_{H^1(\R_-)}&\leq c\left(\|(r_1,r_2)\|_{L^2(\R_-)} \nu^{2}e^{(\nu\omega_I+\gamma)T}+\|r_3\|_{L^2(\R_-)}\nu e^{\frac12(\nu\omega_I+\gamma)T}
			+\|r\|_{L^2(\R_+)} \nu^{\frac{3}{2}}e^{\frac{3}{4}(\nu\omega_I+\ga)T}\right),\\
			\left\|u_{h,2}\right\|_{H^1(\R_-)}&\leq c\left(\|(r_1,r_2)\|_{L^2(\R_-)}\nu e^{\frac12(\nu\omega_I+\gamma)T}+\|r_3\|_{L^2(\R_-)}
			+ \|r\|_{L^2(\R_+)} \nu^{\frac{1}{2}}e^{\frac{1}{4}(\nu\omega_I+\ga)T}\right),\\
			\left\|u_{h,3}\right\|_{H^1(\R_-)}&\leq c\left(\|(r_1,r_2)\|_{L^2(\R_-)} +\|r_3\|_{L^2(\R_-)} \nu^{-1} e^{-\frac12(\nu\omega_I+\gamma)T} + \|r\|_{L^2(\R_+)}\nu^{-\frac{1}{2}}e^{-\frac{1}{4}(\nu\omega_I+\ga)T}\right),
		\end{split}
		\eeq
		and
		\beq\label{Est-up-H1}
		\begin{split}
			\left\|u_{p,1}\right\|_{H^1(\R_+)}&\leq c\|r\|_{H^1(\R_+)},\\
			\left\|u_{p,j}\right\|_{H^1(\R_+)}&\leq c\|r\|_{L^2(\R_+)}, \quad j=2,3,\\
			\left\|u_{p,1}\right\|_{H^1(\R_-)}&\leq c\left(\|(r_1,r_2)\|_{H^1(\R_-)} \nu^{2}e^{(\nu\omega_I+\gamma)T}+\|r_3\|_{L^2(\R_-)}\nu e^{\frac12(\nu\omega_I+\gamma)T}\right),\\
			\left\|u_{p,2}\right\|_{H^1(\R_-)}&\leq c\left(\|(r_1,r_2)\|_{L^2(\R_-)}\nu e^{\frac12(\nu\omega_I+\gamma)T}+\|r_3\|_{L^2(\R_-)}\right),\\
			\left\|u_{p,3}\right\|_{H^1(\R_-)}&\leq c\left(\|(r_1,r_2)\|_{L^2(\R_-)} +\|r_3\|_{L^2(\R_-)} \nu^{-1} e^{-\frac12(\nu\omega_I+\gamma)T}\right),
		\end{split}
		\eeq
		for all $(n,\nu )\in \mathbf{I}$ with $\nu>\gamma/|\omega_I|$ with some constant $c>0$ independent of $\nu$, $n$.
	\end{thm}
	\bpf Again we suppress the dependence of $\mu_\pm$ and $V_\pm$ on $(n,\nu)$.
	The homogeneous solution $u_h$ satisfies $\frac{du_h}{dx}=\mp\mu_\pm u_h$ on $\R_\pm$ and hence using \eqref{est-uh}, one obtains
	\[\begin{split}\left\|u_{h}\right\|_{H^1(\R_+)}&\leq c\left(\|r_1\|_{L^2(\R_-)}\nu^{\frac52}e^{\frac54(\nu\omega_I+\gamma)T}+\|r_2\|_{L^2(\R_-)}\nu^{\frac32}e^{\frac34(\nu\omega_I+\gamma)T}+\|r_3\|_{L^2(\R_-)}\nu^{\frac12}e^{\frac14(\nu\omega_I+\gamma)T}\right)\\
		&+c\|r\|_{L^2(\R_+)},\end{split}\]
	\[\begin{split}\left\|u_{h,1}\right\|_{H^1(\R_-)}&\leq c\left(\|r_1\|_{L^2(\R_-)}\nu^{3}e^{\frac32(\nu\omega_I+\gamma)T}+\|r_2\|_{L^2(\R_-)}\nu^{2}e^{(\nu\omega_I+\gamma)T}+\|r_3\|_{L^2(\R_-)}\nu e^{\frac12(\nu\omega_I+\gamma)T}\right)\\
		&+c\|r\|_{L^2(\R_+)}\nu^{\frac{3}{2}}e^{\frac{3}{4}(\nu\omega_I+\ga)T},\end{split}\]
	\[\begin{split}\left\|u_{h,2}\right\|_{H^1(\R_-)}&\leq c\left(\|r_1\|_{L^2(\R_-)}\nu^{2}e^{(\nu\omega_I+\gamma)T}+\|r_2\|_{L^2(\R_-)}\nu e^{\frac12(\nu\omega_I+\gamma)T}+\|r_3\|_{L^2(\R_-)}\right)\\
		&+c\|r\|_{L^2(\R_+)}\nu^{\frac{1}{2}}e^{\frac{1}{4}(\nu\omega_I+\ga)T},\end{split}\]
	\[\begin{split}\left\|u_{h,3}\right\|_{H^1(\R_-)}&\leq c\left(\|r_1\|_{L^2(\R_-)}\nu e^{\frac12(\nu\omega_I+\gamma)T}+\|r_2\|_{L^2(\R_-)}+\|r_3\|_{L^2(\R_-)}\nu^{-1} e^{-\frac12(\nu\omega_I+\gamma)T}\right)\\
		&+c\|r\|_{L^2(\R_+)}\nu^{-\frac{1}{2}}e^{-\frac{1}{4}(\nu\omega_I+\ga)T},\end{split}\]
	for $\nu>\frac{\ga}{|\omega_I|}$, which implies \eqref{Est-uh-H1}.
	
	It remains to estimate $\|\tfrac{d}{dx}u_p\|_{L^2(\R_\pm)}$. For $x>0$ the derivative is given by
	\[\begin{split}
		\frac{du_{p,2}}{dx}(x)&=\frac{\ri\mu_+^2}{2}\left(e^{\mu_+x}\int_{x}^{\infty}\rho_+^{(1)}(s)e^{-\mu_+s}\dd s-e^{-\mu_+x}\int_{0}^{x}\rho_+^{(2)}(s)e^{\mu_+s}\dd s\right) +\frac{\ri \mu_+}{2}\left( \rho_+^{(2)}(x) -\rho_+^{(1)}(x)\right),\\
		\frac{du_{p,3}}{dx}(x)&=\frac{1}{2} V_+\mu_+\left(e^{\mu_+x}\int_{x}^{\infty}\rho_+^{(1)}(s)e^{-\mu_+s}\dd s+e^{-\mu_+x}\int_{0}^{x}\rho_+^{(2)}(s)e^{\mu_+s}\dd s\right)\\& -\frac{1}{2}V_+\left( \rho_+^{(2)}(x) +\rho_+^{(1)}(x)\right),\\
		\frac{du_{p,1}}{dx}(x)&=\frac{1}{V_+}\left( nk\frac{du_{p,3}}{dx}(x) -\frac{dr_1}{dx}(x)\right).
	\end{split}\]
	Using \eqref{E:est-rho-int} and \eqref{rho-est}, and in the second step Lemma \ref{bds_V_pm}, we obtain the following estimates
	\[\begin{split}
		\left\|\frac{du_{p,1}}{dx}\right\|_{L^2(\R_+)}\leq c\left(\nu^{-1}\left\|r_1'\right\|_{L^2(\R_+)}+\|r\|_{L^2(\R_+)}\right),\quad
		\left\|\frac{du_{p,j}}{dx}\right\|_{L^2(\R_+)}\leq  c\|r\|_{L^2(\R_+)},\quad j=2,~3.
	\end{split}\]
	In an analogous way we get the following estimates for $\left\|\tfrac{d}{dx}u_p\right\|_{L^2(\R_-)}$
	\[\begin{split}
		\left\|\frac{du_{p,1}}{dx}\right\|_{L^2(\R_-)}
		&\leq  c\left(\nu^3e^{\frac 32(\nu\omega_I+\ga)T}\left\|r_1\right\|_{L^2(\R_-)}+\nu^2e^{(\nu\omega_I+\ga)T}\|r_2\|_{L^2(\R_-)}+\nu e^{\frac{1}{2}(\nu\omega_I+\ga)T}\|r_3\|_{L^2(\R_-)}\right)\\
		&+c\nu e^{(\nu\omega_I+\ga)T}\|r'_1\|_{L^2(\R_-)},\\
		\left\|\frac{du_{p,2}}{dx}\right\|_{L^2(\R_-)}&\leq c\left(\nu^2 e^{(\nu\omega_I+\ga)T}\|r_1\|_{L^2(\R_-)}+\nu e^{\frac{1}{2}(\nu\omega_I+\ga)T}\|r_2\|_{L^2(\R_-)}+\|r_3\|_{L^2(\R_-)}\right),\\
		\left\|\frac{du_{p,3}}{dx}\right\|_{L^2(\R_-)}
		&\leq c\left( \nu e^{\frac{1}{2}(\nu\omega_I+\ga)T}\|r_1\|_{L^2(\R_-)}+\|r_2\|_{L^2(\R_-)}+\nu^{-1}e^{-\frac{1}{2}(\nu\omega_I+\ga)T}\|r_3\|_{L^2(\R_-)}\right).
	\end{split}\]
	Together with Theorem \ref{resol_L2} we obtain \eqref{Est-up-H1}. This concludes the proof of Theorem \ref{resol_H1}.
	\epf

	We are now ready to show that assumption \ref{ass:A-resolvEst} of Theorem \ref{T:main} is satisfied by our example.
	
	\begin{cor}
		\label{cor:summary_example}
		Consider the Maxwell interface problem with the linear permittivity given by the Lorentz model \eqref{E:Lorentz-interf}--\eqref{E:Lnk-DL} under the assumptions \ref{ass:B-signOfOmega}--\ref{ass:B-rationality} and the nonlinear susceptibilities given by \eqref{E:chi23-ex} with $T_N< \frac{1}{2\sqrt{3}}T$. 
		Then assumption \ref{ass:A-resolvEst} of Theorem \ref{T:main} holds.
	\end{cor}
	
	\begin{proof}
		From Theorem \ref{resol_H1} we get (for some $c>0$)
		\begin{align*}
			f_1^+(\nu)&=c, &f_1^-(\nu) &= c\nu e^{\frac{1}{2}(\nu\omega_I+\gamma)T},\\
			f_2^+(\nu)&=c\nu^{-\frac{1}{2}}e^{-\frac{1}{4}(\nu\omega_I+\gamma)T}, \hspace{-3cm} &f_2^-(\nu)&=c,\\
			f_3(\nu)&=c\nu^{-1}e^{-\frac{1}{2}(\nu\omega_I+\gamma)T},
		\end{align*}
		for all $\nu>\frac{\gamma}{|\omega_I|}$, where we collected the slowest exponential decay between all the contributing estimates in \eqref{Est-uh-H1} and \eqref{Est-up-H1}. Note that for $\nu < \gamma/|\omega_I|$ the same functions $f_1^+, f_2^\pm, f_3$ as above can be chosen after possibly modifying the constant $c$. The values of $f_1^+$ and $f_2^+$ are irrelevant because $\chi^{(2)}_+=\chi^{(3)}_+=0$ in  \eqref{E:chi23-ex}. For $f_1^-$ we see (using $\omega_I<0$) that 
		$$f_1^-(\nu)\nu^5e^{\sqrt{3}\nu|\omega_I|T_N}=c\nu^6e^{-(\frac{1}{2}T-\sqrt{3}T_N)|\omega_I|\nu}\leq M,~~\forall \nu\in\N,$$
		with $c>0$ independent of $\nu$, due to $\frac{1}{2}T-\sqrt{3}T_N>0$.
		The constant $f_2^-\equiv 1$ obviously satisfies
		$f_2^-(\nu)\leq Me^{\alpha\nu}$ $\forall \nu\in\N$ with $M=c$ and any $\alpha \geq 0$.
	\end{proof}
	
	\brem
	Note that $f_1^+$ does not decay in $\nu$ because the medium in $\R_+$ is non-dispersive. It is this lack of decay which forces us to choose a linear material in $\R_+$, i.e., $\chi^{(2)}_+=\chi^{(3)}_+=0$.
	\erem
	
	With this assumptions \ref{ass:A-cpctSupport}--\ref{ass:A-resolvEst} of Theorem \ref{T:main} have been verified and the proof of Theorem \ref{T:Lor-example} is complete. \qed
	
	\subsection{Numerical Example}
	\label{S:numerics}

	In this section we present our numerical computation of a polychromatic solution to Maxwell's equations as given by Theorem \ref{T:Lor-example}. We choose an interface between a dielectric and a Lorentz metal modeled by \eqref{E:perm-LT} and \eqref{E:chi23-ex} with the following parameters
	\beq \label{eqn:parameters}
	\alpha=2,~ c_L=20,~ \omega_*=2,~ \gamma=0.5,~ k=3, \text{ and } T=1001\frac{\pi}{c_*}.
	\eeq
	The nonlinear susceptibilities on each half space will be given by \eqref{E:chi23_trunc} with parameters 
	\beq \label{eqn:parameters-NL}
	c^{(2)}_{j,p,q}=2000~\delta_{j,p,q},~c^{(3)}_{j,p,q,r}=1000~\delta_{j,p,q,r},~\tilde{\gamma}=1, \text{ and }\tilde{\omega}_*=3.
	\eeq
	
	The values for $c^{(2)}_{j,j,j}$ and $c^{(3)}_{j,j,j,j}$ are chosen rather large, in order to clearly make the numerical solution of the nonlinear problem visually different from the eigenfunction of the linear problem in Figure \ref{fig:num:solutions}. Note, that the factors $c^{(2)}$ and $c^{(3)}$ directly scale the contribution of the nonlinear terms in the solution. 
	
	\medskip
	\noindent
	\textbf{Verification of Assumptions \ref{ass:B-signOfOmega}--\ref{ass:B-formOfT}}
	
	We check only \ref{ass:B-signOfOmega}--\ref{ass:B-formOfT} because the rationality condition in \ref{ass:B-rationality} cannot be checked numerically. This is because we only have an approximate value for the eigenvalue $\omega$ and, further, because we work in a finite precision arithmetic such that every number is essentially treated as a rational number. We can, of course, check if our chosen value $j_*=1001$ is admissible for \ref{ass:B-rationality}. This is not the case since $\frac{\omega_R}{c_*}\notin \frac{2}{j_*}\Z$. The fact that converging numerical solutions are produced anyway supports our conjecture that the conditions of the existence theorem can be satisfied even without \ref{ass:B-rationality}.
	
	The eigenvalues outside of the set $\Omega_0$ of the corresponding untruncated operator are the solutions of the dispersion relation \eqref{E:ev.cond}, which is a polynomial equation of degree four and can thus
	be explicitly solved by computer algebra programs, yielding four distinct eigenvalues. In view of Lemma \ref{lem:evals_lorentz_close}, we may choose any of those values as a good approximation
	for an eigenvalue of the truncated operator, we choose $\omega^\infty_0=\omega_{R,\infty} + \ri \omega_{I,\infty} \approx 1.8179 - 0.1488\ri$. A standard application of Newton's method then produces (up to a tolerance of $10^{-16}$) 
	$\omega^\infty_0$ also as an eigenvalue of the truncated operator with the parameters from \eqref{eqn:parameters}.
	
	Most of assumption \ref{ass:B-signOfOmega} is easily seen from the eigenvalue, while the simplicity of $\omega^\infty_0$ as a zero of the dispersion relation follows from the fact, that it is a fourth degree polynomial with four distinct roots. Also, $\frac{\gamma}{|\omega_{I,\infty}|}=3.3602$ is sufficiently far away from a natural number to verify \ref{ass:B-omIfarFromGamma}. Checking \ref{ass:B-Vmin}--\ref{ass:B-nonResonance} requires calculating some quantities for finitely many values of $(n,\nu)\in\mathbf{I}$, namely those with $\nu<\frac{\ga}{|\omega_{I,\infty}|},\ |n|\leq\nu$,  or $\nu<\nu_0$ with some $\nu_0$ for $\tilde{\beta}_\infty$ in \ref{ass:B-betatil-phi-inf}. Doing so for \ref{ass:B-Vmin}, \ref{ass:B-muMin} and \ref{ass:B-nonResonance},
	and taking the minimal value over all relevant pairs $(n,\nu)$, we arrive at the data shown in Table \ref{tab:numerics}. All quantities are far enough from zero, verifying the assumptions. Doing the same for $|\varphi_\infty|$ in \ref{ass:B-betatil-phi-inf}, but excluding pairs $(n,\nu)$ with $n=0$, we get the minimal value $0.2104$. For $\tilde{\beta}_\infty$ we first study the quantity $\min_{0<|n|\leq \nu} |\tilde{\beta}_\infty(n,\nu)|$ as a function of $\nu$ to see that for the given parameter values the quadratic asymptotics proposed in Remark \ref{R:beta-phi-inf-asymp} hold for $\nu>60$. It is thus enough to calculate the minimal value of $|\tilde{\beta}_\infty|$ for $(n,\nu)\in\mathbf{I}$ with $\nu\leq 60$, $n\neq 0$. This minimal value is approximately $0.1237$, so again far enough from zero, which validates \ref{ass:B-betatil-phi-inf}.
	\begin{table}[h]
		\centering
		\begin{tabular}{c|c c c}
			Quantity & $\left|\mu_{-,\infty}^2(n,\nu)\right|$ & $\left|V_\infty^{(a)}(n,\nu)\right|$ & $\dist (n\omega_{R,\infty}+\ri\nu\omega_{I,\infty},\sigma_p(\cL_{nk}^\infty))$\\
			\hline \\
			$\min\limits_{\substack{(n,\nu)\in\mathbf{I}\\ \nu <\gamma /|\omega_I|}}$ & $0.0477$ & $0.3207$ & $0.1488$ 
		\end{tabular}
		\caption{Numerically computed minimal values of the quantities relevant for Assumptions \ref{ass:B-Vmin}, \ref{ass:B-muMin} and \ref{ass:B-nonResonance}, for the case of example \eqref{E:perm-LT} with parameters \eqref{eqn:parameters}.}
		\label{tab:numerics}
	\end{table}
	
	Finally, condition \ref{ass:B-formOfT} clearly holds for the above choice of $T$. 
	
	\medskip
	\noindent
	\textbf{Computation of an Approximate Polychromatic Solution}
	
	To calculate a polychromatic solution to the Maxwell equations, we proceed iteratively, analogously to the proof of Theorem \ref{T:main}. Setting $r^{n,\nu}:=\ri h^{n,\nu}(\cdot,\omega_0,\vec{u})$ (where $r^{(n,\nu)}_3=0$ due to the form of $h^{n,\nu}$, see \eqref{eqn:operatorFormulation}--\eqref{eqn:nonlinearity}) and $\omega_0^{(n,\nu)}:=n\omega_R+\ri\nu\omega_I$ in 
	\eqref{eqn:operatorFormulation}, we solve repeatedly the equations 
	\begin{equation}\label{eqn:num:full_sys}
		\begin{rcases}
			\ri nk u_3-\ri V_{\pm}(n,\nu)u_1 &= r_1^{n,\nu} \\
			-u_3'-\ri  V_{\pm}(n,\nu) u_2&= r_2^{n,\nu} \\
			u_2'-\ri nku_1-\ri\omega_0^{(n,\nu)} u_3&=0
		\end{rcases}
		\text{ on $\R_\pm$}
	\end{equation}
	with the interface conditions $\llbracket u_2\rrbracket=\llbracket u_3\rrbracket =0$ to ensure a solution $u\in D_\cL$. The third equation in \eqref{eqn:num:full_sys} can be solved for $u_3$ to obtain an effective system of two equations. In the case $n=0$ we obtain decoupled equation for $u_1$ and $u_2$, namely 
	
	\begin{equation*}
		\label{eqn:num:effective_n0}
		\begin{aligned}
			u_1 &= \frac{\ri}{V_{\pm}(0,\nu)}r^{0,\nu}_1\\
			-u_2'' +\omega_0^{(n,\nu)}V_\pm(0,\nu) u_2 &=\ri\omega_0^{(0,\nu)}r^{0,\nu}_2
		\end{aligned}
	\end{equation*}
	with the interface conditions $\llbracket u_2\rrbracket = \llbracket u_2'\rrbracket=0$, i.e. $u_2$ continuously differentiable. In this case, $u_1$ is given explicitly, while the equation for $u_2$ can be solved by a standard finite difference method. In the case $nk\neq 0$ we get
	
	\begin{equation}
		\label{eqn:num:effective-system}
		\begin{rcases}
			u_2' -\ri \left(\frac{\omega_0^{(n,\nu)} V_\pm (n,\nu)}{nk} +nk\right)u_1 =\frac{\omega_0^{(n,\nu)}}{nk} r^{n,\nu}_1\\
			-u_2''+\ri nku_1' +\omega_0^{(n,\nu)} V_\pm (n,\nu)u_2 =\ri\omega_0^{(n,\nu)} r^{n,\nu}_2
		\end{rcases}
		\text{ on $\R_\pm$}
	\end{equation}
	with the interface conditions $\llbracket u_2\rrbracket =\llbracket \ri nk u_1 -u_2'\rrbracket =0$. We note that $u_1$ is generally not continuous across the interface. The numerics for this ODE system require some care, as for Maxwell systems, so called spectral pollution is known to occur if a standard discretization is used, see \cite{LABV1990}. We describe our procedure in Appendix \ref{app:num}. It is based on the staggered grid $x_j := -d + jh, j = 0, 1, ...,N$ and $\tilde{x}_j = -d + (j + 1/2 )h$, $j = 0, 1, ...,N - 1$ on the interval $[-d,d]$.
	
	The construction of the polychromatic solution then follows the proof of Theorem \ref{T:main} and is described next. The eigenfunctions of $\cL^T_{k}$ for the eigenvalue $\omega_0=\omega_R+\ri\omega_I$ are known explicitly. They are multiples of
	\beq\label{E:efn}
	\varphi(x):=\begin{cases}
		\begin{pmatrix}
			-\ri k e^{\mu_-(1,1)x}\\
			\mu_- e^{\mu_-(1,1) x}\\
			-\ri V_-(1,1) e^{\mu_-(1,1) x}
		\end{pmatrix},~~&x< 0,\\
		\begin{pmatrix}
			\ri k \frac{\mu_-(1,1)}{\mu_+(1,1)}e^{-\mu_+(1,1)x}\\
			\mu_- e^{-\mu_+(1,1) x}\\
			\ri \frac{\mu_-(1,1)}{\mu_+(1,1)} V_+(1,1) e^{-\mu_+(1,1) x}
		\end{pmatrix}, &x>0,
	\end{cases}
	\eeq
	see \cite{BDPW25}. We start by choosing some $\varepsilon>0$ and setting 
	$$\bU^{(1,1)}_j:=\varepsilon\varphi_1(x_j),~\bV^{(1,1)}_j:=\varepsilon\varphi_2(\tilde{x}_j),~\bU^{(-1,1)}_j:=\overline{\bU^{(1,1)}_j},~\bV^{(-1,1)}_j:=\overline{\bV^{(1,1)}_j}.$$
	Next, we construct a discretized version of the right hand side in \eqref{eqn:num:effective-system} for $\nu=2, n=0,1,2$, and then solve the resulting linear system of equations, producing $\bU^{(n,2)}_j,$ $\bV^{(n,2)}_j$ for $n=0,1,2$ (and $\bU^{(-n,2)}_j,\bV^{(-n,2)}_j$ for $n=1,2$ by complex conjugation).
	We then proceed iteratively to compute the components $\bU^{(n,\nu)}_j,\bV^{(n,\nu)}_j$ up to some predefined value of $\nu\in\N$ and construct an approximation of the solution $\psi$ in \eqref{E:Maxw-sol} as the discretization of the corresponding partial sum
	\beq\label{E:part-sum}
	\psi^{(M)}(x,y,t):= \psi(x,y,t):=\sum_{\nu=1}^M\sum_{n\in \Z, |n|\leq \nu}u^{n,\nu}(x)e^{-\ri n(\omega_R t -ky)}e^{\nu \omega_I t}. \eeq
	Figure \ref{fig:num:solutions} shows such an approximation for $M=15$ along with the eigenfunction $\varphi$. Both were computed using $N = 80001$. Clearly, the higher order terms
	contribute only small corrections to the linear eigenfunction. This is also seen in Figure \ref{fig:convergence_solution}, where the $L^2$-norms of the sums $u^\nu(x,y):=\sum_{|n|\leq\nu} u^{n,\nu}(x)e^{\ri nky}$ are plotted against $\nu$, showing the exponential decay 
	of these terms, as expected from Theorem \ref{T:main}.
	
	In Figure \ref{fig:num:solutions} we choose a large value of the scaling parameter, namely $\eps=20$, in order to make the difference between the approximation $\psi^{(M)}$ of the polychromatic solution and the linear eigenfunction $\varphi$ well visible. It is interesting to note that the series appears to converge even for this large value of $\eps$.
	\begin{figure}[]
		\centering
		\begin{subfigure}[t]{0.4\textwidth}
			\includegraphics[width=\textwidth]{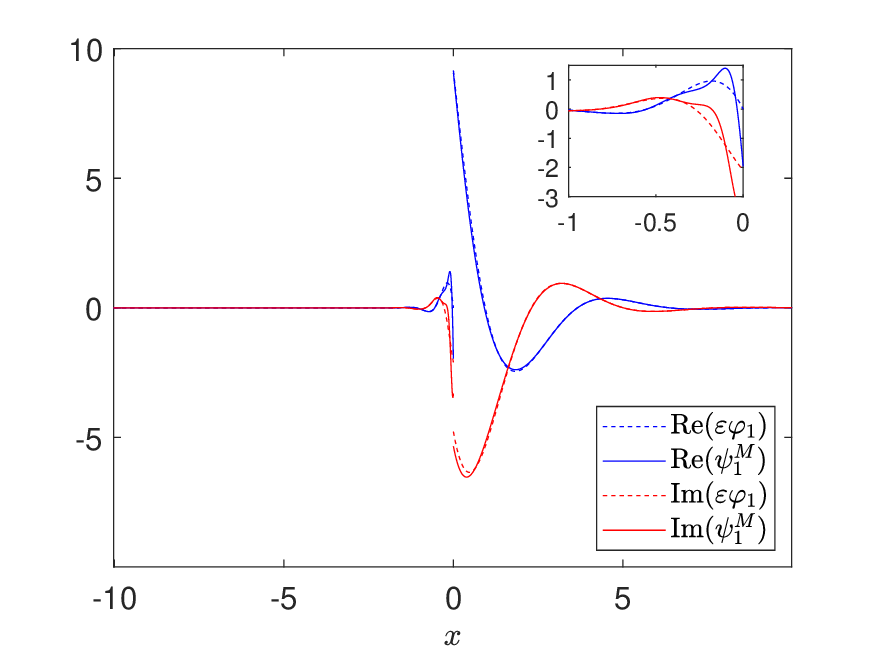}
		\end{subfigure}
		\hspace{0.03\textwidth}
		\begin{subfigure}[t]{0.4\textwidth}
			\includegraphics[width=\textwidth]{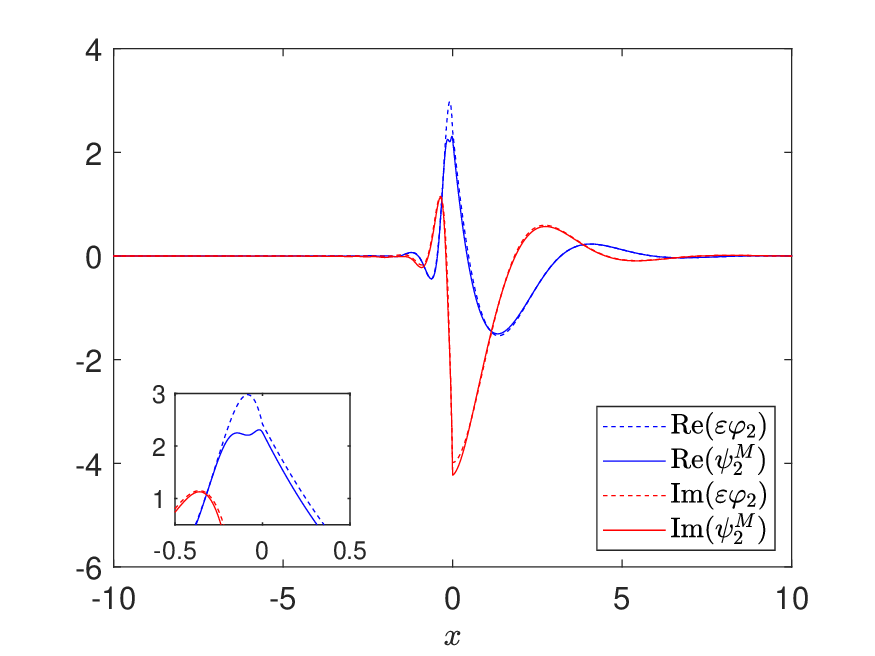}
		\end{subfigure}
		\caption{\label{fig:num:solutions}%
			Numerical approximation of the first and second component of $\eps \varphi$ and the partial sum $\psi^{(M)}$ with $M=15$, see \eqref{E:part-sum}, and with $\eps=20$. Here $\varphi$ is the eigenfunction \eqref{E:efn} of the truncated operator $\cL_{3}^T$ with parameters given by \eqref{eqn:parameters}. Insets show regions, in which the polychromatic solution deviates the most from the linear eigenfunction.}
	\end{figure}
	
	\begin{figure}[h]
		\centering
		\begin{subfigure}[c]{0.4\textwidth}
			\includegraphics[width=\textwidth]{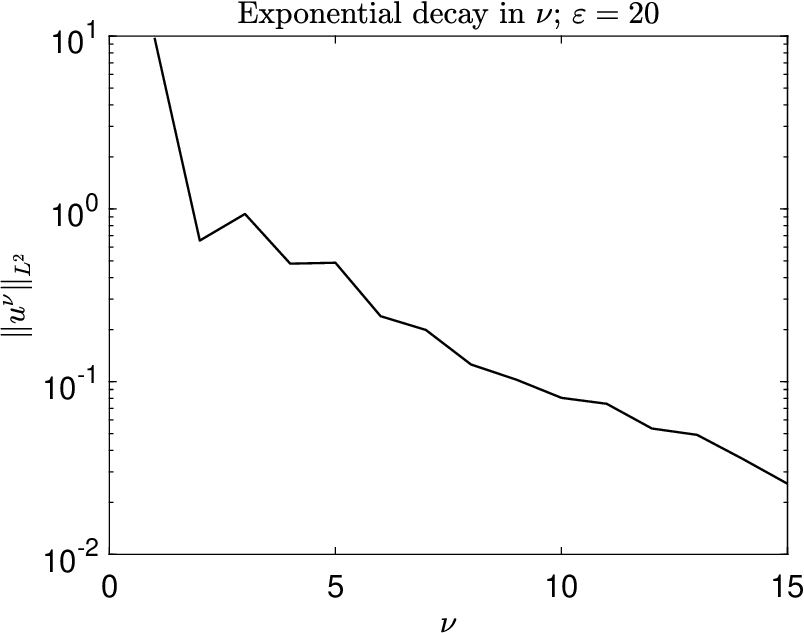}
		\end{subfigure}
		\hspace{0.03\textwidth}
		\begin{subfigure}[c]{0.4\textwidth}
			\includegraphics[scale=0.55]{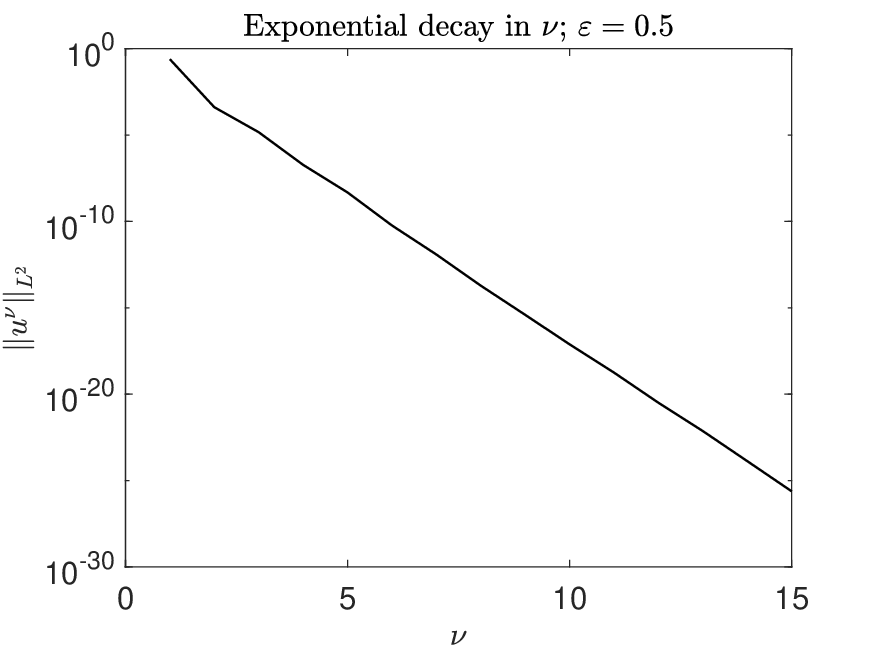}
		\end{subfigure}
		\caption{\label{fig:convergence_solution}%
			The $L^2$-norms over $\R\times (0,2\pi)$ of the numerically computed approximations $u^\nu(x,y):=\sum_{|n|\leq\nu} u^{n,\nu}(x)e^{\ri nky}$, i.e., all terms that decay at rate $\nu\omega_I$ evaluated at time $t=0$, again for the set of parameters \eqref{eqn:parameters} and $\varepsilon=20$ (left) or $\varepsilon=0.5$ (right). Note that the plot is semilogarithmic, such that the dependence on $\nu$ is exponential.}
	\end{figure}
	
	\section{Discussion} \label{Sec_diss}
	
	We have proved the existence of polychromatic exponentially damped breather solutions of cubically nonlinear Maxwell equations in a waveguide geometry, where the guiding is in one dimension ($x$). The polarization model in Maxwell equations is dispersive, i.e., it allows memory terms in time. Our construction requires the memory to be finite. The solutions take the form of a Fourier series with complex frequencies and the series is constructed recursively via a sequence of linear ODE problems on $\R$ where the first term is an eigenfunction. This construction as well as the result appear to be new. The general existence theorem (Theorem \ref{T:main}) assumes firstly that all frequencies except the first one lie in the resolvent set and secondly it requires a certain decay of the resolvent operator in terms of the imaginary part of the spectral parameter. Our specific existence theorem (Theorem \ref{T:Lor-example}) is an application of Theorem \ref{T:main} to an example of a waveguide given by the one dimensional interface (at $x=0$) of two homogeneous materials, one being nonlinear and described by the Lorentz model with a finite memory and the other one linear and non-dispersive. The assumptions of Theorem \ref{T:main} are satisfied by this example under certain conditions on the spectrum of the interface problem.
	
	Theorem \ref{T:main} easily generalizes to finitely many interfaces. In assumption \ref{ass:A-periodic} the susceptibilities then need to be eventually periodic only in the unbounded layers. In assumption \ref{ass:A-resolvEst} the same estimates need to hold in each layer. On the other hand, a corresponding generalization of Theorem \ref{T:Lor-example} is less trivial. The spectrum of the linear problem with more than one interface and without temporal truncation needs to be studied first. The proof becomes more technical also due to the additional interface conditions. A generalization to higher order polynomial nonlinearities can be covered by slightly modifying the recursive sequence of linear problems. For higher dimensional settings, i.e., with $\chi^{(j)}=\chi^{(j)}(x,y)$, or even $\chi^{(j)}=\chi^{(j)}(x,y,z)$, a general existence result analogous  to Theorem \ref{T:main} can still be formulated. Polychromatic solutions are then constructed using a sequence of linear PDEs on $\R^2$ or $\R^3$, respectively. The main difficulty here is to generate an example which satisfies the required decay of the resolvent operator.

	\appendix
	
	\section{Convolutions and the Fourier Transform for Distributions of Compact Support}
	\label{S:distrib}
	This review is based on the standard discussions of distributions, convolutions, and the Fourier transform in \cite{rudinFA} and \cite{trevesTVS}. We denote the space of distributions in $\R^n$ by $\cD'(\R^n)$.
	\begin{defn}
		For $n \in \N$ we define $\cD_c'(\R^n)$ as the space of distributions with compact support, i.e., $\Lambda \in \cD_c'(\R^n)$ means that $\Lambda \in \cD'(\R^n)$ and there exists a compact set $M\subset \R^n$ such that $\Lambda(\varphi)=0$ for all $\varphi \in C^\infty_c(\R^n)$ with $\supp~ \varphi\subset \R^n\setminus M$. We write $\supp~\Lambda \subset M$.
	\end{defn}
	The convolution between $\Lambda \in \cD'(\R^n)$ and $\varphi\in C^\infty_c(\R^n)$ is defined as $(\Lambda*_{\R^n}\varphi)(t_1,\dots,t_n):=\Lambda(\varphi((t_1,\dots,t_n)-\cdot))$. As shown in Theorem 6.24 in \cite{rudinFA}, distributions $\Lambda \in \cD_c'(\R^n)$ extend to continuous linear functionals on $C^\infty(\R^n)$. It is even the case, that the space $\cD_c'(\R^n)$ is exactly the functional analytic dual of $C^\infty(\R^n)$, see Theorem 24.2 in \cite{trevesTVS}. In particular, the convolution of our compactly supported $\chi^{(1)}(x), \chi^{(2)}(x)$, and $\chi^{(3)}(x)$ with the exponential functions $e^{(z)}$, $e^{(z_1,z_2)}$, and $e^{(z_1,z_2,z_3)}$ 
	make sense. Here we define for $z_1,\dots, z_n \in \C$
	$$e^{(z_1,\dots,z_n)}:\R^n\to \C, \ e^{(z_1,\dots,z_n)}(t_1,\dots,t_n):=e^{\ri(z_1t_1+\dots+z_nt_n)}.$$
	We  have
	\beq\label{E:convol-e}
	\begin{aligned}
		\left(\chi^{(n)}(x)*_{\R^n}e^{(z_1,\dots,z_n)}\right)(t_1,\dots,t_n)&=\chi^{(n)}(x)(e^{(z_1,\dots,z_n)}((t_1,\dots,t_n)-\cdot)) \\
		&= e^{\ri(z_1t_1+\dots+z_nt_n)} \chi^{(n)}(x)(e^{(-z_1,\dots,-z_n)}).
	\end{aligned}
	\eeq
	
	Next, we discuss the Fourier transform of $\Lambda \in \cD'_c(\R^n)$. Recall that for an $L^1$-function $f$ the temporal Fourier transform $\hat{f}$ was defined in \eqref{E:FT-L1}. In the notation of distributions, $\hat{f}(\omega_1,\dots,\omega_n) = f(e^{(\omega_1,...,\omega_n)})$ for all $(\omega_1,\dots,\omega_n)\in \R^n$ because for $\Lambda  \in \cD'(\R^n)$ the Fourier transform is a distribution defined via $\hat{\Lambda}(\varphi):=\Lambda(\hat{\varphi})$ for all $\varphi \in C^\infty_c(\R^n)$. The theorem of Paley-Wiener (see Theorem 7.23 in \cite{rudinFA}) shows that for $\Lambda \in \cD'_c(\R^n)$ the map $(z_1,\dots,z_n)\mapsto \Psi(z_1,\dots,z_n):=\Lambda(e^{(z_1,\dots,z_n)})$ is entire (in $\C^n$) and the restriction $\Psi|_{\R^n}$ equals the temporal Fourier transform $\hat{\Lambda}$. To simplify the notation, we use the notation $\hat{\Lambda}$ for $\Psi$ on $\C^n$, i.e., given $\Lambda \in \cD_c'(\R^n)$, we define
	\beq\label{E:FT-distrib-comp}
	\hat{\Lambda}(\omega_1,\dots,\omega_n):=\Lambda(e^{(\omega_1,\dots,\omega_n)}) \quad \forall (\omega_1,\dots,\omega_n)\in \C^n.
	\eeq
	This is consistent with the definition in \eqref{E:FT-L1} for $L^1$-functions.
	
	Let us now apply the above discussion to our susceptibility tensors. For ease of notation we drop the lower indices of $\chi^{(n)}$. The convolution of the compactly supported $\chi^{(1)}(x)\in \cD'_c(\R)$ with the $t-$dependent exponential factor $e^{-\ri(n\omega_R+\ri\nu \omega_I)t}$ in our polychromatic ansatz yields the temporal Fourier transform of $\chi^{(1)}(x)$. In fact, for any $\omega \in \C$ we have from \eqref{E:convol-e}
	\beq\label{E:chi1-convol}
	\left(\chi^{(1)}(x)*_{\R}e^{(-\omega)}\right)(t)=e^{-\ri\omega t}\chi^{(1)}(x)(e^{(\omega)})=e^{-\ri\omega t}\hat{\chi}^{(1)}(x,\omega).
	\eeq
	Analogously, for $\chi^{(2)} \in \cD_c'(\R^2)$ and $\chi^{(3)}(x)\in \cD'_c(\R^3)$ we have for any $\omega_1,\omega_2,\omega_3\in \C$
	\beq\label{E:chi3-convol}
	\begin{aligned}
		\left(\chi^{(2)}(x)*_{\R^2}e^{(-\omega_1,-\omega_2)}\right)(t,t)&=e^{-\ri(\omega_1+\omega_2) t}\hat{\chi}^{(2)}(x,\omega_1,\omega_2),\\
		\left(\chi^{(3)}(x)*_{\R^3}e^{(-\omega_1,-\omega_2,-\omega_3)}\right)(t,t,t)&=e^{-\ri(\omega_1+\omega_2+\omega_3) t}\hat{\chi}^{(3)}(x,\omega_1,\omega_2,\omega_3).
	\end{aligned}
	\eeq
	Equations \eqref{E:chi1-convol} and \eqref{E:chi3-convol} explain the form of the polarization $\cP$ in \eqref{E:D-single-harm}.

	\brem
	In the case of the instantaneous response we have $\chi^{(1)}(x)=a^{(1)}(x)\delta$ with $\delta \in \cD'(\R)$, $\chi^{(2)}(x)=a^{(2)}(x) \delta$ with $\delta \in \cD'(\R^2)$, and $\chi^{(3)}(x)=a^{(3)}(x)\delta$ with $\delta \in \cD'(\R^3)$. Then, clearly, $\hat{\chi}^{(1)}(x,\omega)=a^{(1)}(x)$ for all $\omega \in \C$, $\hat{\chi}^{(2)}(x,\omega_1,\omega_2)=a^{(2)}(x)$ for all $\omega_1,\omega_2\in \C$,
	and $\hat{\chi}^{(3)}(x,\omega_1,\omega_2,\omega_3)=a^{(3)}(x)$ for all $\omega_1,\omega_2,\omega_3\in \C.$
	\erem
	
	Finally, we explain in detail the statement of Remark \ref{R:chihat-cont} that continuity and differentiability of a map $f:\R\to\cD_c'(\R^n)$ implies the same for the Fourier transform $\hat{f}(\cdot,\omega ):\R\to\C$ for all $\omega\in\C^n$, where $\hat{f}(x,\omega):=\widehat{f(x)}(\omega)$. Since the weak dual topology on $\cD_c'(\R^n)$ is initial with respect to the dual pairings $p_\varphi:\cD_c'(\R^n)\to\C,~p_\varphi(\Lambda):=\Lambda (\varphi)$ for all $\varphi\in C^\infty (\R^n)$,
	$f$ is continuous if and only if the maps $x\mapsto f(x)\left( \varphi\right)$ are continuous for all $\varphi\in C^\infty(\R^n)$. In particular, $\hat{f}(\cdot ,\omega)=f(\cdot)\left( e^{(\omega)}\right)$ is a continuous map for all $\omega\in\C^n$.
	
	Similarly, differentiability of the map $f$ at a point $x_0\in\R$ means that 
	$$\lim_{x_0\to x} \frac{f(x)-f(x_0)}{x-x_0}=:f'(x_0)\in \cD_c'(\R^n)$$
	exists. Again, because the weak dual topology is initial, this limit exists if and only if the limits 
	$$\lim_{x\to x_0}\left(\frac{f(x)-f(x_0)}{x-x_0} \right)\left( \varphi\right) = \lim_{x\to x_0}\frac{f(x)(\varphi )-f(x_0)(\varphi )}{x-x_0} = f'(x_0)(\varphi )$$
	exist for all $\varphi\in C^\infty (\R^n)$, or in other words if the map $x\mapsto f(x)(\varphi)$ is differentiable at $x_0$. Choosing in particular $\varphi = e^{(\omega )}$ shows the differentiability of $\hat{f}(\cdot,\omega)$.

	\section{Drude Model with Truncated Memory}\label{app:Drude}
	We give here the reason for choosing the Lorentz model for a dispersive material in our main example instead of the simpler Drude model
	\beq\label{E:Drude}
	\hat{\chi}^{(1)}_D(\omega)=-\frac{c_D}{\omega^2+\ri\gamma\omega}
	\eeq
	with $c_D, \gamma > 0$, see \cite{Pitarke_2007,ACL2018}.	One easily checks that in the time domain
	$$\chi^{(1)}_D(t):=\frac{c_D}{\gamma}(1-e^{-\gamma t})\theta(t),$$
	but \eqref{E:Drude} is the Fourier-Laplace transform of $\chi^{(1)}_D$ only for $\Imag(\omega)>0$, and again $\chi^{(1)}_D$ does not satisfy our finite memory assumption. Applying the same truncation 
	\begin{equation*}\label{E:Drude-T}
		\chi^{(1)}_{D,T}(t):=\frac{c_D}{\gamma}(1-e^{-\gamma t})\mathbbm{1}_{[0,T]}(t)
	\end{equation*}
	with $T>0$ as for the truncated Lorentz model in \eqref{E:Lorentz-T}, we get 
	\begin{equation*}\label{E:Drude-T-FT}
		\hat{\chi}^{(1)}_{D,T}(\omega)=-\frac{c_D}{\omega^2+\ri\gamma\omega}\left[1-e^{\ri\omega T}\left(1-\frac{\ri\omega}{\gamma}(1-e^{-\gamma T})\right)\right].
	\end{equation*}
	The interface with a non-dispersive material on the right and a Drude material on the left side of the interface is then modeled via 
	\begin{equation*}\label{E:perm-DT}
		\perm(x,\omega)=\begin{cases}
			\perm_{-,D}(\omega):= \perm_0(1+\hat{\chi}^{(1)}_{D,T}(\omega)), & x<0,\\ 	\perm_+(\omega)= \perm_0(1+\alpha), & x>0
		\end{cases}
	\end{equation*}
	and we denote the corresponding operator from \eqref{E:op-def} by $\cL_{nk,D}^T(\omega)$ while the operator with an untruncated Drude material on the left half plane, i.e., with
	
	\beq\label{E:perm-D}
	\perm(x,\omega)=\begin{cases}
		\perm_{-,D}^\infty(\omega):= \perm_0(1+\hat{\chi}^{(1)}_{D}(\omega)), & x<0,\\ 	\perm_+(\omega)= \perm_0(1+\alpha), & x>0,
	\end{cases}
	\eeq
	is denoted by $\cL^\infty_{nk,D}(\omega).$
	
	The spectrum of $\cL^\infty_{nk,D}$ was described in Example 3.5 of \cite{BDPW25}, showing that there are at most four eigenvalues outside the set $\Omega_0$, see \eqref{eqn:Omega0}, all of which are located in the strip $S_\gamma=\{\omega \in \C: \Imag (\omega)\in (-\gamma,0)\}$. We show now that in
	terms of the point spectrum the operator $\cL_{nk,D}^T$ is not a good approximation of the operator $\cL^{\infty}_{nk,D}$ in the sense of the following lemma.
	
	\begin{lemma}
		Let $k\in\R$ be arbitrary and $K\subset S_\gamma=\{\omega\in\C :~ \Imag\omega\in (-\gamma ,0)\}$ be any compact set that contains all eigenvalues outside $\Omega_0$ of the untruncated operator $\cL^{\infty}_{k,D}$. Then for $T>0$ large enough,
		there are no eigenvalues of $\cL_{k,D}^T$ inside of $K$.
	\end{lemma}
	
	\begin{proof}
		All eigenvalues of $\cL_{k,D}^T$ outside of the set $\Omega_0$ are given by the dispersion relation
		\eqref{E:ev.cond} with $\perm$ defined by \eqref{E:perm-D}. Rewriting the dispersion relation, we see that any eigenvalue of $\cL_{k,D}^T$ must fulfill
		$$k^2\perm_+ + \perm_{-,D}(\omega)\left( k^2-\omega^2\boldsymbol{\mu}_0\perm_+\right)=0.$$

		We substitute $$\perm_{-,D}(\omega)=\perm_0\left( 1-\frac{c_D}{\omega^2+\ri\gamma\omega}\left[1-e^{\ri\omega T}\left(1-\frac{\ri\omega}{\gamma}(1-e^{-\gamma T})\right)\right]\right)$$
		and multiply by $\perm_0^{-1}(\omega^2+\ri\gamma\omega)e^{-\ri\omega T}$ to obtain the condition
		\beq \label{E:dispRel-Drude}
		\begin{aligned}
			F_T(\omega ):=&\left( (\omega^2+\ri\gamma\omega)\left( (\perm_0^{-1}k^2-\omega^2\boldsymbol{\mu}_0)\perm_+ +k^2\right) - c_D(k^2-\omega^2\boldsymbol{\mu}_0\perm_+)\right) e^{-\ri \omega T} \\
			&+ c_D\left( 1-\frac{\ri\omega}{\gamma}(1-e^{-\gamma T})\right) (k^2-\omega^2\boldsymbol{\mu}_0\perm_+) =0.
		\end{aligned}
		\eeq
		
		Choose now any compact set $K\subset S_\gamma$ as in the claim of the lemma. Then the polynomial $$\left( (\omega^2+\ri\gamma\omega) \left( (\perm_0^{-1}k^2-\omega^2\boldsymbol{\mu}_0)\perm_+ +k^2\right) - c_D(k^2-\omega^2\boldsymbol{\mu}_0\perm_+)\right)$$ is bounded on $K$, thus
		$$\Psi_T(\omega ):=\left( (\omega^2+\ri\gamma\omega) \left( (\perm_0^{-1}k^2-\omega^2\boldsymbol{\mu}_0)\perm_+ +k^2\right) - c_D(k^2-\omega^2\boldsymbol{\mu}_0\perm_+)\right) e^{-\ri \omega T}$$
		is uniformly small on $K$ for all $T>0$ large enough because $\Imag (\omega)\leq -c <0$ for all $\omega\in K$. More precisely, for any $\delta>0$, there is some $T_K(\delta )>0$ such that
		$$\max\limits_{\omega\in K} |\Psi_T(\omega )|<\delta,~~\forall T>T_K.$$
		This means, however, that for $T>T_K$, \eqref{E:dispRel-Drude} can only be satisfied inside of $K$ if also $|\Phi_T(\omega )|<\delta$, where
		$$\Phi_T(\omega ):= c_D\left( 1-\frac{\ri\omega}{\gamma}(1-e^{-\gamma T})\right) (k^2-\omega^2\boldsymbol{\mu}_0\perm_+).$$
		Since $\Phi_T$ is a cubic polynomial with roots $r_1=\frac{-\ri\gamma }{1-e^{-\gamma T}}$ and $r_{2,3}=\pm\sqrt{\frac{k^2}{\perm_+}}$, $\Phi_T$ will only be small close to those roots. In other words,
		for any $\eta>0$ there is some $\delta_0>0$ such that
		\begin{equation*} \label{E:Drude-ballEstimate}
			\{ \omega\in\C :~|\Phi_T(\omega )|<\delta \} \subset B_\eta (r_1)\cup B_\eta (r_2)\cup B_\eta (r_3),~~\forall \delta<\delta_0.
		\end{equation*}
		
		To conclude that there are no solutions of \eqref{E:dispRel-Drude} in $K$ for large $T$, it suffices to show that for some $\eta>0$ we have $B_\eta(r_1)\cap K = B_\eta(r_2)\cap K =B_\eta(r_3)\cap K = \emptyset$ for all $T$ large enough. Since $-\ri\gamma, r_2$, and $r_3$ are not in $S_\gamma$ (note that $r_{2,3}\in\R$) and thus not in $K$, we set
		$$\eta := \frac{1}{3} \dist (\{-\ri\gamma ,r_2,r_3\},K)>0.$$
		Clearly, because $r_2$ and $r_3$ are $T$-independent, we get $B_\eta(r_2)\cap K =B_\eta(r_3)\cap K = \emptyset$ for all $T$.
		As $e^{-\gamma T}\to 0$ for $T\to\infty$, we choose $T_0$ such that $|r_1+\ri\gamma|<\eta$ for all $T>T_0$ and get $B_\eta(r_1)\cap K =\emptyset$ for all $T>T_0$. Hence, we have the existence of $\delta>0$ such that
		$$\{ \omega\in\C :~|\Phi_T(\omega )|<\delta \}\cap K =\emptyset ,$$
		i.e., $|\Phi_T(\omega )|\geq\delta$ in $K$,  provided $T>T_0$.
		
		Choosing now $T>\max\{T_0,T_K(\delta )\}$, we conclude
		$$|F_T(\omega )|\geq |\Phi_T(\omega)|-|\Psi_T(\omega )|>\delta - \delta =0, ~~\forall \omega \in K,$$
		so there are no eigenvalues of $\cL_{k,D}^T$ in the set $K\setminus \Omega_0$.
		
		A similar argument shows, that $\perm_-(\omega)\neq 0$ for all $\omega\in S_\gamma$ if $T$ is large enough. Thus, also eigenvalues from the set $\Omega_0$ cannot lie in $K$.
	\end{proof}
	
	The reason, why the spectrum of the Drude model is strongly perturbed by the truncation in the time domain, is that  $\chi^{(1)}_D(t)$ does not decay for large $t$. This is in contrast with the Lorentz model, where $\chi^{(1)}_L(t)\to 0$ exponentially as $t\to\infty$. In the Drude model the history does not become less important with age.

	\section{Discretization on a Staggered Grid} \label{app:num}
	
	In this appendix we give a brief overview of the discretization used to compute solutions to equation \eqref{eqn:num:effective-system}.
	We discretize on a large interval $[-d,d]$ with some $d>0$ and impose perfect conductor boundary conditions, i.e., $u_2(-d)=u_2(d)=0$.
	To avoid spectral pollution we choose to work on a staggered grid as suggested in \cite{LABV1990}, i.e., we fix some $N\in \N_{\mathrm{even}}$, implying a step size $h:=\frac{2d}{N}$, and define the integer grid points $x_j:=-d + jh$, $j=0,1,...,N$ and the half integer grid points $\tilde{x}_j=-d +(j+\frac{1}{2})h$, $j=0,1,...,N-1$.
	The functions $u_1$ and $u_2$ are then evaluated at $(x_j)_j$ and $(\tilde{x}_j)_j$ respectively. Thus, they are approximated by 
	$$\mathbf{U}:=(u_1(x_0),u_1(x_1),...,u_1(x_N)) \text{ and } \mathbf{V}:=(u_2(\tilde{x}_0),u_2(\tilde{x}_1),...,u_2(\tilde{x}_{N-1})),$$
	yielding $2N+1$ (complex valued) unknowns. Since $u_1$ is not necessarily continuous in $x_{N/2}=0$, the value of $\bU$ at this point is to be understood as the left sided limit $\mathbf{U}_{N/2}=\lim_{x\to 0^-} u_1(x)$, while the right sided limit will be determined by the interface condition as explained later.
	
	Evaluating the first equation in \eqref{eqn:num:effective-system} at the integer grid points, but also in the right sided limit at the interface $x_{N/2}=0$, and the second equation at the half integer grid points produces $2N+2$ complex equations. Hence, we need to introduce another unknown to obtain a uniquely solvable system, for which we choose 
	$\bV_{N+1}:=u_2(0)$. We thus also fix both the right and left sided limit of $u_2$ to the same value, which ensures continuity of $u_2$ across the interface, i.e., the interface condition $\llbracket u_2\rrbracket=0$. 
	Following \cite{LABV1990}, the derivatives in \eqref{eqn:num:effective-system} are expressed by the standard second order central difference quotients, where near the interface
	also the right sided limit $u_1(0^+):=\lim_{x\to 0^+} u_1(x)$ appears. Since that is not part of the $2N+1$ unknowns, it needs to be expressed by the rest of them via the second interface condition $\llbracket \ri nk u_1 -u_2'\rrbracket =0$, 
	i.e., $$\lim_{x\to 0^+} \ri nk u_1(x) - u_2'(x) = \lim_{x\to 0^-} \ri nk u_1(x) -u_2'(x),$$
	where the left and right sided derivatives of $u_2$ in $0$ can be expressed by the one sided difference quotients 
	$$\lim_{x\to 0^\pm}u_2'(x)\approx \pm\frac{1}{3h}\left( -8u_2(0) +9u_2\left(\pm\frac{h}{2}\right) - u_2\left(\pm\frac{3h}{2}\right)\right).$$
	
	We may thus express the right sided limit $u_1(0^+)$ as 
	\begin{equation*} \label{eqn:num:interface}
		u_1(0^+)\approx u_1(0^-) + \frac{\ri}{nk}\frac{1}{3h} \left( u_2\left(-\frac{3h}{2}\right) - 9 u_2\left(-\frac{h}{2}\right) +16u_2(0)-9u_2\left(\frac{h}{2}\right) + u_2\left(\frac{3h}{2}\right)\right) .
	\end{equation*}
	\begin{table}[h!]
		\centering
		\begin{tabular}{c c l}
			Function & evaluated at & Discretization\\
			\hline 
			$u_1$ & $x_{j}$ (*) & $u_1(x_j)\approx \bU_j$ \\
			$u_2$ & $\tilde{x}_j$ & $u_2(\tilde{x}_j)\approx \bV_j$\\
			$u_2$ & $0$ & $u_2(0)\approx \bV_N$\\
			$u_1$ & $0+$ & $u_1(0+)\approx \bU_{N/2} + \frac{\ri}{nk}\frac{1}{3h} \left( \bV_{N/2 -2}- 9 \bV_{N/2 -1} \right.$\\ 
			& & \qquad \qquad \qquad \qquad \qquad $\left. +16\bV_N - 9\bV_{N/2} + \bV_{N/2+1}\right)$\\
			\hline
			$u_1'$ & $\tilde{x}_j$, $j\neq \frac{N}{2}$ & $u_1'(\tilde{x}_j)\approx \frac{1}{h} (\bU_{j+1}-\bU_j)$\\
			$u_1'$ & $\tilde{x}_{N/2}$ & $u_1'(\tilde{x}_{N/2})\approx \frac{1}{h}\left( \bU_{N/2+1} - u_1(0+)\right)$\\
			$u_2'$ & $x_j$, $ j \neq \frac{N}{2}$ & $u_2'(x_j)\approx \frac{1}{h}\left( \bV_j-\bV_{j-1}\right)$\\
			$u_2'$ & $0\pm$ & $\lim_{x\to 0\pm} u_2'(x)\approx \pm\frac{1}{3h}\left( -8 \bV_N +9\bV_{(N-1\pm 1)/2} - \bV_{(N-1\pm 3)/2}\right)$\\
			\hline 
			$u_2''$ & $\tilde{x}_j$, $j\neq \frac{N-1}{2}\pm \frac{1}{2}$ & $u_2''(\tilde{x}_j) \approx \frac{1}{h^2}\left( \bV_{j-1}-2\bV_j +\bV_{j+1}\right)$\\
			$u_2''$ & $\tilde{x}_{(N-1 \pm 1)/2} $ & $u_2''(\tilde{x}_{(N-1\pm 1)/2}) \approx \frac{1}{3h^2}\left( 8 \bV_N -12 \bV_{(N-1\pm 1)/2} +4 \bV_{(N-1\pm 3)/2}\right)$
		\end{tabular}
		\caption{The finite difference schemes used in the discretization of \eqref{eqn:num:effective-system}. (*): left sided limit in $x_{N/2}=0$.}
		\label{tab:discretization}
	\end{table}
	
	The second effect of the interface on our discretization is the fact that $u_2$ (while continuous) is not necessarily differentiable across the interface, so the difference quotients need to be adapted there to include only values of $u_2$ on the same side of the interface.
	In summary, we arrive at the discretization of equations \eqref{eqn:num:effective-system} shown in Table \ref{tab:discretization}. The result is an inhomogeneous system of linear equations, which can be solved by standard algorithms (e.g. Matlab's $\setminus$ operator).
	
	To experimentally deduce an order of convergence for this method, we first calculate a reference solution as a discretized partial sum $\psi^{(M)}$, see \eqref{E:part-sum}, with $M = 10$ on a very fine grid ($d=100$, $N\approx 2.5\cdot 10^6$). Next, we compute approximations on coarser grids and determine the $L^2$-norm of their differences from the reference solution. Figure \ref{fig:convInH} shows such an error plot, confirming a convergence of order $O(N^{-2})$ as one would expect from the use of second order finite difference quotients.
	\begin{figure}[h!]
		\centering
		\includegraphics[width=0.5\textwidth]{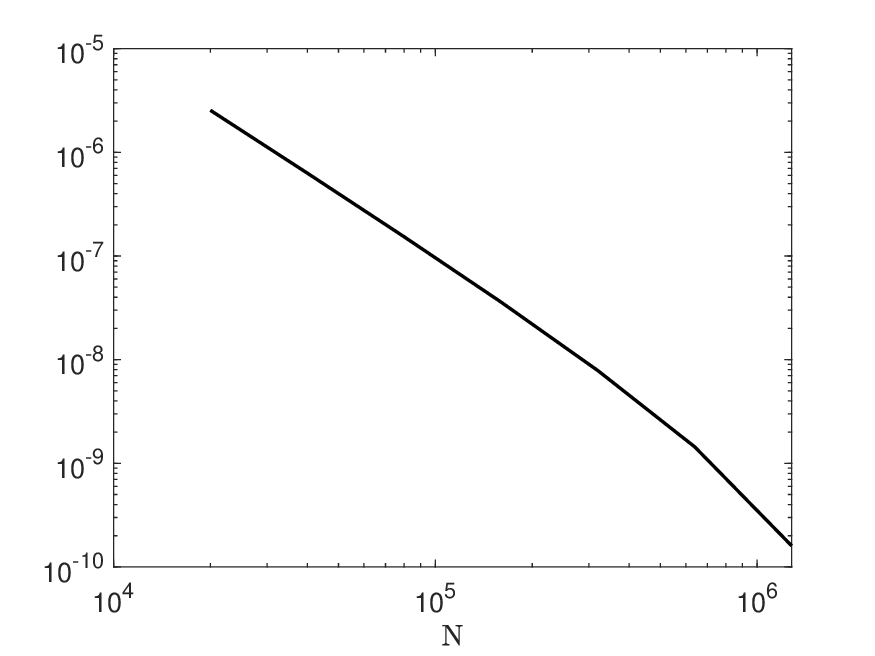}
		\caption{The $L^2$-norms of the error in the numerical solution $(u_1,u_2)$ calculated on a grid with $N$ points as a partial sum with $M=10$. The reference solution was calculated on a grid with $d=100$ and with $N\approx 2.5\cdot 10^6$ points using the same method (described in App. \ref{app:num}). The susceptibilities are given by \eqref{E:Lorentz-interf} and \eqref{E:chi23-ex} with the parameters \eqref{eqn:parameters}, \eqref{eqn:parameters-NL} and with $\varepsilon=0.5$. }
		\label{fig:convInH}
	\end{figure}
	
	\noindent
	
	\noindent
	\textbf{Use of artificial intelligence:} A version of the AI assistant Claude was used for the final proof reading of Section \ref{sec:proof-Thm-2}. No other part of the research or manuscript writing was assisted by AI.
	
	\noindent
	\textbf{Author Contributions.}	 T.D., M.H., and R.H. contributed equally to the writing of the original draft, 
	reviewing, and editing. In addition, M.H. was responsible for the numerical computations.
	
	\smallskip
	\noindent
	\textbf{Funding.} Open Access funding enabled and organized by Project DEAL.
	
	\smallskip
	\noindent
	\textbf{Competing interests.}  The authors declare no competing interests.

	\bibliographystyle{abbrv}
	\bibliography{biblio-polychrom}
\end{document}